\documentclass[a4paper,12pt]{amsart}

\usepackage[latin1]{inputenc}
\usepackage{amsfonts,amsmath,amssymb,amsthm}
\usepackage[english]{babel}
\usepackage[dvipsnames]{xcolor}
\usepackage[margin=3cm, marginpar=2.5cm]{geometry}
\usepackage{array}
\usepackage{booktabs}
\usepackage{hyperref}
\usepackage{longtable}
\usepackage{graphicx}
\usepackage{subfig}

\newtheorem{thm}{Theorem}[section]
\newtheorem{lemma}[thm]{Lemma}
\newtheorem{prop}[thm]{Proposition}
\newtheorem{cor}[thm]{Corollary}
\newtheorem{defn}[thm]{Definition}

\theoremstyle{remark}
\newtheorem{ex}[thm]{Example}
\newtheorem{rmk}[thm]{Remark}

\numberwithin{figure}{section}
\numberwithin{equation}{section}

\newcommand{\C}{\mathbb{C}}
\newcommand{\Z}{\mathbb{Z}}
\newcommand{\CP}{\mathbb{CP}^2}
\newcommand{\CPone}{\mathbb{CP}^1}
\newcommand{\CPbar}{\overline{\mathbb{CP}}\,\!^2}

\newcommand{\G}{\mathcal{G}}
\newcommand{\Qc}{\mathcal{Q}}
\newcommand{\Vc}{\mathcal{V}}
\renewcommand{\epsilon}{\varepsilon}

\DeclareMathOperator{\Arf}{Arf}

\title{Symplectic isotopy of rational cuspidal sextics and septics}
\author{Marco Golla}
\email{marco.golla@univ-nantes.fr}
\address{CNRS, Laboratoire Jean Leray, Nantes University, Nantes, France}

\author{Fabien K\"utle}
\email{fabien.kutle@univ-nantes.fr}
\address{Laboratoire Jean Leray, Nantes University, Nantes, France}

\date{}

\begin{document}

\begin{abstract}
We classify rational cuspidal curves of degrees 6 and 7 in the complex projective plane, up to symplectic isotopy. The proof uses topological tools, pseudoholomorphic techniques, and birational transformations.
\end{abstract}

\maketitle


\section{Introduction}

One of the central problems in symplectic topology in dimension 4 is the symplectic isotopy problem: it asks whether every non-singular symplectic surface in the complex projective plane $\CP$, equipped with the Fubini--Study form $\omega_{\rm FS}$, is symplectically isotopic to a non-singular complex curve. Equivalently, it asks whether there is a unique symplectic isotopy class of non-singular symplectic surfaces in each degree.
The problem is known to have an affirmative answer in degrees up to 17~\cite{Gromov, Sikorav, Shevchishin, SiebertTian}, but is open in higher degrees. (See also~\cite{Starkston} for another proposed strategy.)

\emph{Singular} symplectic surfaces and their isotopies are somewhat less studied.
Following~\cite{GS}, we will restrict our attention to curves whose singularities are modelled over complex curve singularities; motivated by questions in algebraic geometry, we will also impose further restrictions on the curves, namely that they are \emph{rational} and \emph{cuspidal}.
By work of McDuff~\cite{McDuff-Jhol} and of Micallef and White~\cite{MicallefWhite}, all these restrictions together can be summarised by saying that we are looking at curves that are the injective image of $\CPone$ via a $J$--holomorphic map (for \emph{some} $\omega_{\rm FS}$--compatible almost-complex structure $J$ on $\CP$).

In~\cite{GS}, the first author and Starkston studied the singular symplectic isotopy problem for rational cuspidal curves, and gave some classification results for curves of low degree (up to $5$) or low ``complexity''. We use the terminology from that paper and refer to a symplectic surface as a \emph{symplectic curve}.
We say that two symplectic curves are \emph{equisingular} if they have the same singularities \emph{up to topological equivalence} (that is, the links of their singularities are isotopic as knots in $S^3$) and that they are \emph{symplectically isotopic} if they are isotopic through equisingular symplectic curves. The main theorem of this paper is an extension of~\cite[Theorem~1.3]{GS} to degrees $6$ (sextics) and $7$ (septics).

\begin{thm}\label{t:main}
Every symplectic rational cuspidal curve of degree $6$ or $7$ is symplectically isotopic to a complex curve; moreover, any two such curves are symplectically isotopic if and only if they are equisingular and have the same degree.
\end{thm}

The corresponding problem in algebraic geometry has a rich history: the classification of rational cuspidal curves, even up to equisingularity, is not fully understood. In degrees up to $6$, one can work out the classification by hand ~\cite{Namba, Fenske} (see~\cite{Moe} for a more focused exposition for curves of degree up to $5$). We could not find an analogous statement of the classification of rational cuspidal septics in the literature.
In another direction, curves whose complement has the log Kodaira dimension $\overline\kappa = -\infty$ or $\overline\kappa = 1$ are classified in~\cite{Kashiwara, Miyanishi, MiyanishiSugie} and~\cite{Tono-PhD, Tono, Miyanishi} respectively, and there are no curves with $\overline\kappa = 0$~\cite{Tsunoda}.
(We are unaware of similar results from a symplectic perspective.)
When the curve is supposed to have only one singular point whose link is a torus knot (in algebro-geometric terms, the singularity has one Puiseux pair or one Newton pair), the classification was worked out by Fern\'andez de Bobadilla, Luengo, Melle Hern\'andez, and N\'emethi~\cite{FLMN} (this is what we meant by ``low complexity'' earlier; see~\cite[Theorem~1.2]{GS} for the corresponding symplectic result).
There is an equisingular classification result (up to projective equivalence) for curves of log general type (i.e. $\overline\kappa = 2$) in~\cite{PalkaPelka1, PalkaPelka2}, assuming the Negativity Conjecture~\cite{Palka}.

In order to even state the singular symplectic isotopy problem, the first issue is to delimit the field, and decide the allowed types of singularities. For instance, symplectic surfaces can have negative double points~\cite{McCarthyWolfson}, non-isolated singularities, or isolated singularities that are cones over arbitrary transverse links~\cite[Section~1.1]{hats}.
The choice to restrict to complex-type singularities is motivated as follows: on the one hand, we want to be able to compare symplectic objects with algebro-geometric objects; on the other hand, this is the only class of singularities that have both a symplectic smoothing and a symplectic resolution, and both aspects are crucial in the proof of Theorem~\ref{t:main}.

Using braid monodromy techniques, Moishezon~\cite{Moishezon} gave examples of symplectic curves with nodes (transverse double points, i.e.~locally modelled on $\{x^2 = y^2\}\subset \C^2$) and simple cusps  (i.e. locally modelled on $\{x^2 = y^3\} \subset \C^2$) that are not isotopic to any complex curve; see also~\cite{AurouxDonaldsonKatzarkov} for a different point of view. In the direction of giving conditions to ensure equisingular isotopy, early results about nodal symplectic surfaces are due to Shevchishin~\cite{Shevchishin-nodes} and for curves with nodes and simple cusps to Francisco~\cite{Francisco}. These curves are particularly relevant for realising symplectic $4$--manifolds as branched covers of $\CP$~\cite{Auroux}.

We also note that, even in the rational case, not all symplectic curves are isotopic to complex curves: one can easily construct examples for line arrangements (see~\cite{RubermanStarkston}); in the irreducible case, Orevkov constructed a symplectic rational curve of degree $8$, whose singularities are \emph{not} cuspidal, that is not isotopic to any complex curve (see~\cite[Section~8]{GS}).

As a corollary of the proof of Theorem~\ref{t:main}, we also prove a symplectic version of the Coolidge--Nagata conjecture for curves of degree at most $7$, showing that every symplectic rational cuspidal curve $C$ of degree at most $7$ is \emph{Cremona equivalent} to a line. This means that there are two sequences of blow-ups of $\CP$, giving two symplectic $4$--manifolds $(X,\omega)$ and $(X',\omega')$ such that there exists a diffeomorphism $\psi: X\to X'$ that sends the proper transform $\widetilde{C}$ of $C$ to the proper transform of a line, and such that $\psi^*\omega'$  deforms to $\omega$. We note that the Coolidge--Nagata conjecture~\cite{Coolidge, Nagata} in the algebro-geometric context was recently proved by Koras and Palka~\cite{KorasPalka}.

\begin{prop}\label{p:cremona}
Every symplectic rational cuspidal curve of degree at most $7$ is Cremona equivalent to a line.
\end{prop}

In fact, in all cases we examine, the proper transform of $C$ in its minimal resolution (see Section~\ref{s:recap} below for the definition) has positive self-intersection, so that the result follows directly from McDuff's theorem (Theorem~\ref{t:McDuff} below), which (up to further blow-ups) identifies the proper transform of $C$ with a line in a blow-up of $\CP$ that has \emph{not} been blown-up.

In a different direction, as an immediate corollary to Theorem~\ref{t:main} and \cite[Theorem~1.3]{GS}, any two complex curves of degree up to $7$ are symplectically isotopic if and only if they have the same singularity data; in particular, their complements are diffeomorphic. This might be known to experts in complex curves, but we were unable to find the statement in the literature.

\begin{cor}
If $C$ and $C'$ are two equisingular complex rational cuspidal curves of degree at most $7$, then $\CP\setminus C$ and $\CP\setminus C'$ are diffeomorphic.\hfill$\qed$
\end{cor}

In essence, the corollary says that one cannot tell apart two complex rational cuspidal curves by looking at the homotopy of their complements; for instance, their Alexander polynomials~\cite{Libgober} agree.
Another way of phrasing the corollary is that there are no (symplectic or complex) \emph{Zariski pair} of rational cuspidal curves in low degree.

The proof of Theorem~\ref{t:main} is divided into two parts: one existence and uniqueness part, which is Theorem~\ref{t:existence} below, and an obstruction part, Theorem~\ref{t:obstruction}.
It is, in essence, a case-by-case analysis. The (singular) adjunction formula~\eqref{e:adjunction} allows us to reduce the problem to a finite one; in topological terms, the adjunction formula says that the sum of the Seifert genera of all the links of the singularities is determined by the degree-genus formula. The number of possible configurations of singularities on a sextic satisfying this constraint is $106$, while for septics the number jumps to $718$. Out of these, only $11$ configurations of singularities of sextics, and $11$ configurations of septics, are symplectically realised.
For an outlook on what awaits us beyond degree $7$, the adjunction formula gives $5612$ possible configurations of singularities on a curve of degree $8$.

\begin{rmk}
In~\cite{GS}, the results are actually stronger than what stated in~\cite[Theorem~1.3]{GS}: not only are isotopy classes of rational cuspidal curves of degree $d \le 5$ classified, but all \emph{relatively minimal pairs} $(X,C)$ of a symplectic $4$--manifold $X$ with a symplectic curve $C \subset X$ with an allowed configuration of singularities and Euler number $d^2$. This in turn also classifies the strong symplectic fillings of the cuspidal contact structure associated to the curve. Here we focus on curves in $\CP$ instead; this allows us to use more tools (both for obstructions and for constructions). However, in Section~\ref{s:fillings} we look at relatively minimal pairs for the curves with configurations of singularities covered in Theorem~\ref{t:existence}.
\end{rmk}

The existence and uniqueness part builds on McDuff's fundamental theorem, which we state as Theorem~\ref{t:McDuff} below; her result asserts that every pair $(X,\Sigma)$ comprising a symplectic $4$--manifold $X$ and a non-singular symplectic sphere $\Sigma\subset X$ of self-intersection $+1$ is symplectomorphic to a blow-up of $(\CP, \ell)$, where $\ell$ is a line. This, together with $J$--holomorphic techniques (e.g. positivity of intersections), reduces each case to the study of the uniqueness and isotopy for a symplectic configuration of (usually non-singular) symplectic curves in $\CP$.
Note that, even though all such curves were previously known to exist, for each of them we do give a rather explicit \emph{symplectic} construction, which is essentially algebro-geometric in nature and also gives algebraic representatives.
Similar constructions were already considered in the closely related context of \emph{$\mathbb{Q}$--homology planes}, i.e.~open complex manifolds whose rational homology is the same as $\C^2$; see, for example,~\cite{tomDieckPetrie}.

The obstruction part uses various kinds of tools. Two of them are particularly easy to code up and effective at eliminating many cases: the semigroup obstruction of Borodzik and Livingston~\cite{BorodzikLivingston}, which we refer to as the \emph{Heegaard Floer obstruction}, and the Riemann--Hurwitz formula, stated below as Theorem~\ref{t:semigroupobstruction} and Proposition~\ref{p:RHobstruction}, respectively. In degree $6$, the Heegaard Floer obstruction rules out $45$ cases and the Riemann--Hurwitz obstruction rules out $31$ ($9$ cases are ruled out by both); in degree $7$, these numbers are $522$ and $267$ (with $156$ overlaps). This still leaves out $39$ cases in degree $6$ and $85$ in degree $7$. If one wanted to extend the result of Theorem~\ref{t:main} to the case of curves of degree $8$, after testing the Heegaard Floer and Riemann--Hurwitz obstructions there are still $318$ cases left to examine (out of the $5612$ that satisfy the adjunction formula).

Many of the remaining topological arguments make use of branched covers and the signature of the intersection form of $4$--manifolds. This strategy has a long history in the smooth context~\cite{Massey, Rokhlin, HsiangSzczarba, Ruberman}, originating in the Atiyah--Singer index theorem~\cite{AtiyahSinger}; singular analogues in the algebraic context can be found, for instance, in~\cite{Persson}, and in the symplectic context in~\cite{Auroux, RubermanStarkston}; Ruberman and Starkston's beautiful paper~\cite{RubermanStarkston}, in particular, was a strong source of inspiration. Our analysis of branched covers culminates with Theorem~\ref{p:signatureobstruction} below, which gives a computable obstruction; to the best of our knowledge, the statement has not been previously appeared in this generality in the literature.

Whenever possible, we try to emphasise which of our results obstruct the existence of certain \emph{PL spheres} in $\CP$; for instance, the Heegaard Floer obstructions of~\cite{BorodzikLivingston, BorodzikHom} do. Here, a PL sphere is the image of a piecewise-linear embedding $S^2\to \CP$. As observed in~\cite{BorodzikLivingston}, a (symplectic or complex) rational cuspidal curve of degree $d$ gives an embedding of a $4$--dimensional compact manifold $X_{d^2}(K)$ obtained by attaching a 2--handle to the 4--ball along a knot $K$ with framing $d^2$; $X_{d^2}(K)$ is called the trace of $d^2$--surgery along $K$, and it is the regular neighbourhood of a PL sphere with Euler number $d^2$ (encoding the framing data) and a point that is a cone over $(S^3,K)$. More generally, when $K$ is expressed as a connected sum $K_1\#\dots\# K_\nu$, we can think of the sphere as having $\nu$ singular points that are cones over $(S^3,K_i)$. Whenever we have such an embedding for which the knot $K$ is a connected sums of algebraic links that satisfy the adjunction formula, we say that there is an \emph{adjunctive} PL sphere with those singularities.

We think that understanding adjunctive PL spheres in $\CP$ is a natural extension of the problem we are studying here; for instance, in the context of adjunctive PL spheres whose only singularity is of type $T(p,q)$ (i.e.~locally modelled on $\{x^p + y^q = 0\} \subset \C^2$), there exists an adjunctive PL sphere of degree $d$ with a singularity of type $T(p,q)$ if and only if there is an algebraic one (see~\cite[Theorem~2.3]{Liu} or~\cite[Remark~6.18]{BCG}). Non-adjunctive PL spheres with one singularity of type $T(p,q)$ are, not unexpectedly, harder to study; it is easy to show that, for instance, there is a degree-$d$ PL sphere with a singularity of type $T(d,d+1)$ (instead of $T(d-1,d)$, which is adjunctive). Nevertheless, the problem of the existence of such PL sphere is studied (and almost solved) in~\cite{AGLL}.

\subsection*{Organisation} In Section~\ref{s:recap} we set the notation and recall some background results. Sections~\ref{s:existence} and~\ref{s:obstruction} we prove the construction and obstruction parts of Theorem~\ref{t:main}. Finally, Section~\ref{s:fillings} explores some contact-theoretic aspects of the problem.

\subsection*{Acknowledgements} We benefitted from talking to J\'ozsi Bodn\'ar, Maciej Borodzik, Anthony Conway. We warmly thank Tomasz Pe{\l}ka and Laura Starkston for several interesting conversations and for their comments on an earlier draft.

\section{Background, notation and conventions}\label{s:recap}

We gather here some useful facts about complex singularities, pseudoholomorphic techniques and birational techniques. We follow~\cite{GS}, to which we refer the reader for further details. A lot of the results we need are based on work of McDuff~\cite{McDuff, McDuff-Jhol}; an excellent comprehensive reference is~\cite{Wendl}.

Unless specified otherwise, we consider homology and cohomology with integer coefficients.

We first briefly recall some definitions. The \emph{degree} of a (possibly singular) symplectic curve $C$ in $\CP$ is the positive integer $d$ such that $[C] = dh \in H_2(\CP)$, where $h$ is the homology class of a line, oriented so that $\omega_{\rm FS}$ integrates positively on it. An \emph{equisingular symplectic isotopy} is a one-parameter family $\{C_t\}_{t\in [0,1]}$ of singular symplectic curves $C_t \subset (X, \omega)$ such that for each $t,t' \in [0, 1]$ the curves $C_t$ and $C_{t'}$ have topologically equivalent singularities; since we only consider equisingular isotopies, we systematically drop the adjective `equisingular', and only talk about \emph{symplectic isotopies}.
A singular symplectic curve $C$ is said to be \emph{minimally embedded} in a symplectic $4$--manifold $(X, \omega)$ if $X \setminus C$ contains no exceptional symplectic $(-1)$--spheres

\subsection{Resolution of singularities}

As mentioned in the introduction, we restrict our attention to curves whose singularities are modelled over complex curves. The reason behind this choice is that these are the only singularities that can occur for pseudoholomorphic curves (see~\cite{McDuff-Jhol} and~\cite{MicallefWhite}). We only consider cuspidal curves, that is to say curves whose singularities have a single branch; this is the same as requiring that the link of each singularity is a knot. The topological type of a cuspidal singularity can be characterised in several ways: one can consider its multiplicity sequence, its semigroup, its Puiseux sequence, or its link. Each of these data determine the other ones. We refer to~\cite{Wall} for more details. The dictionary presented in Table~\ref{t:dictionary} gathers some of this information for the cuspidal singularities that appear in this paper.

\begin{table}[h!]
\centering
\begin{tabular}{llll}
Multiplicity sequence & Cabling parameters & Puiseux pairs & ADE\\
$[6]$ & $(6,7)$ & $(6,7)$ &\\
$[5,2,2]$ & $(5,7)$ & $(5,7)$ &\\
$[5]$ & $(5,6)$ & $(5,6)$ &\\
$[4,3]$ & $(4,7)$ & $(4,7)$ &\\
$[4,2^{[k+1]}]$ & $(2,3;2,2k+11)$ & $(2,3)$, $(2,2k-1)$ &\\
$[4]$ & $(4,5)$ & $(4,5)$ &\\
$[3^{[k]},2]$ & $(3,3k+2)$ & $(3,3k+2)$ & $k=1$: $E_8$\\
$[3^{[k]}]$ & $(3,3k+1)$ & $(3,3k+1)$ & $k=1$: $E_6$ \\
$[2^{[k]}]$ & $(2,2k+1)$ & $(2,2k+1)$ & $A_{2k}$
\end{tabular}
\caption{A dictionary between multiplicity sequences of singularities and cabling parameters of their links. The parameter $k$ is always positive. We indicated which of these singularities are simple (or ADE, or du Val).}\label{t:dictionary}
\end{table}

For every singular curve $C$ in a symplectic $4$--manifold $X$, there exists a composition of blow-ups $\pi : \tilde{X} \rightarrow X$ such that the proper transform $\tilde{C}$ of $C$ is smooth. We call any such $\tilde{C}$ a resolution of $C$. There are two natural stopping points when resolving a singularity: the \emph{minimal resolution} is the smallest resolution such that the proper transform $\tilde{C}$ of $C$ is smooth; the \emph{normal crossing resolution} is the smallest resolution such that the total transform $\tilde{C}$ of $C$ is a normal crossing divisor, i.e.~all singularities are transverse double points. The \emph{multiplicity sequence} of a singularity $(C,p)$ is a finite sequence of integers that records the multiplicities of the singularities that appear after each blow-up of the minimal resolution of $(C,p)$ (the first element being the multiplicity of $(C,p)$). Here our convention is that $[m_1,\dots,m_k]$ denotes the multiplicity sequence of a singularity (it ends with $m_k > 1$), $[[m_1,\dots,m_k]]$ denotes the \emph{multiplicity multisequence} of a curve (that is the union of all multiplicity sequences of singularities of a curve), $[[m^1_j], \dots, [m^c_j]]$ denotes the collection of multiplicity sequences of a curve, where $a^{[b]}$ denotes the string $a,\dots,a$ of length $b$. We say that a singularity is of type $[m_1,\dots,m_k]$ and a curve is of type\footnote{A word of caution: there is a slight ambiguity in the notation when it comes to unicuspidal curves. For instance, $[[3,3,3,2]]$ denotes both the type of a sextic with multiplicity multisequence $[[3,3,3,2]]$ and the type of a sextic with a single singularity of type $[3,3,3,2]$. The context should be sufficient to disambiguate.} $[[m^1_j], \dots, [m^c_j]]$ or $[[m_1,\dots,m_k]]$.  Sometimes singularities are indicated by their topological cabling parameters. Throughout this paper, when we say that we blow up multiple times at a cuspidal singular point $(C,p)$, it means that at after each blow-up, the next blow-up takes place at the intersection between the proper transform of $C$ and the other curves of the total transform (which is a unique point because the singularity is cuspidal).

\subsection{Embeddings of plumbings into $\CP \# N \CPbar$}

Suppose $P$ is a configuration of symplectic spheres, such that one of the spheres has self-intersection $+1$. In our context, this will typically be a neighborhood of the normal crossing resolution of a rational cuspidal curve (or possibly a further blow-up). A theorem of McDuff strongly restricts the closed symplectic manifolds in which $P$ can symplectically embed.

\begin{thm}[\cite{McDuff}]\label{t:McDuff}
If $(X, \omega)$ is a closed symplectic $4$--manifold and $C_0 \subset X$ is a smooth symplectic sphere of self-intersection number $+1$, then there is a symplectomorphism of $(X, \omega)$ to a symplectic blow-up of $(\CP, \lambda \omega_{\rm FS})$ for some positive $\lambda$, such that $C_0$ is identified with a line.
\end{thm}

We will now discuss how to classify all symplectic embeddings of $P$ into $\CP \# N \CPbar$. A symplectic embedding of a rational cuspidal curve is equivalent (by a sequence of blow-ups supported in a neighborhood of the rational cuspidal curve) to a symplectic embedding of the plumbing associated to its normal crossing resolution. In order to classify embeddings of $P$ into $\CP \# N \CPbar$, we first determine the possibilities for the map on second homology induced by the embedding. Since the core spheres of the plumbing form a basis for $H_2(P)$, we just need to classify the possible classes in $H_2(\CP \# N \CPbar$) that these symplectic spheres can represent.

For this purpose, this paper uses the following lemmas from \cite[Section~3.3]{GS}. They are all proved using pseudoholomorphic techniques. The strength of using pseudoholomorphic curves is that we keep control over geometric intersections, whereas two symplectic surfaces may intersect with a cancelling pair of positive and negative intersections (in which case they could not be both realised as pseudoholomorphic curves for the same almost complex structure). 

Here our convention is that $h$ is the class of a line, and $e_i$ are the classes of the exceptional divisors such that $h,e_1,\dots, e_N$ forms the standard basis for $H_2(\CP \# N \CPbar)$.

\begin{lemma}\label{l:adjclass}\label{l:hom}
Suppose $\Sigma$ is a smooth symplectic sphere in $\CP\#N\CPbar$ intersecting a line non-negatively. Then writing $[\Sigma]=a_0h+a_1e_1+\dots+a_Ne_N$, we have:
\begin{enumerate}
	\item If $a_0=0$, there is one $i_0$ such that $a_{i_0}=1$ and all other $a_i\in \{0,-1\}$.
	\item If $a_0=1$ or $a_0=2$, $a_i\in \{0,-1\}$ for all $1\leq i\leq N$.
	\item If $a_0=3$, then there exists a unique $i_0$ such that $a_{i_0}=-2$, and $a_i\in \{0,-1\}$ for all other $i$.
\end{enumerate}
The self-intersection number of $\Sigma$ can be used to compute how many $a_i$ have coefficient $0$ versus $-1$.
\end{lemma}

In the following lemmas, $C_i$ and $C_j$ are smooth symplectic spheres in a positive plumbing in $\CP\#N\CPbar$ such that $[C_i] \cdot h = [C_j] \cdot h = 0$.

\begin{lemma}\label{l:consecutive}
If $[C_i]\cdot[C_j]=1$ (and $[C_i]\cdot h=[C_j]\cdot h=0$), there is exactly one exceptional class $e_i$ which appears with non-zero coefficient in both $[C_i]$ and $[C_j]$. The coefficient of $e_i$ is $+1$ in one of $[C_i],[C_j]$ and $-1$ in the other.
\end{lemma}

\begin{lemma}\label{l:pos}
If $e_m$ appears with coefficient $+1$ in $[C_i]$ then it does not appear with coefficient $+1$ in the homology class of any other sphere in the plumbing.
\end{lemma}
	
\begin{lemma}\label{l:share2}
If $[C_i]\cdot [C_j]=0$, then either there is no exceptional class which appears with non-zero coefficients in both, or there are exactly two exceptional classes $e_m$ and $e_n$ appearing with non-zero coefficients in both. One of these classes, $e_m$, has coefficient $-1$ in both $[C_i]$ and $[C_j]$ and the other, $e_n$, appears with coefficient $+1$ in one of $[C_i]$ or $[C_j]$ and coefficient $-1$ in the other.
\end{lemma}

\begin{lemma}\label{l:2chain}
Suppose $\Sigma_1,\dots, \Sigma_k$ is a chain of symplectic spheres of self-intersection $-2$ disjoint from a line $\CPone$ in $\CP\# N\CPbar$. Then the homology classes are given by one of the following two options, up to re-indexing the exceptional classes:
\begin{enumerate}
	\item \label{i:to} $[\Sigma_i]=e_i-e_{i+1}$ for $i=1,\dots, k$.
	\item \label{i:from} $[\Sigma_i]=e_{i+1}-e_i$ for $i=1,\dots, k$.
\end{enumerate}
The homology class of any surface disjoint from the chain has the same coefficient for $e_1,\dots, e_{k+1}$.

Moreover, if the chain is attached to another symplectic sphere $\Sigma_0$ which does intersect the line, option~\eqref{i:from} can only occur if $e_2,\dots, e_{k+1}$ all appear with coefficient $-1$ in $[\Sigma_0]$. In particular if $[\Sigma_0]\cdot h=1$, option~\eqref{i:from} can only occur if $[\Sigma_0]^2\leq 1-k$.
\end{lemma}

We will also often use the following lemma to control the effects of some blow-downs on a given configuration of curves.

\begin{lemma}[\cite{McDuff}] \label{l:blowdown}
Suppose $\mathcal{C}$ is a configuration of positively intersecting symplectic surfaces in $\CP \# N\CPbar$. Let $e_{i_1},\dots,e_{i_\ell}$ be exceptional classes which have non-negative algebraic intersections with each of the symplectic surfaces in the configuration $\mathcal{C}$. Then there exist disjoint exceptional spheres $E_{i_1},\dots, E_{i_\ell}$ representing the classes $e_{i_1},\dots, e_{i_\ell}$ respectively such that any geometric intersections of $E$ with $\mathcal{C}$ are positive.
\end{lemma}

Once all the possibilities for the map on second homology induced by the embedding are determined, one can proceed to construct explicitly the embbedings by using birational transformations on a configuration of curves for which existence and uniqueness of the equisingular isotopy class in $\CP$ is known.

\subsection{Birational transformations}

In complex dimension $2$, a birational transformation is a sequence of blow-ups and blow-downs. Since blow-ups and blow-downs can be done symplectically (see~\cite{McDuff}), these transformations from algebraic geometry can be imported into the symplectic context. We recall the two ways introduced in~\cite{GS} of relating singular symplectic surfaces in $\CP$ using birational transformations.

The first notion is weaker, but for two surfaces related in this way, the existence of one type of singular surface will imply the existence of another type of singular surface.

\begin{defn}
A symplectic surface $\Sigma_2 \subset (M, \omega)$ is \emph{birationally derived} from another symplectic surface $\Sigma_1 \subset (M, \omega)$ if there is a sequence of blow-ups of the pair $(M,\Sigma_1)$ to the total transform $(M \# N \CPbar,\tilde{\Sigma}_1)$, followed by a sequence of blow-downs of exceptional spheres $\pi : M \# N \CPbar \rightarrow M$, such that $\Sigma_2 = \pi (\tilde{\Sigma}_1)$.

A symplectic surface $\Sigma_1 \subset (M, \omega)$ is \emph{birationally equivalent} to another symplectic surface $\Sigma_2 \subset (M, \omega)$ if there is a sequence of blow-ups of the pair $(M, \Sigma_1)$ to the total transform $(M \# N \CPbar,\tilde{\Sigma}_1)$, followed by a sequence of blow-downs of exceptional spheres $\pi : M \# N \CPbar \rightarrow M$ such that the exceptional locus of $\pi$ is contained in $\tilde{\Sigma}_1$ and $\Sigma_2 = \pi (\tilde{\Sigma}_1)$.
\end{defn}

Note that the first relation is not symmetric, that the second relation is an equivalence relation and that the number of components of a configuration is preserved by a birational equivalence. This last definition will be useful to relate the symplectic isotopy classifications of two symplectic surfaces.

\subsection{Local intersection between a cuspidal singularity and its tangent}

To obstruct some types of rational cuspidal septics by birational transformations, we will sometimes consider the $J$--holomorphic line $t_p$ tangent to a $J$--holomorphic curve $C$ at one of its cuspidal points, $p$. In general, if $C$ is a $J$--holomorphic symplectic curve of degree $d$ and $p$ is a cuspidal point of $C$, the local intersection number between $C$ and $t_p$ at $p$ can vary depending on the almost complex structure $J$: we know that it belongs to the semigroup of the cusp at $p$, that it is greater than or equal to the third element of the semigroup, and that it is at most $d$. We show that we can always assume that this local intersection number is of the least possible order, that is to say the third element of the semigroup of $(C,p)$, which we will denote by $\Gamma_{(C,p)}(2)$ (we label the first element of the semigroup, which is $0$, by $\Gamma_{(C,p)}(0)$).

\begin{lemma} \label{l:generictangent}
Let $J$ an almost complex structure on $\CP$ tamed by $\omega_{\rm FS}$ and $C$ a simple $J$--holomorphic curve with a cuspidal singularity at $p$. Then there exists an almost complex structure $J'$, $\mathcal{C}^0$--close to $J$ such that the local intersection between the $J'$--holomorphic line tangent to $C'$ at $p$ and $C'$ is equal to $\Gamma_{(C,p)}(2)$.
\end{lemma}

For the proof, we will need the following theorem.

\begin{thm}[\cite{McDuff-Jhol}] \label{t:McDuff92} Let $C$ be a $J$--holomorphic curve in an almost complex $4$--manifold $(X, J)$. There is an almost complex structure $J'$, arbitrarily close to $J$ in the $\mathcal{C}^0$--topology, such that $C$ is $J'$--holomorphic and $J'$ is integrable near each of the singular points of $C$.
\end{thm}

\begin{proof}[Proof of Lemma~\ref{l:generictangent}]
Consider the $J$--holomorphic line $t_p$ tangent to $C$ at $p$, and blow up twice at $p$. Denote by $e_1$ and $e_2$ the components of the total transform of $C$ (apart from the proper transform of $C$) numbered by their order of appearance. When blowing up twice at $p$, we lower the local intersection number between the proper transform of $C$ and the proper transform of $t_p$ by $\Gamma_{(C,p)}(2)$, which is the lowest possible local intersection number between a cusp and its tangent. Therefore, the local intersection between $C$ and $t_p$ at $p$ is equal to $\Gamma_{(C,p)}(2)$ if and only if the proper transform of $t_p$ is disjoint from the proper transform of $C$ near $p$.

If the proper transform of $t_p$ is not disjoint from the proper transform of $C$ then, using Theorem~\ref{t:McDuff92}, we can assume that $J$ is integrable near $p$ (up to perturbing $J$). We next perturb locally symplectically the proper transform $\tilde t_p$ of $t_p$ to a new line $l$ that coincides with $\tilde t_p$ outside a neighbourhood of $p$ and such that:
\begin{itemize}
\item[--] the perturbation is complex near $p$ (it can be done in a complex chart),
\item[--] $l$ does not pass through $p$ and intersects $e_2$ transversely exactly once,
\item[--] the new intersections created between $l$ and the proper transform of $C$ are all positive and transverse.
\end{itemize}

Now, using Theorem~\ref{t:McDuff92} once again, we can find a new $J'$ compatible with $\omega_{\rm FS}$ such that $C \cup t_p' \cup e_1 \cup e_2$ is $J'$--holomorphic. Finally, we contract $e_2$ and $e_1$. The blow-down of $l$ is now the $J'$--holomorphic line tangent to $C$ at $p$, and since $l$ is disjoint from $C$ near $p$, the local intersection between the blow-down of $l$ and $C$ is of the least possible order at $C$.
\end{proof}

\begin{rmk}
Alternatively, with the same argument (but without using any blow-up or blow-down), one could locally perturb the Puiseux parametrisation of the curve $C$ near the cusp $p$ given by the integrable almost complex structure $J$ near $p$ in order to make appear the term of order $\Gamma_{(C,p)}(2)$ in the Puiseux expansion (note that it does not change the type of the singularity). That would lower the local intersection with the tangent $J$--holomorphic line at $p$ to $\Gamma_{(C,p)}(2)$. One could conclude in the same way by applying Theorem~\ref{t:McDuff92}.
\end{rmk}

\section{Existence and uniqueness} \label{s:existence}

The aim of this section is to prove Theorem~\ref{t:existence}; namely, we want to prove that every rational cuspidal curve of degree $6$ or $7$ has a unique equisingular isotopy class, and that this class contains a complex representative. We will actually provide classification results of symplectic embeddings of these cuspidal curves (with prescribed normal Euler number) into any closed symplectic manifold, equivalently classifying the strong symplectic fillings of the associated cuspidal contact structures (see Section~\ref{s:fillings}).

\begin{thm}\label{t:existence}
Rational cuspidal curves in $\CP$ with the following singularities exist and are unique up to symplectic isotopy.
\begin{center}
\begin{tabular}{l|l||l|l}
\multicolumn{2}{c||}{Degree $6$} & \multicolumn{2}{c}{Degree $7$}\\[2pt]
Cusps (MS) & Cusps (link) &	Cusps (MS) & Cusps (link)\\[2pt]
\hline
& & &\\[-10pt]
$[5]$ 		& 			$(5,6)$ &			$[6]$ & 				$(6,7)$\\
$[4,2,2,2,2]$ &			$(2,3;2,17)$ &		$[5,2,2],[2,2,2]$ &	$(5,7), (2,7)$\\
$[4,2,2,2], [2]$ &			$(2,3;2,15), (2,3)$ &	$[5],[2,2,2,2,2]$ &	$(5,6), (2,11)$\\
$[4,2,2], [2,2]$ &			$(2,3;2,13), (2,5)$ &	$[5],[2,2,2,2],[2]$ &	$(5,6), (2,9), (2,3)$\\
$[4], [2,2,2,2]$ & 			$(4,5), (2,9)$ &		$[5],[2,2,2],[2,2]$&	$(5,6), (2,7), (2,5)$\\
$[4], [2,2,2],[2]$ &		$(4,5), (2,7), (2,3)$ &	$[4,3],[3,3]$ &		$(4,7), (3,7)$\\
$[4], [2,2],[2,2]$ & 		$(4,5), (2,5), (2,5)$ &	$[4],[3,3,3]$ &		$(4,5), (3,10)$\\
$[3,3,3,2]$ & 			$(3,11)$ & 			$[4,2,2,2],[3,3]$ & 	$(2,3;2,15), (3,7)$\\
$[3,3,3], [2]$ & 			$(3,10), (2,3)$ &		$[4,2,2],[3,3,2]$ &	$(2,3;2,13), (3,8)$\\
$[3,3,2], [3]$ & 			$(3,8), (3,4)$ &		$[4,2,2],[3,3],[2]$ &	$(2,3;2,13), (3,7), (2,3)$\\
$[3,3], [3,2]$ & 			$(3,7), (3,5)$ & 		$[3,3,3,3],[2,2,2]$ &	$(3,13), (2,7)$\\
\end{tabular}
\end{center}
\end{thm}

The cases of unicuspidal rational symplectic curves whose only cusp is the cone on a torus knot are already treated in~\cite[Section~6.4]{GS}. Therefore Theorem~\ref{t:existence} is already proved for the rational unicuspidal sextics of type $[[5]]$ and $[[3,3,3,2]]$, and the rational unicuspidal septic of type $[[6]]$. For the remaining cases, there are two key results from~\cite{GS} that we will need, together with McDuff's theorem.

\begin{thm}[\cite{GS}]\label{t:addtheline}
Suppose $\mathcal{C}_1$ is a configuration of curves in $\CP$ obtained from $\mathcal{C}_0$ by adding a single symplectic
line $L$ intersecting $\mathcal{C}_0$ positively such that either
\begin{enumerate}
\item $L$ has a simple tangency to one of the curves of $\mathcal{C}_0$ (either at a special point or a generic point
on that curve) and intersects no other singular points of $\mathcal{C}_0$, or
\item $L$ intersects the curves of $\mathcal{C}_0$ transversally in at most two singular points of $\mathcal{C}_0$
\end{enumerate}
then $\mathcal{C}_0$ has a unique equisingular symplectic isotopy class if and only if $\mathcal{C}_1$ does.
\end{thm}

\begin{prop}[\cite{GS}]\label{p:birationalequivalence}
Suppose $\Sigma_1, \Sigma_2 \subset(M; \omega)$ are birationally equivalent. There is a unique equisingular symplectic isotopy class for $\Sigma_1 \subset(M; \omega)$, if and only if there is a unique equisingular symplectic isotopy class for $\Sigma_2 \subset(M;\omega)$. Moreover, if the equisingular symplectic isotopy class contains complex representatives for one, it contains complex representatives for the other.
\end{prop}

We split the proof degree by degree. For each case, we blow up $N$ times to a resolution where the proper transform of the resolution is smooth and has self-intersection $+1$. We then apply McDuff's theorem to identify the $+1$--sphere with a line in $\CP \# N \overline{\CP}$. Using the lemmas presented in Section~\ref{s:recap}, we then express the possible homology classes of the components of the total transform of the curve in an orthogonal basis of $\CP \# N \overline{\CP}$, consisting of the homology class $h$ of a line and of $N$ disjoint exceptional curves denoted by $e_i$. This step will only be detailed for the first case (the argument is the same for all the other cases).  We finally use those homology classes and Lemma~\ref{l:blowdown} to blow down the total tranform of the curves to a configuration of curves in $\CP$ for which we can show existence and uniqueness (for instance using Theorem~\ref{t:addtheline}) and we conclude thanks to Proposition~\ref{p:birationalequivalence}.

For some cases, we will use the existence and uniqueness of equisingular isotopy class of the following configurations.

\begin{prop}[{\cite[Lemma~2.7]{Starkston-fillings}}]\label{p:sixlines}
A symplectic line arrangement in $\CP$ with at most six lines has a unique symplectic isotopy class.
\end{prop}

Let $\mathcal{G}_3$ denote the configuration consisting of two conics $Q_1$ and $Q_2$ and a line $L_1$ tangent to both $Q_1$ and $Q_2$ at distinct points such that $Q_1$ and $Q_2$ intersect
at one point with multiplicity $3$ and at another point transversally. See Figure~\ref{fig:G3}.

\begin{figure}[h]
	\centering
	\includegraphics[scale=0.5]{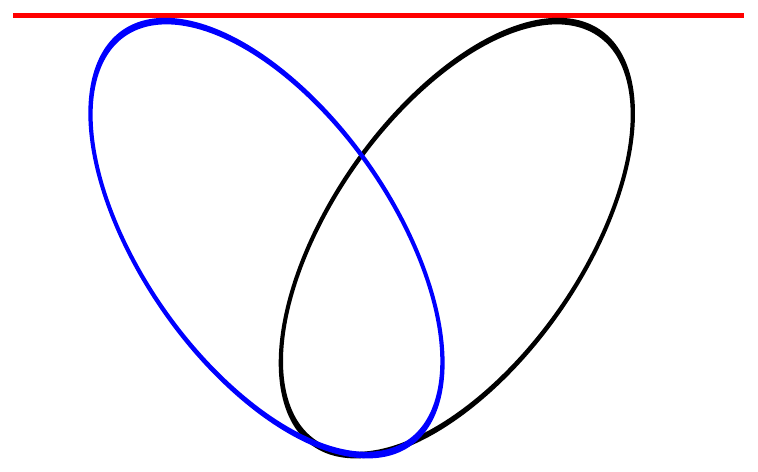}\\
	\caption{A $\mathcal{G}_3$ configuration.}
	\label{fig:G3}
\end{figure}

\begin{prop}[\cite{GS}]\label{p:existenceG3}
The configuration $\mathcal{G}_3$ in $\CP$ has a unique equisingular symplectic isotopy class.
\end{prop}

In the following subsections, some of the figures represent configurations of curves together with their homology classes and self-intersection numbers. For the sake of clarity, we omit the self-intersection numbers for all $(-2)$--curves; transverse intersections will be drawn transversely, simple tangencies as tangencies, while an order-$t$ tangency will be indicated by a ``$t$'' next to the tangency point.

\subsection{Sextics} \label{s:existencesextics}

We start with the rational cuspidal sextics. We gather similar cases in the same propositions.

\begin{prop} \label{p:sextic[4,2,2,..]}
If a rational cuspidal sextic $C$ has one of the following three types
\begin{enumerate}
\item $[[4,2,2,2,2]]$
\item $[[4,2,2,2],[2]]$
\item $[[4,2,2],[2,2]]$
\end{enumerate} 
then the only relatively minimal symplectic embedding of $C$ is into $\CP$ and this embedding is unique up to symplectic isotopy.
\end{prop}

\begin{proof}
For the curve of type $[[4,2,2,2,2]]$, blow up one more time than the minimal resolution at the cusp, and twice at an arbitrary smooth point as in Figure~\ref{fig:6McDuff2}. The homology classes relative to the $+1$--line are uniquely determined (up to reordering the $e_i$ classes), and use eight $e_i$ classes (the same number of exceptional divisors as in the resolution). We detail this step for this case only. By Lemma~\ref{l:adjclass}, we show that the $(-2)$--curve intersecting the line once has homology class $h-e_1-e_2-e_3$, the $(-1)$--curve intersecting the line at the same point has homology class $h-e_4-e_5$ and the $(-1)$--curve intersecting the line at another point has homology class $h-e_1-e_4$ (because it is disjoint from the two previous curves). The $(-2)$--curve that intersects once the curve with homology class $h-e_1-e_4$ and is disjoint from the other curves has homology class $e_4-e_5$ by Lemma~\ref{l:2chain}. For the chain of $(-2)$--curves, Lemma~\ref{l:2chain} show that there are two possibilities for the homology classes, namely $(h-e_1-e_2-e_3, e_2-e_6,e_6-e_7, e_7-e_8)$ and $(h-e_1-e_2-e_3, e_2-e_6,e_3-e_2, e_1-e_3)$. However the chain of $(-2)$--curves is disjoint from the curve with homology class $h-e_1-e_4$, so the second possibility is not allowed. Next, the homology class of the $(-3)$--sphere is of the form $e_i -e_j-e_k$ by Lemma~\ref{l:adjclass}. It has one $e_i$ class in common with $e_6-e_7$ by Lemma~\ref{l:consecutive}, and the $e_i$ with coefficient $+1$ cannot be $e_6$ or $e_7$ by Lemma~\ref{l:pos}. Using Lemma~\ref{l:share2}, the homology class of the $3$--sphere is necessarily of the form $e_i -e_2 -e_6$ (because it is disjoint from the curve with homology class $e_2 -e_6$). Finally the $(-3)$--sphere is disjoint from the curve with homology class $h-e_1-e_2-e_3$ (so the $e_i$ class with coefficient $+1$ could either be $e_1$ or $e_3$) and from the curve with homology class $h-e_1-e_4$, hence its homology class is $e_3 -e_2 -e_6$.

Therefore $C$ embeds symplectically minimally only in $\CP$: this is because the number of exceptional divisors in the resolution is the same as the number of exceptional classes used and~\cite[Theorem~7.3]{Wendl} (see also~\cite[Theorem~1.4]{GS}).
Using Lemma~\ref{l:blowdown} to blow down the exceptional divisors in the $e_i$ classes, the configuration of curves descends to a configuration of four lines in $\CP$. This has a unique symplectic isotopy class by Proposition~\ref{p:sixlines}. Since this line configuration is birationally derived from $C$, Proposition~\ref{p:birationalequivalence} implies that $C$ has a unique isotopy class in $\CP$.

\begin{figure}[h]
	\centering
		\includegraphics[scale=0.3]{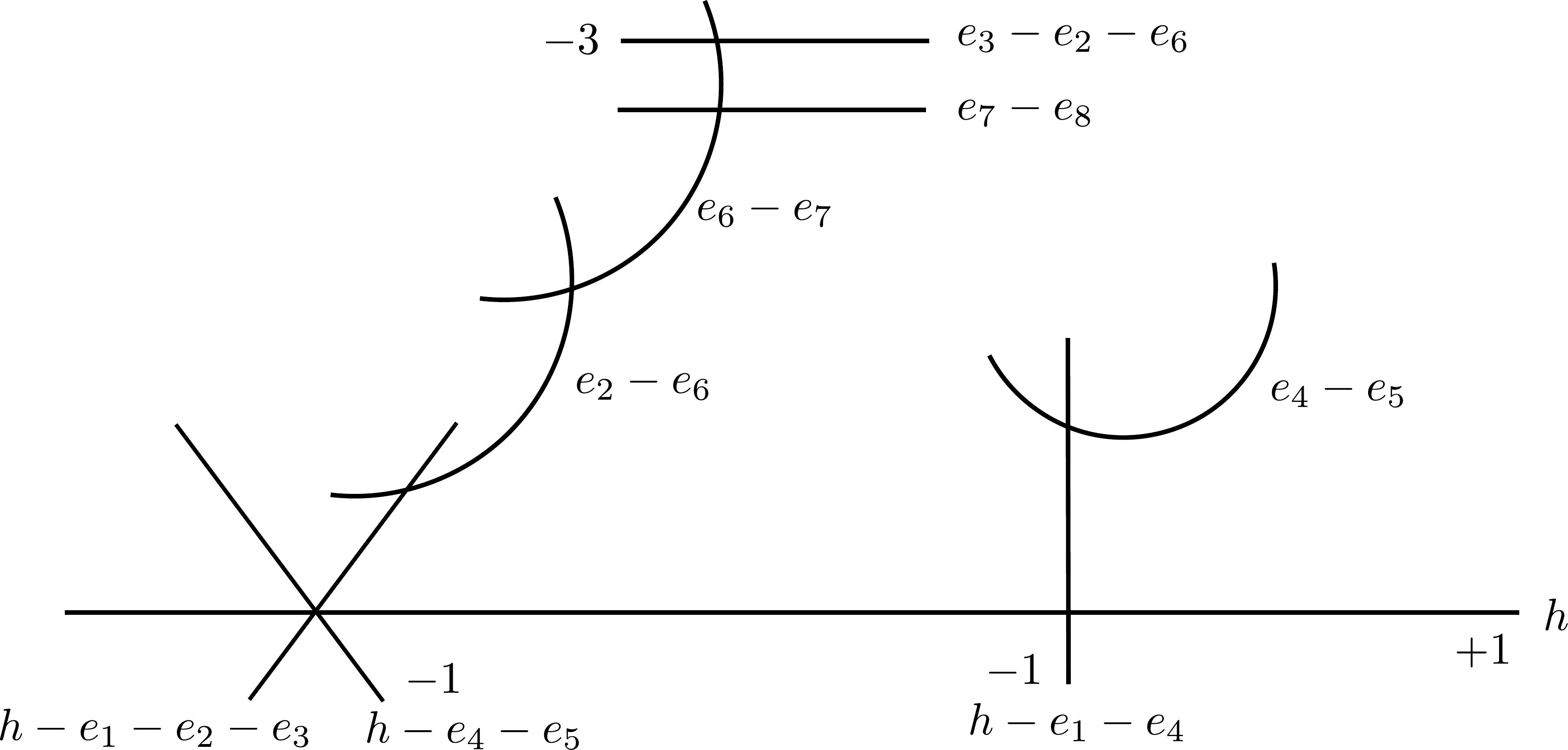}
	\caption{A resolution of a rational cuspidal sextic of type $[[4,2,2,2,2]]$ with the only possible homological embedding.}
	\label{fig:6McDuff2}
\end{figure}

For the curve of type $[[4,2,2,2],[2]]$, blow up one more time than the minimal resolution at each cusp, and once at an arbitrary smooth point as in Figure~\ref{fig:6McDuff3}. The homology classes relative to the $+1$--line are uniquely determined, and use eight $e_i$ classes (the same number of exceptional divisors as in the resolution). Therefore $C$ embeds symplectically minimally only in $\CP$. Using Lemma~\ref{l:blowdown} to blow down the exceptional divisors in the $e_i$ classes, the configuration of curves descends to a configuration of six lines $\{L_i\}$ in $\CP$. This has a unique symplectic isotopy class by Proposition~\ref{p:sixlines}. Since this line configuration is birationally derived from $C$, Proposition~\ref{p:birationalequivalence} implies that $C$ has a unique isotopy class in $\CP$.

\begin{figure}[h]
	\centering
		\includegraphics[scale=0.3]{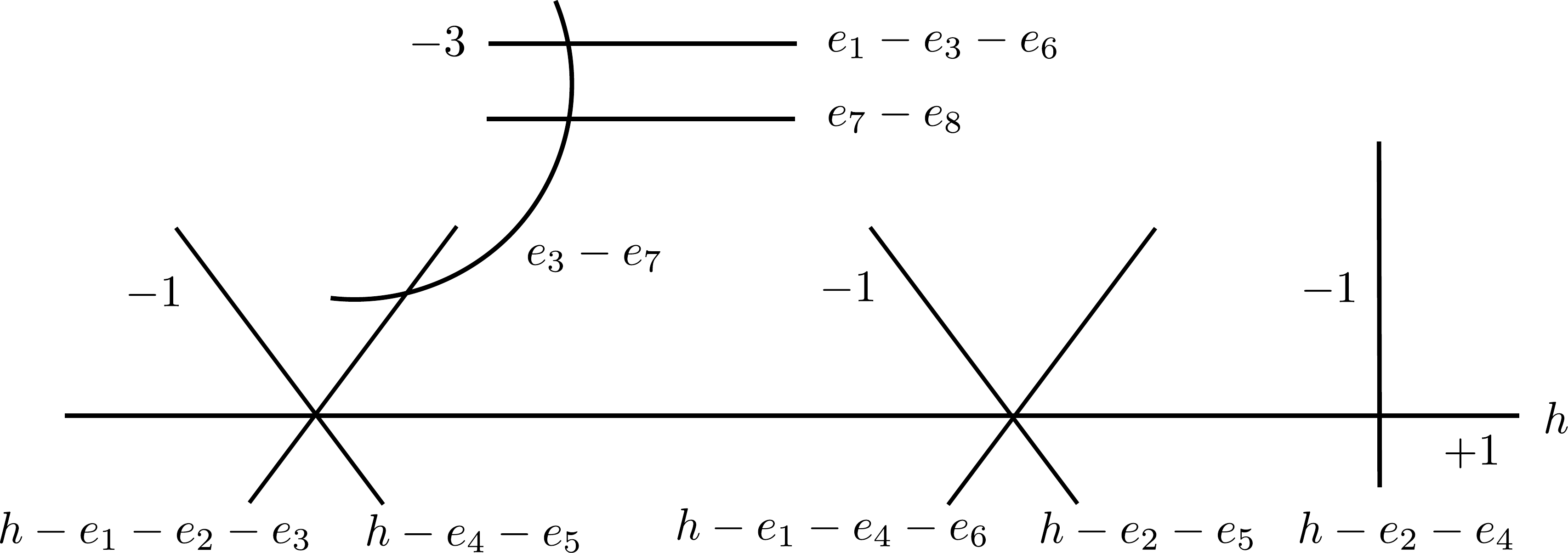}\\
	\caption{A resolution of a rational cuspidal sextic of type $[[4,2,2,2],[2]]$ with the only possible homological embedding.}
	\label{fig:6McDuff3}
\end{figure}

For the curve of type $[[4,2,2],[2,2]]$, blow up one more time than the minimal resolution at each cusp, and once at an arbitrary smooth point as in Figure~\ref{fig:6McDuff4}. The homology classes relative to the $+1$--line are uniquely determined, and use eight $e_i$ classes (the same number of exceptional divisors as in the resolution). Therefore $C$ embeds symplectically minimally only in $\CP$. Using Lemma~\ref{l:blowdown} to blow down the exceptional divisors in the $e_i$ classes, the configuration of curves descends to a configuration of six lines. This implies that $C$ has a unique isotopy class in $\CP$.
\end{proof}

\begin{figure}[h]
	\centering
		\includegraphics[scale=0.3]{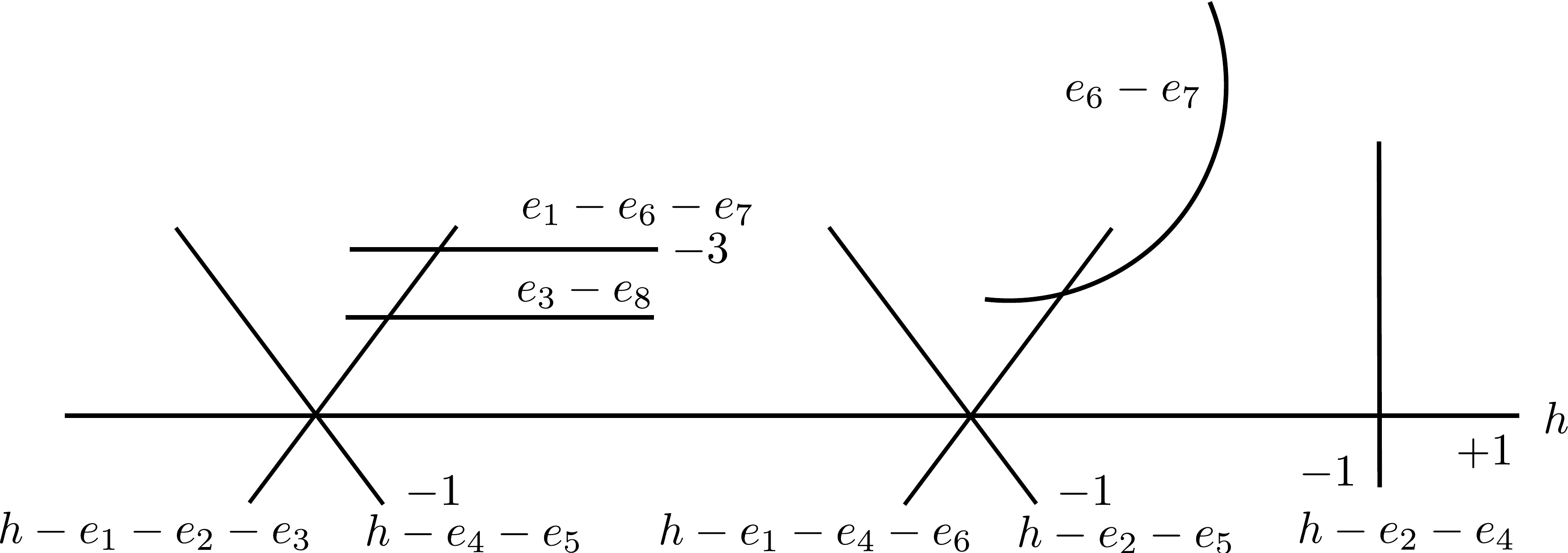}\\
	\caption{A resolution of a rational cuspidal sextic of type $[[4,2,2],[2,2]]$ with the only possible homological embedding.}
	\label{fig:6McDuff4}
\end{figure}

\begin{prop} \label{p:sextic[4]}
If a rational cuspidal sextic $C$ has one of the following three types
\begin{enumerate}
\item $[[4],[2,2,2,2]]$
\item $[[4],[2,2,2],[2]]$
\item $[[4],[2,2],[2,2]]$
\end{enumerate} 
then the only relatively minimal symplectic embedding of $C$ is into $\CP$ and this embedding is unique up to symplectic isotopy. 
\end{prop}

\begin{proof}
For the curve of type $[[4],[2,2,2,2]]$, blow up three more times than the minimal resolution at the cusp of type $[4]$ as in Figure~\ref{fig:6McDuff5}. The homology classes relative to the $+1$--line are uniquely determined, and use eight $e_i$ classes (the same number of exceptional divisors as in the resolution). Therefore $C$ embeds symplectically minimally only in $\CP$. Using Lemma~\ref{l:blowdown} to blow down the exceptional divisors in the $e_i$ classes, the configuration of curves descends to a configuration of three concurrent lines $\{L_i\}$ and a conic $Q$ in $\CP$, where $Q$ is tangent to $L_1$ and $L_2$, and the other intersections are generic. To see that this configuration has a unique symplectic isotopy class, start with the conic $Q$, which is known to have a unique symplectic isotopy class, then add successively $L_1$, $L_2$ and $L_3$ using Theorem~\ref{t:addtheline}. Since this curve configuration is birationally derived from $C$, Proposition~\ref{p:birationalequivalence} implies that $C$ has a unique isotopy class in $\CP$.

\begin{figure}[h]
	\centering
		\includegraphics[scale=0.3]{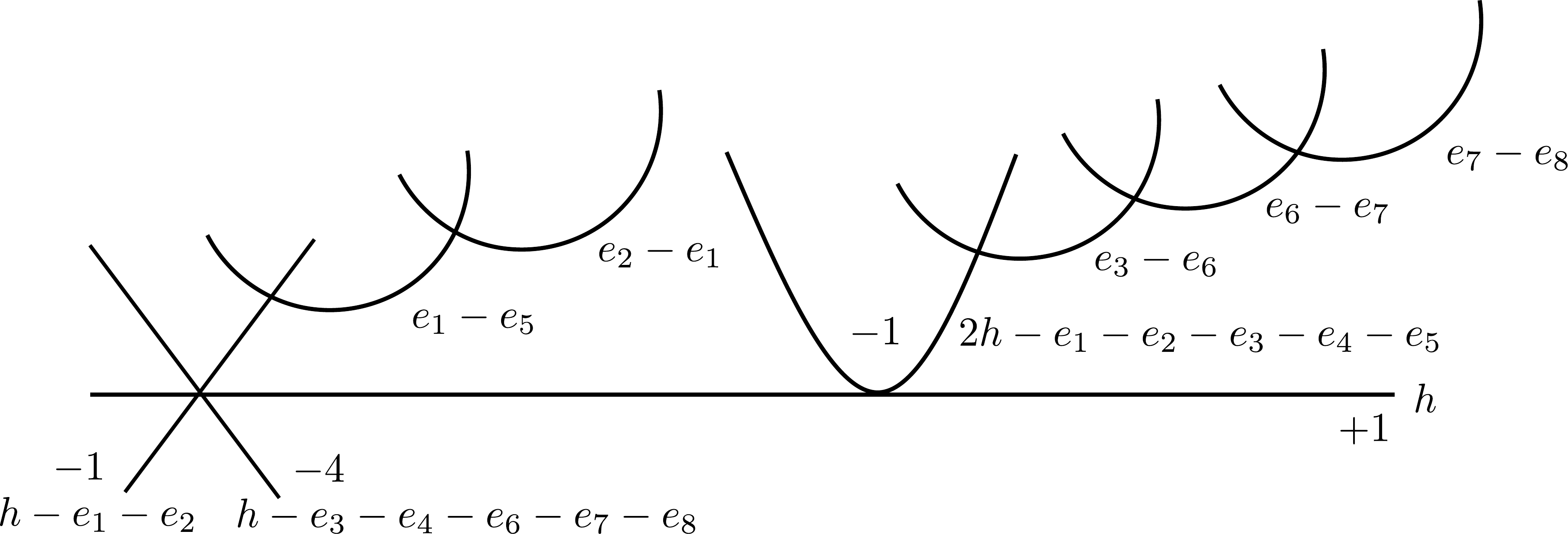}\\
	\caption{A resolution of a rational cuspidal sextic of type $[[4],[2,2,2,2]]$ with the only possible homological embedding.}
	\label{fig:6McDuff5}
\end{figure}

For the curve of type $[[4],[2,2,2],[2]]$, blow up three more times than the minimal resolution at the cusp of type $[4]$ as in Figure~\ref{fig:6McDuff6}. The homology classes relative to the $+1$--line are uniquely determined, and use eight $e_i$ classes (the same number of exceptional divisors as in the resolution). Therefore $C$ embeds symplectically minimally only in $\CP$. Using Lemma~\ref{l:blowdown} to blow down the exceptional divisors in the $e_i$ classes, the configuration of curves descends to a configuration of three concurrent lines $\{L_i\}$ and two conics in $\CP$.  The two conics intersect each other tangentially at a point with multiplicity $3$, $L_1$ is tangent to the two conics at this point, $L_2$ passes through the other point of intersection between the two conics, $L_3$ is tangent to the two conics at distinct points, and the other intersections are generic. To see that this configuration has a unique symplectic isotopy class, notice that the two conics together with $L_3$ form a $\mathcal{G}_3$ configuration (that has a unique symplectic isotopy class according to Proposition~\ref{p:existenceG3}), then add successively $L_1$ and $L_2$ using Theorem~\ref{t:addtheline}. Since this curve configuration is birationally derived from $C$, Proposition~\ref{p:birationalequivalence} implies that $C$ has a unique isotopy class in $\CP$.

\begin{figure}[h]
	\centering
		\includegraphics[scale=0.3]{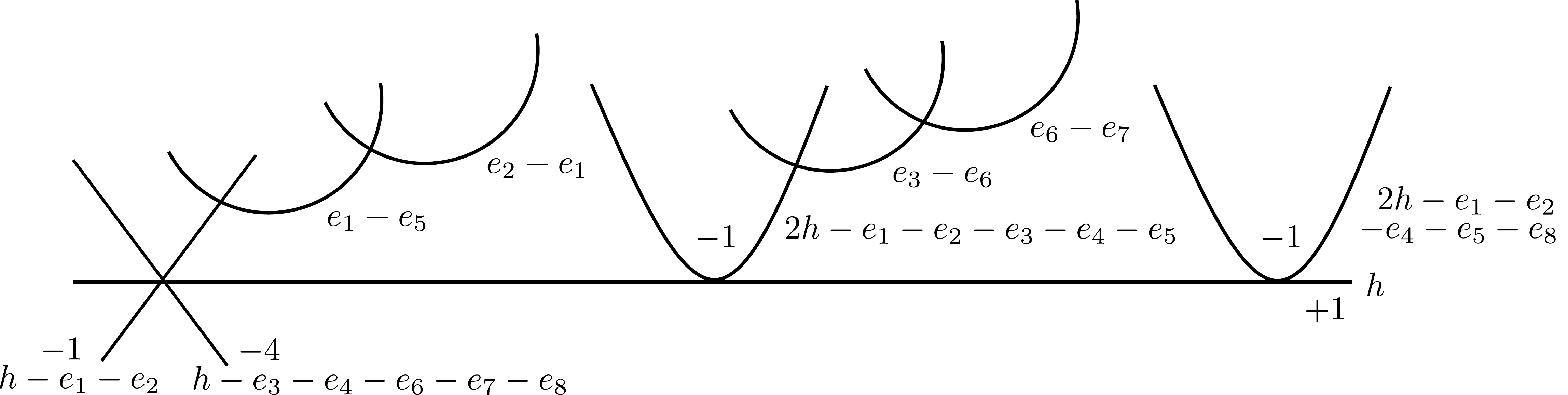}\\
	\caption{A resolution of a rational cuspidal sextic of type $[[4],[2,2,2],[2]]$ with the only possible homological embedding.}
	\label{fig:6McDuff6}
\end{figure}

For the curve of type $[[4],[2,2],[2,2]]$, blow up three more times than the minimal resolution at the cusp of type $[4]$ as in Figure~\ref{fig:6McDuff7}. The homology classes relative to the $+1$--line are uniquely determined, and use eight $e_i$ classes (the same number of exceptional divisors as in the resolution). Therefore $C$ embeds symplectically minimally only in $\CP$. Using Lemma~\ref{l:blowdown} to blow down the exceptional divisors in the $e_i$ classes, the configuration of curves descends to the same configuration as in the previous case. This implies that $C$ has a unique isotopy class in $\CP$.
\end{proof}

\begin{figure}[h]
	\centering
		\includegraphics[scale=0.3]{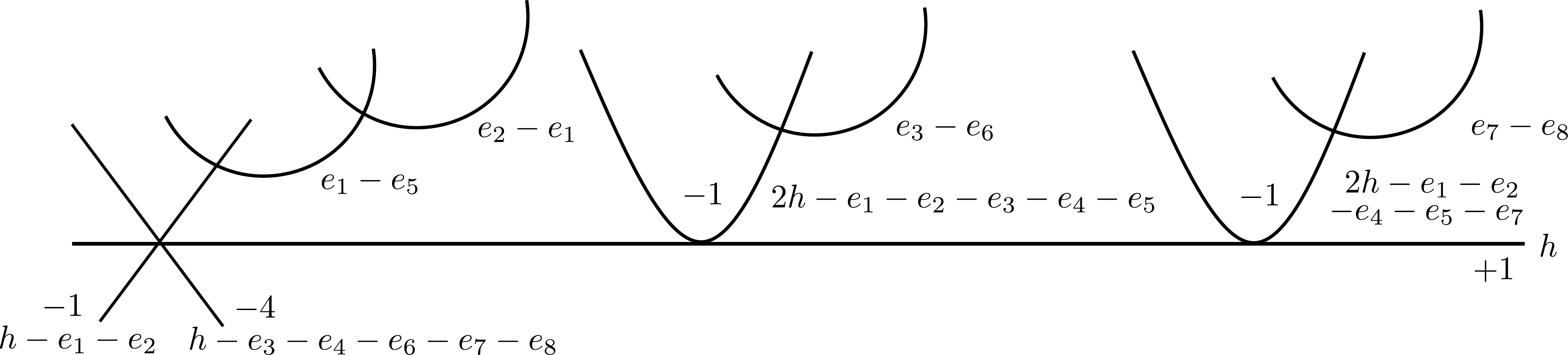}\\
	\caption{A resolution of a rational cuspidal sextic of type $[[4],[2,2],[2,2]]$ with the only possible homological embedding.}
	\label{fig:6McDuff7}
\end{figure}

\begin{prop}\label{p:sextic[3,3,..]}
If a rational cuspidal sextic $C$ has one of the following two types
\begin{enumerate}
\item $[[3,3,3],[2]]$
\item $[[3,3,2],[3]]$
\end{enumerate} 
then there is one relatively minimal symplectic embedding of $C$ into $\CP$ and this embedding is unique up to symplectic isotopy.

If $C$ is of type $[[3,3],[3,2]]$, then the only relatively minimal symplectic embedding of $C$ is into $\CP$ and this embedding is unique up to symplectic isotopy.
\end{prop}

\begin{proof}
For the curve of type $[[3,3,3],[2]]$, blow up three more times than the minimal resolution at the cusp of type $[3,3,3]$ and one more time than the minimal resolution at the cusp of type $[2]$, as in Figure~\ref{fig:6McDuff9}. The homology classes in $H_2 (\CP \# 8\CPbar)$ relative to the $+1$--line are uniquely determined, and use eight $e_i$ classes (the same number of exceptional divisors as in the resolution). Therefore $C$ embeds symplectically minimally in $\CP$. Using Lemma~\ref{l:blowdown} to blow down the exceptional divisors in the $e_i$ classes, the configuration of curves descends to a configuration of four lines. This has a unique symplectic isotopy class by Proposition~\ref{p:sixlines}. Since this curve configuration is birationally derived from $C$, Proposition~\ref{p:birationalequivalence} implies that $C$ has a unique isotopy class in $\CP$.

Note that there might be another relatively minimal symplectic embedding in $S^2 \times S^2$ corresponding to another possibility for homology classes in $H_2 (\CP \# 9\CPbar)$ ($e_8-e_9$ instead of $e_7-e_6$ in Figure~\ref{fig:6McDuff9}); indeed there exists a $(-4)$--curve in the complement of the total transform of $C$ (the curve with holomogy class $e_7-e_6-e_8-e_9$), so the relatively minimal embedding cannot be in $\CP \# \CPbar$.

\begin{figure}[h]
	\centering
		\includegraphics[scale=0.3]{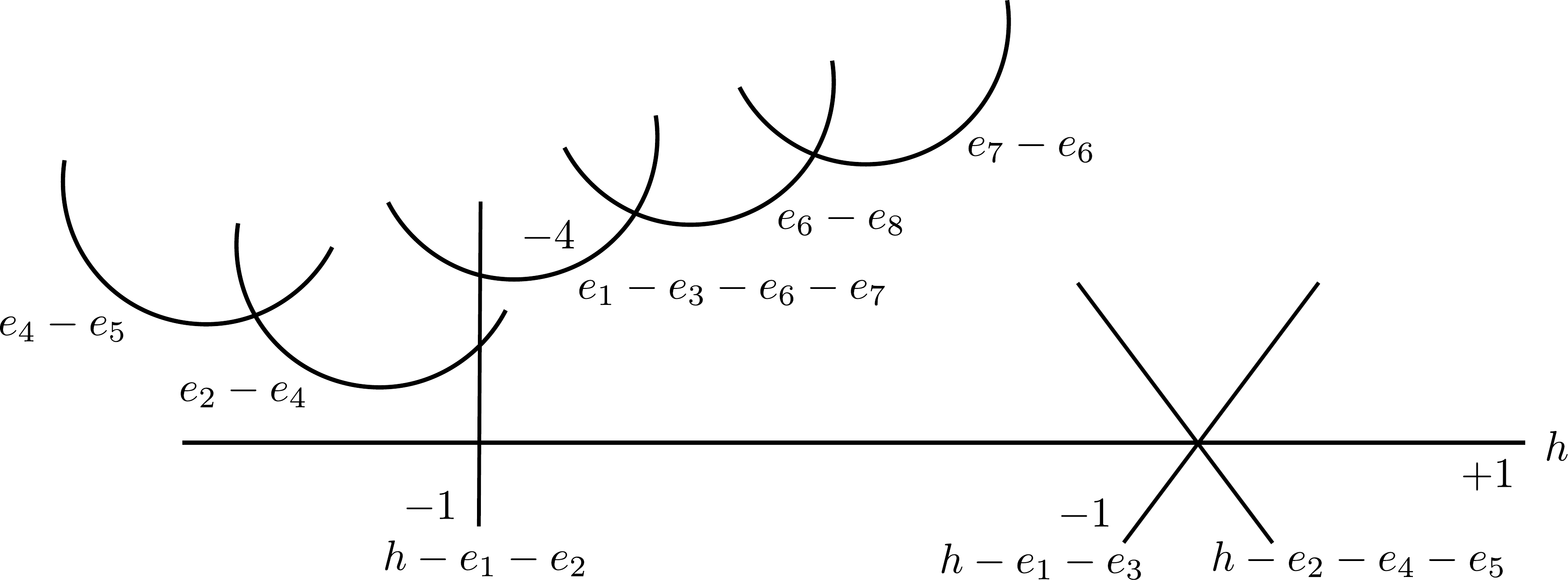}\\
	\caption{A resolution of a rational cuspidal sextic of type $[[3,3,3],[2]]$ with the only possible homological embedding in $\CP \# 8\CPbar$.}
	\label{fig:6McDuff9}
\end{figure}

For the curve of type $[[3,3,2],[3]]$, blow up two more times than the minimal resolution at each cusp as in Figure~\ref{fig:6McDuff10}. The homology classes in $H_2 (\CP \# 8\CPbar)$ relative to the $+1$--line are uniquely determined, and use eight $e_i$ classes (the same number of exceptional divisors as in the resolution). Therefore $C$ embeds symplectically minimally in $\CP$. Using Lemma~\ref{l:blowdown} to blow down the exceptional divisors in the $e_i$ classes, the configuration of curves descends to a configuration of four lines. This implies that $C$ has a unique isotopy class in $\CP$.

Note that there might be another relatively minimal symplectic embedding in $S^2 \times S^2$ corresponding to another possibility for homology classes in $H_2 (\CP \# 9\CPbar)$ ($e_8-e_9$ instead of $e_5-e_6$ in Figure~\ref{fig:6McDuff10}); indeed there exists a $(-4)$--curve in the complement of the total transform of $C$ (the curve with holomogy class $e_5-e_6-e_8-e_9$), so the relatively minimal embedding cannot be in $\CP \# \CPbar$.

\begin{figure}[h]
	\centering
		\includegraphics[scale=0.3]{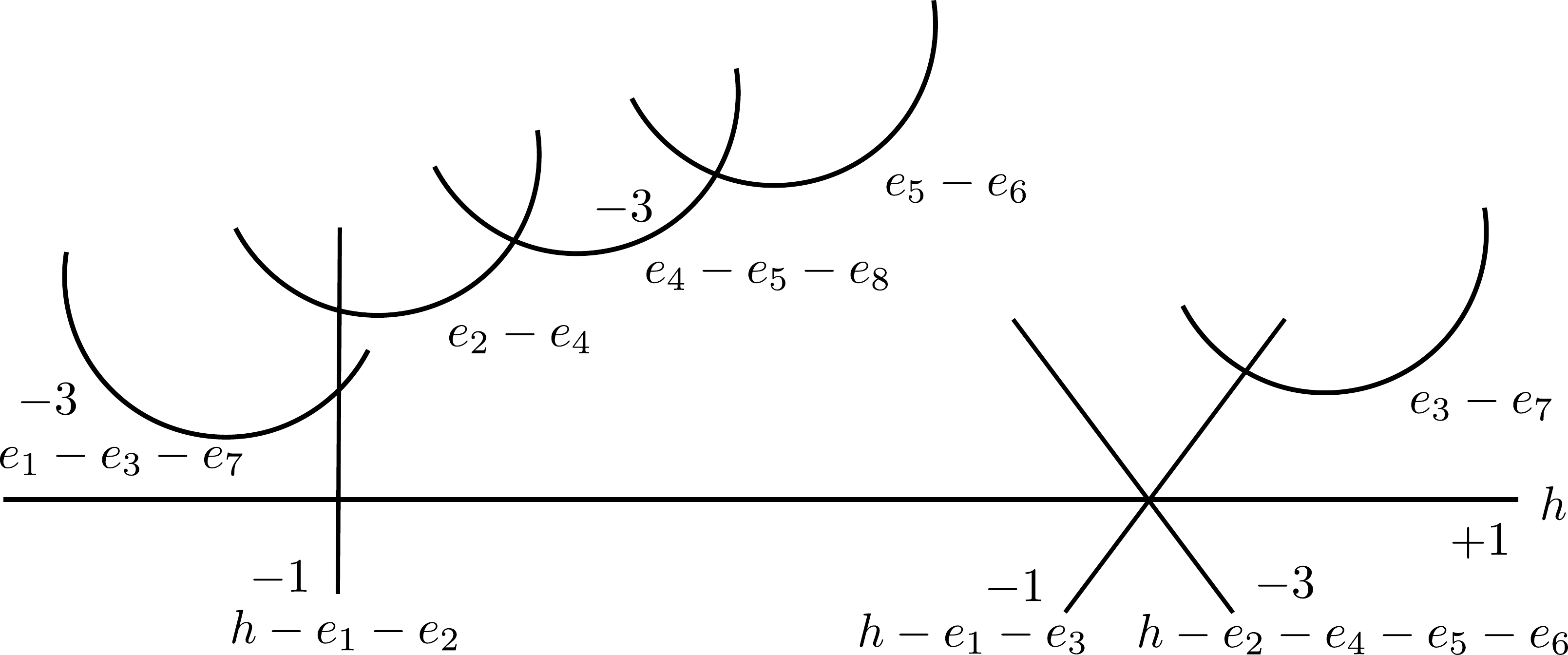}\\
	\caption{A resolution of a rational cuspidal sextic of type $[[3,3,2],[3]]$ with the only possible homological embedding in $\CP \# 8\CPbar$.}
	\label{fig:6McDuff10}
\end{figure}

For the curve of type $[[3,3],[3,2]]$, blow up three more times than the minimal resolution at the cusp of type $[3,3]$ and one more time than the minimal resolution at the cusp of type $[3,2]$ as in Figure~\ref{fig:6McDuff13}. The homology classes relative to the $+1$--line are uniquely determined, and use eight $e_i$ classes (the same number of exceptional divisors as in the resolution). Therefore $C$ embeds symplectically minimally only in $\CP$. Using Lemma~\ref{l:blowdown} to blow down the exceptional divisors in the $e_i$ classes, the configuration of curves descends to a configuration of four lines. This implies that $C$ has a unique isotopy class in $\CP$.
\end{proof}

\begin{figure}[h]
	\centering
		\includegraphics[scale=0.3]{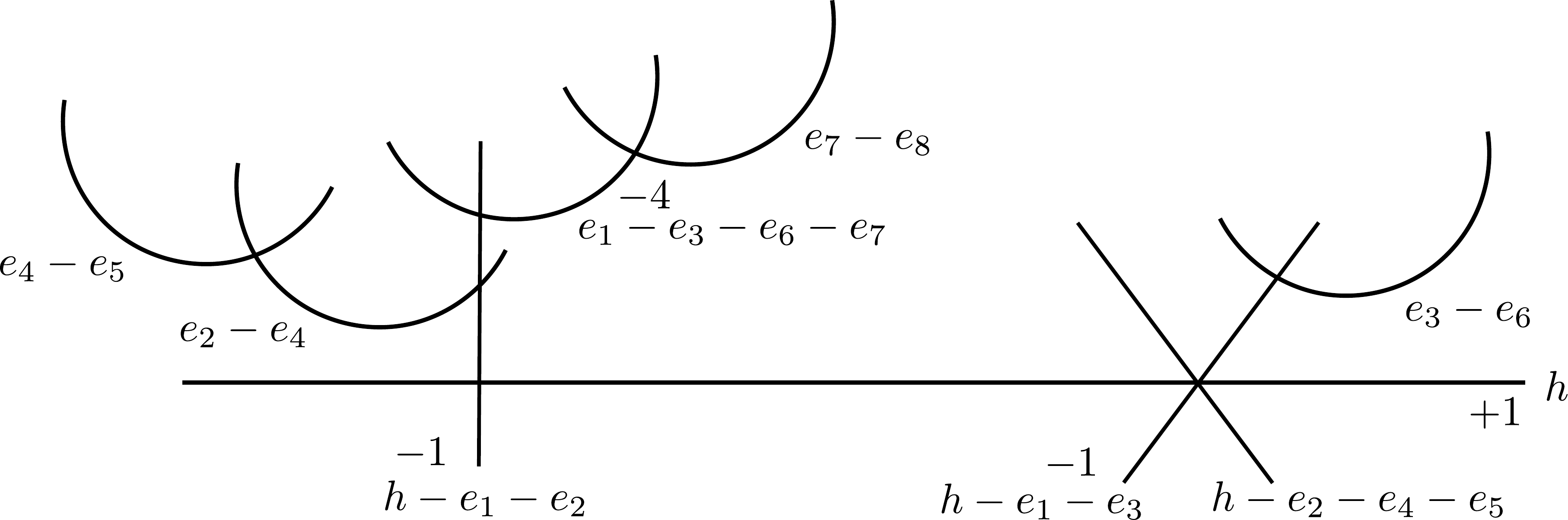}\\
	\caption{A resolution of a rational cuspidal sextic of type $[[3,3],[3,2]]$ with the only possible homological embedding.}
	\label{fig:6McDuff13}
\end{figure}

\subsection{Septics} \label{s:existenceseptics}

We next focus on the remaining cases of rational cuspidal septics. As in the previous subsection, we gather similar cases in the same propositions.

\begin{prop}\label{p:septic[5,2,2]}
If a rational cuspidal septic $C$ is of type $[[5,2,2],[2,2,2]]$, then the only relatively minimal symplectic embedding of $C$ is into $\CP$ and this embedding is unique up to symplectic isotopy.
\end{prop}

\begin{proof}
Blow up two more times than the minimal resolution at the cusp of type $[5,2,2]$ and one more time than the minimal resolution at the cusp of type $[2,2,2]$ as in Figure~\ref{fig:7McDuff2}. The homology classes relative to the $+1$--line are uniquely determined and use nine $e_i$ classes (the same number of exceptional divisors as in the resolution). Therefore $C$ embeds symplectically minimally only in $\CP$. Using Lemma~\ref{l:blowdown} to blow down the exceptional divisors in the $e_i$ classes, the configuration of curves descends to a configuration of four lines. This has a unique symplectic isotopy class by Proposition~\ref{p:sixlines}. Since this curve configuration is birationally derived from $C$, Proposition~\ref{p:birationalequivalence} implies that $C$ has a unique isotopy class in $\CP$.
\end{proof}

\begin{figure}[h]
	\centering
		\includegraphics[scale=0.4]{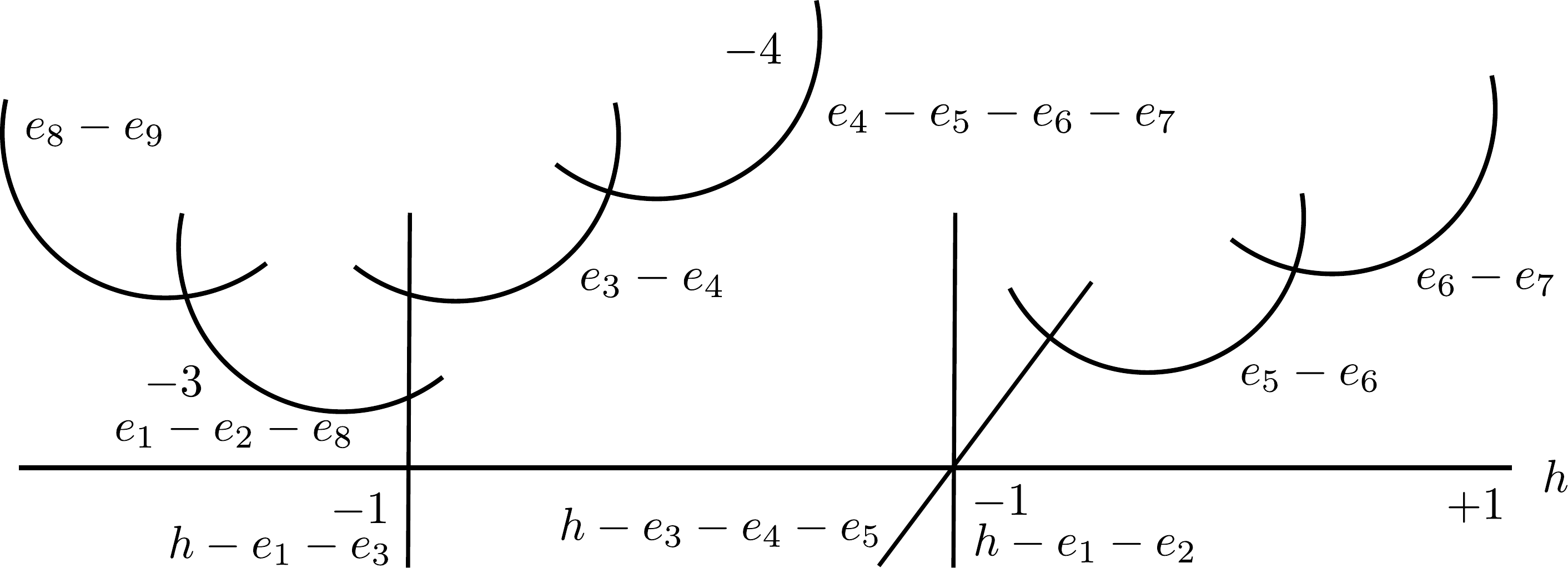}\\
	\caption{A resolution of a rational cuspidal septic of type $[[5,2,2],[2,2,2]]$ with the only possible homological embedding.}
	\label{fig:7McDuff2}
\end{figure}

For the next proposition, we introduce another auxiliary configuration and we show that it has a unique equisingular isotopy class. Let $\mathcal{C}$ denote the ladybug (``coccinelle" in French) configuration, consisting of a quadrilateral with an inscribed conic, and a line passing through a vertex and the points of tangency of the other sides with the conic. 

\begin{prop}\label{p:existenceladybug}
The ladybug configuration $\mathcal{C}$ has a unique equisingular symplectic isotopy class.
\end{prop}

\begin{figure}[h]
	\centering
		\includegraphics[scale=0.25]{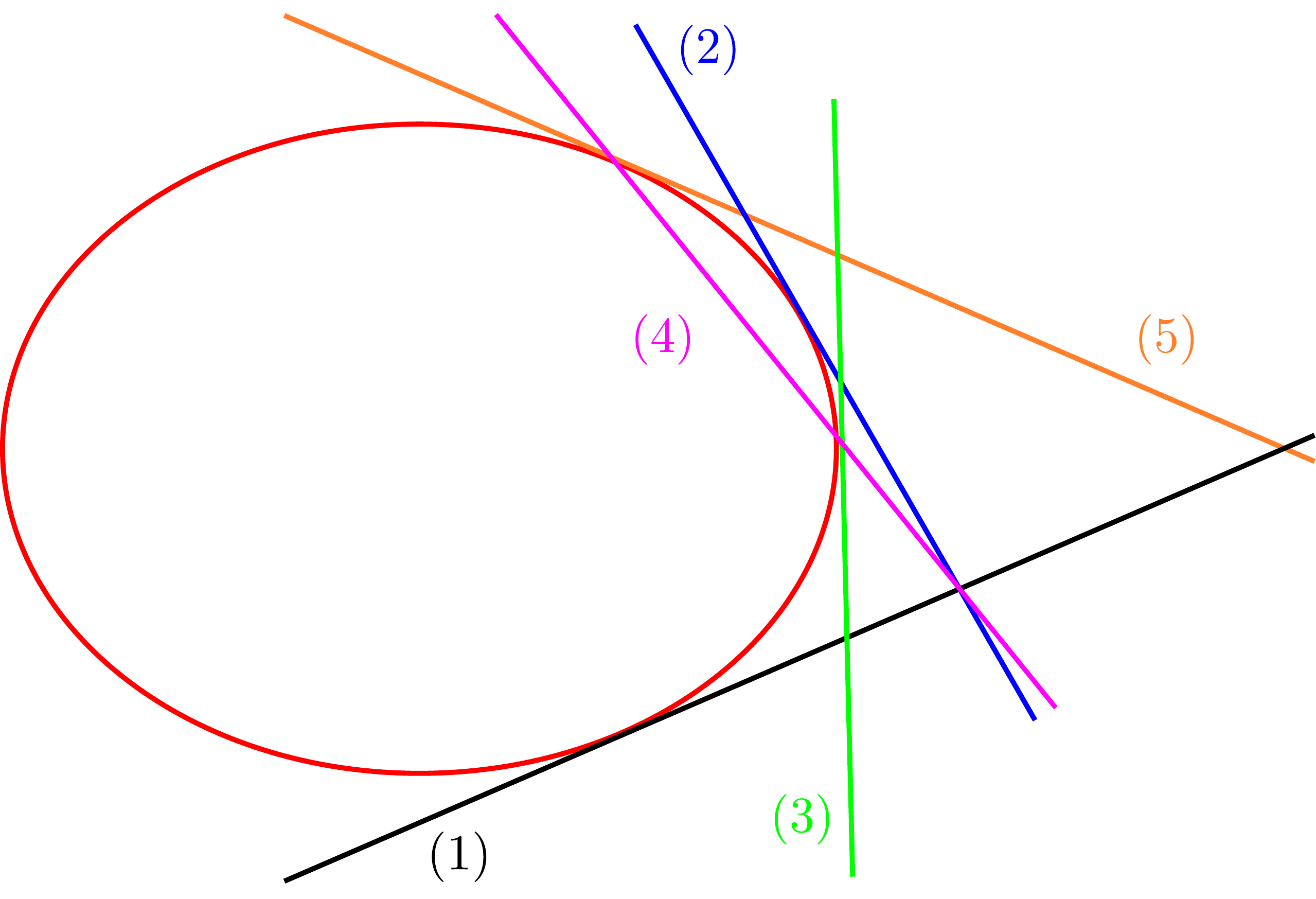}\\
	\caption{A ladybug configuration $\mathcal{C}$.}
	\label{fig:ladybug6}
\end{figure}

\begin{proof}
Start with a conic and apply Theorem~\ref{t:addtheline} to add successively three tangent lines. Then add the line passing through the intersection between the first two lines, and through the tangential intersection between the third line and the conic. Finally add the tangent line at the tangential intersection between the fourth line and the conic, i.e. the fourth side of the quadrilateral.
\end{proof}

\begin{prop} \label{p:[5]}
If a rational cuspidal septic $C$ has one of the following three types
\begin{enumerate}
\item $[[5],[2,2,2,2,2]]$
\item $[[5],[2,2,2,2],[2]]$
\item $[[5],[2,2,2],[2,2]]$
\end{enumerate}
then the only relatively minimal symplectic embedding of $C$ is into $\CP$ and this embedding is unique up to symplectic isotopy.
\end{prop}

\begin{proof}
For the curve of type $[[5],[2,2,2,2,2]]$, blow up three more times than the minimal resolution at the cusp of type $[5]$  as in Figure~\ref{fig:7McDuff3}. The homology classes relative to the $+1$--line are uniquely determined, and use nine $e_i$ classes (the same number of exceptional divisors as in the resolution). Therefore $C$ embeds symplectically minimally only in $\CP$. Using Lemma~\ref{l:blowdown} to blow down the exceptional divisors in the $e_i$ classes, the configuration of curves descends to a configuration of two lines $\{ L_i \}$ and two conics $\{ Q_i \}$. The two conics intersect each other tangentially at a point with multiplicity $3$, $L_1$ is tangent to the two conics at distinct points, $L_2$ passes through the tangential intersection point between $Q_1$ and $Q_2$ and through the intersection point between $L_1$ and $Q_1$, and the other intersections are generic. To see that this configuration has a unique symplectic isotopy class, notice that the two conics together with $L_1$ form a $\mathcal{G}_3$ configuration (that has a unique symplectic isotopy class according to Proposition~\ref{p:existenceG3}), then add $L_2$ using Theorem~\ref{t:addtheline}. Since this curve configuration is birationally derived from $C$, Proposition~\ref{p:birationalequivalence} implies that $C$ has a unique isotopy class in $\CP$.

\begin{figure}[h]
	\centering
		\includegraphics[scale=0.3]{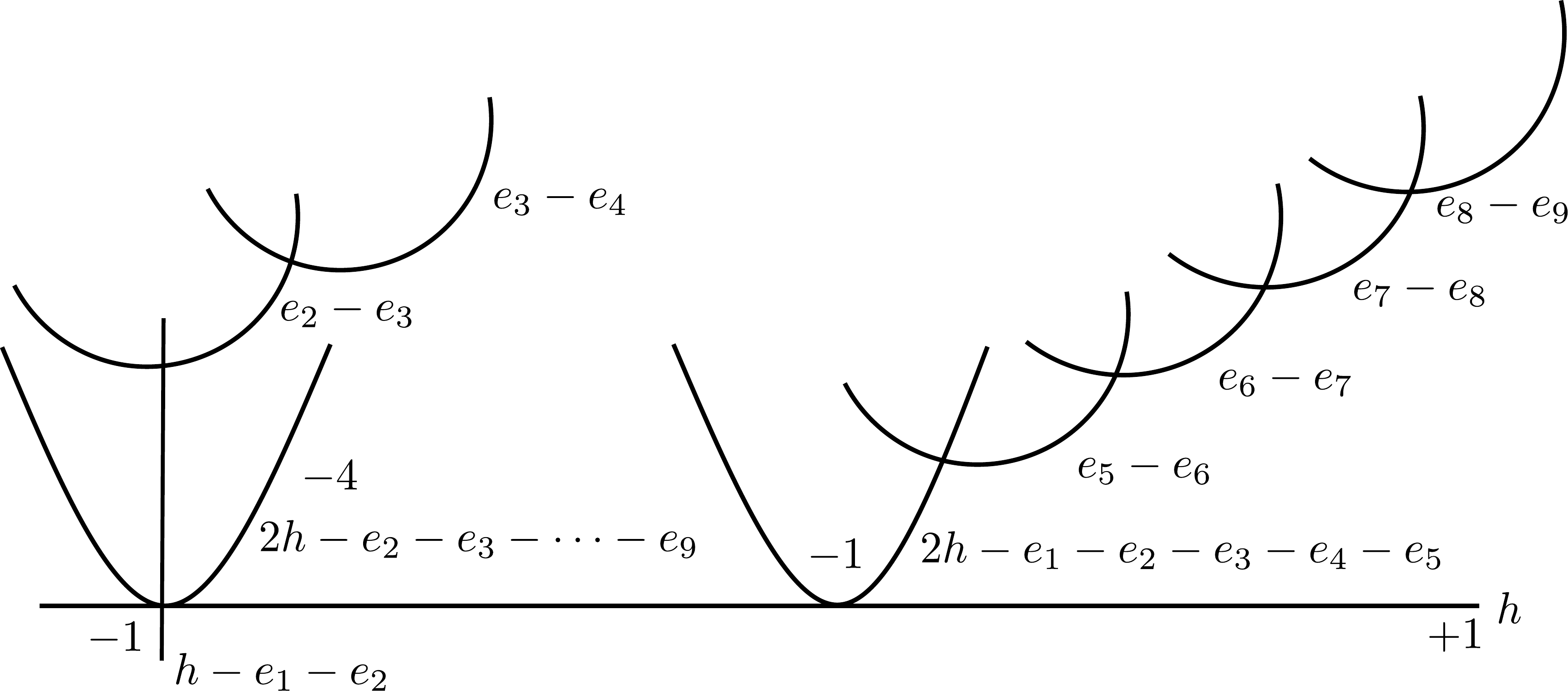}\\
	\caption{A resolution of a rational cuspidal septic of type $[[5],[2,2,2,2,2]]$ with the only possible homological embedding.}
	\label{fig:7McDuff3}
\end{figure}

For the curve of type $[[5],[2,2,2,2],[2]]$, blow up three more times than the minimal resolution at the cusp of type $[5]$ as in Figure~\ref{fig:7McDuff4}. The homology classes relative to the $+1$--line are uniquely determined, and use nine $e_i$ classes (the same number of exceptional divisors as in the resolution). Therefore $C$ embeds symplectically minimally only in $\CP$. 

\begin{figure}[h]
	\centering
		\includegraphics[scale=0.3]{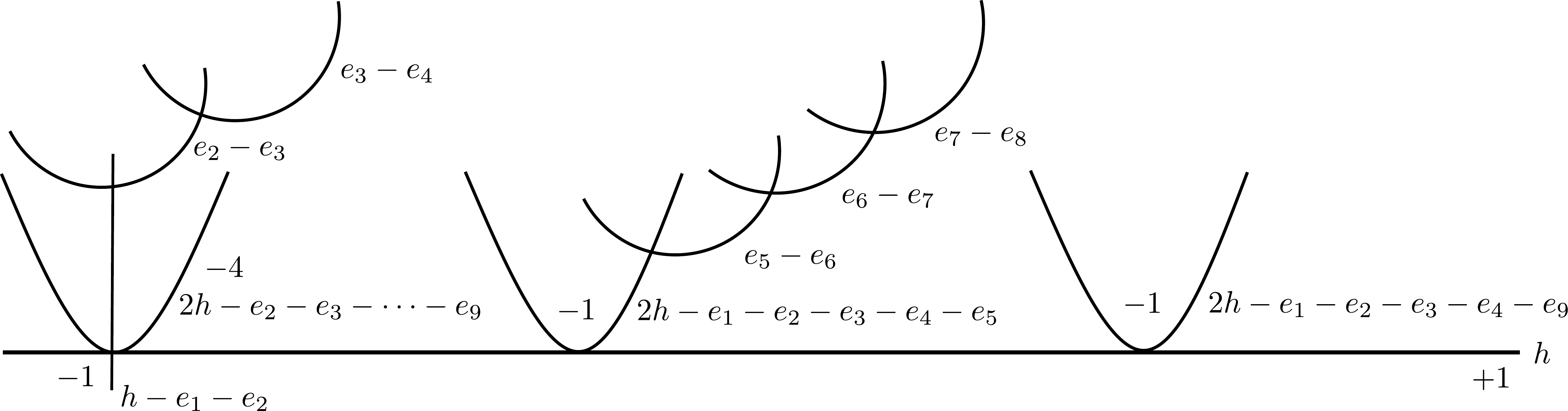}\\
	\caption{A resolution of a rational cuspidal septic of type $[[5],[2,2,2,2],[2]]$ with the only possible homological embedding.}
	\label{fig:7McDuff4}
\end{figure}

Using Lemma~\ref{l:blowdown} to blow down the exceptional divisors in the $e_i$ classes, the configuration of curves descends to a configuration of two lines $\{ L_i \}$ and three conics $\{ Q_i \}$. The three conics intersect each other tangentially at the same point with multiplicity $3$, $L_1$ is tangent to the three conics, $L_2$ passes through the tangential intersection point between the three conics, through the intersection point between $L_1$ and $Q_1$ and through the intersection point between $Q_2$ and $Q_3$, and the other intersections are generic. See Figure~\ref{fig:ladybug2345}. We next show that this configuration is birationally derived from a ladybug configuration $\mathcal{C}$.

Blow up three times at the point of intersection between the three conics (note that the exceptional curves associated to those blow-ups are actually $e_2$, $e_3$ and $e_4$ of Figure~\ref{fig:7McDuff4}) then apply McDuff's Theorem to identify the proper transform of one of the conics to a line as in Figure~\ref{fig:ladybug2345}. The new homology classes relative to this $+1$--line are uniquely determined, and use three $f_i$ classes (classes of the new exceptional curves). The proper transforms of the two remaining conics have self-intersection number $+1$ and each of them intersects the line once, so their homology class is $h$ by Lemma~\ref{l:adjclass}. The proper transform of $L_1$ has a tangency of order $2$ with the line and self-intersection number $+1$, therefore Lemma~\ref{l:adjclass} shows that its homology class is $2h-f_1-f_2-f_3$. Finally, the proper transform of $L_2$ intersects the line once, has self-intersection number $-1$ and is disjoint from the proper transform of $L_1$, therefore we apply Lemma~\ref{l:adjclass} again to show that its homology class is $h-f_1-f_2$.   Using Lemma~\ref{l:blowdown} to blow down the exceptional divisors in the $f_i$ classes (first $f_3$, then $f_1$ and $f_2$), the configuration of curves descends to a ladybug configuration. See Figure~\ref{fig:ladybug2345}.

\begin{figure}[htbp]
\centering
\subfloat[]{\label{fig:ladybuga}\includegraphics[width=0.45\linewidth]{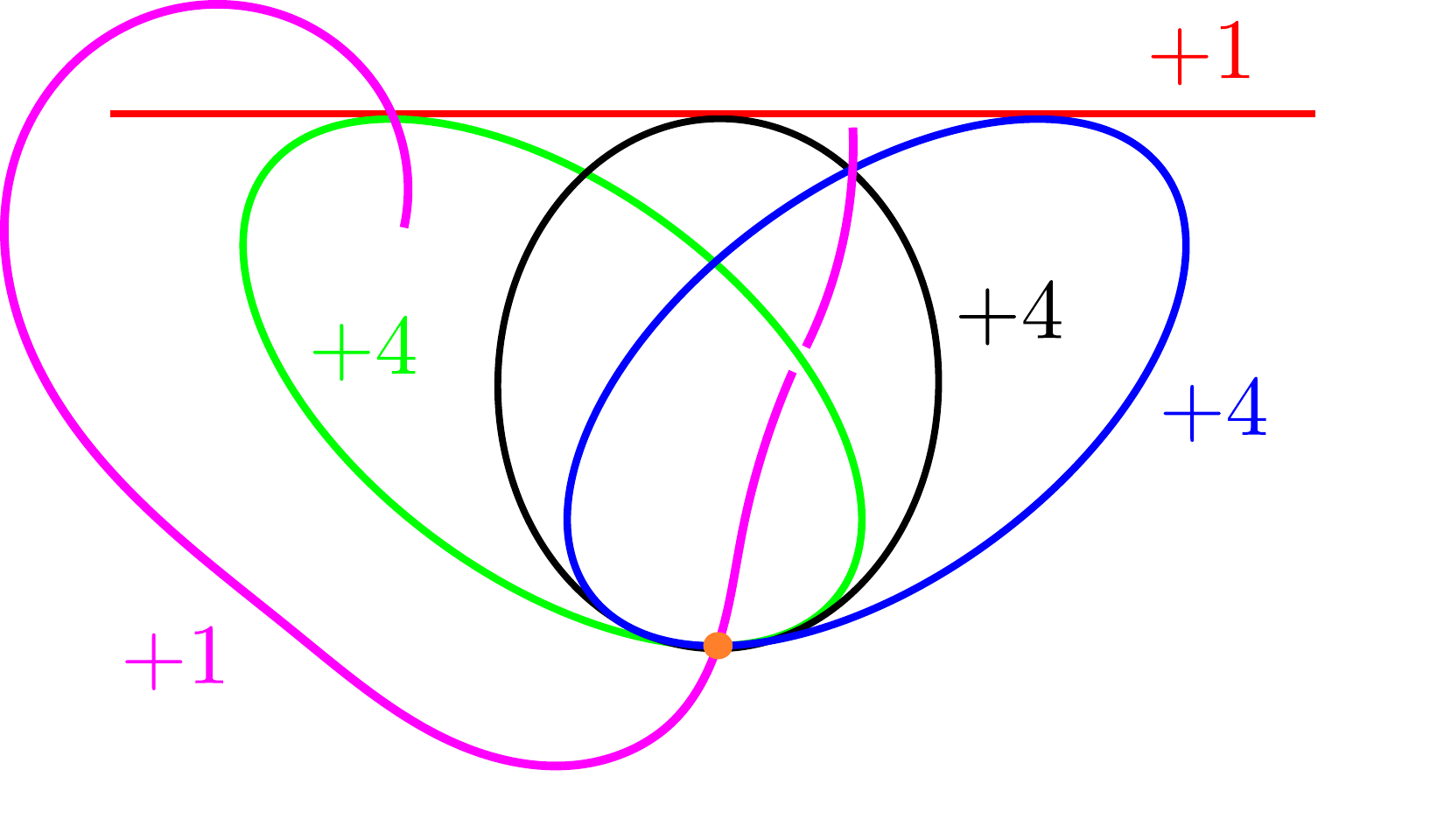}}\qquad
\subfloat[]{\label{fig:ladybugb}\includegraphics[width=0.45\linewidth]{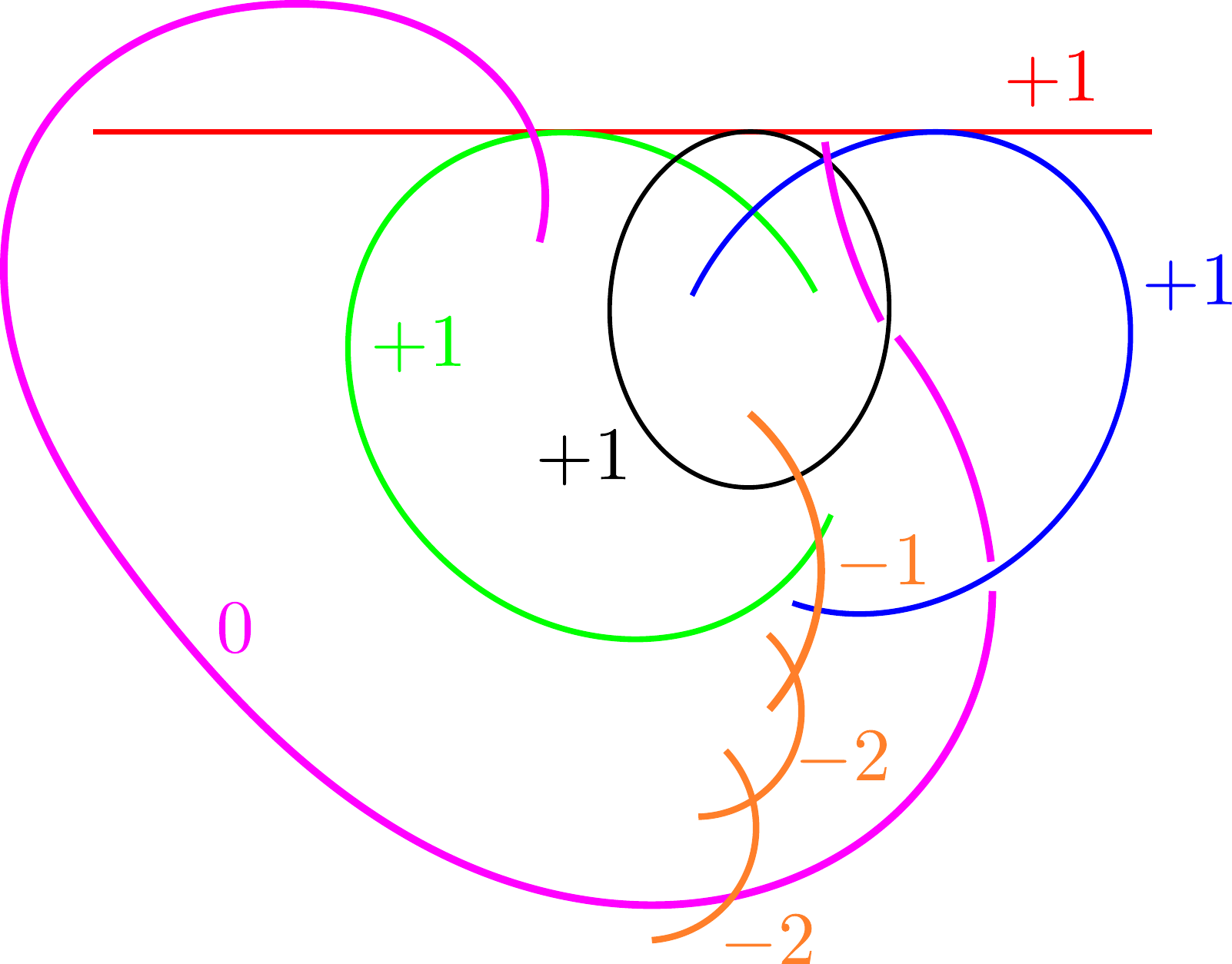}}\\
\subfloat[]{\label{fig:ladybugc}\includegraphics[width=0.45\textwidth]{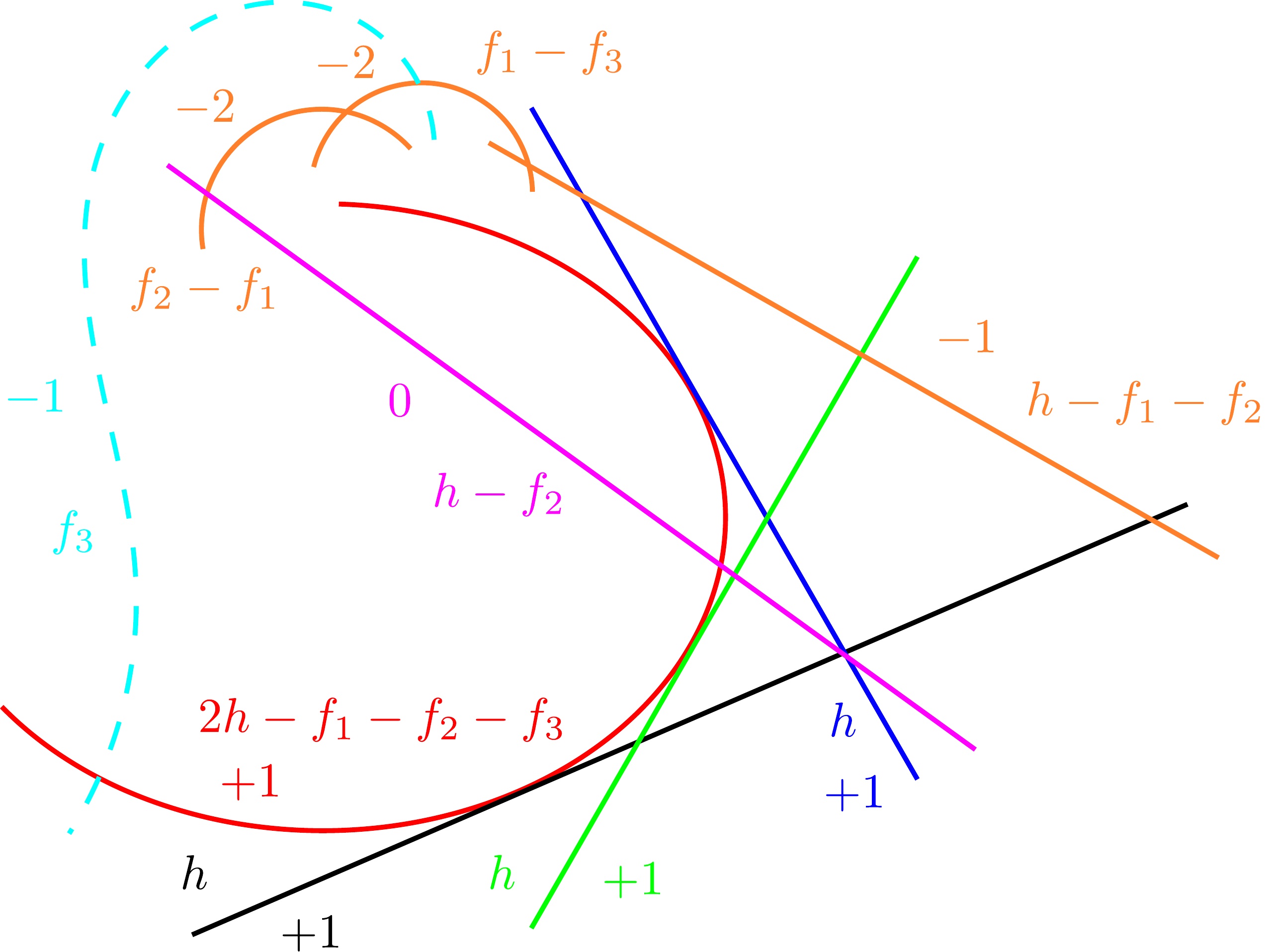}}\qquad%
\subfloat[]{\label{fig:ladubugd}\includegraphics[width=0.45\textwidth]{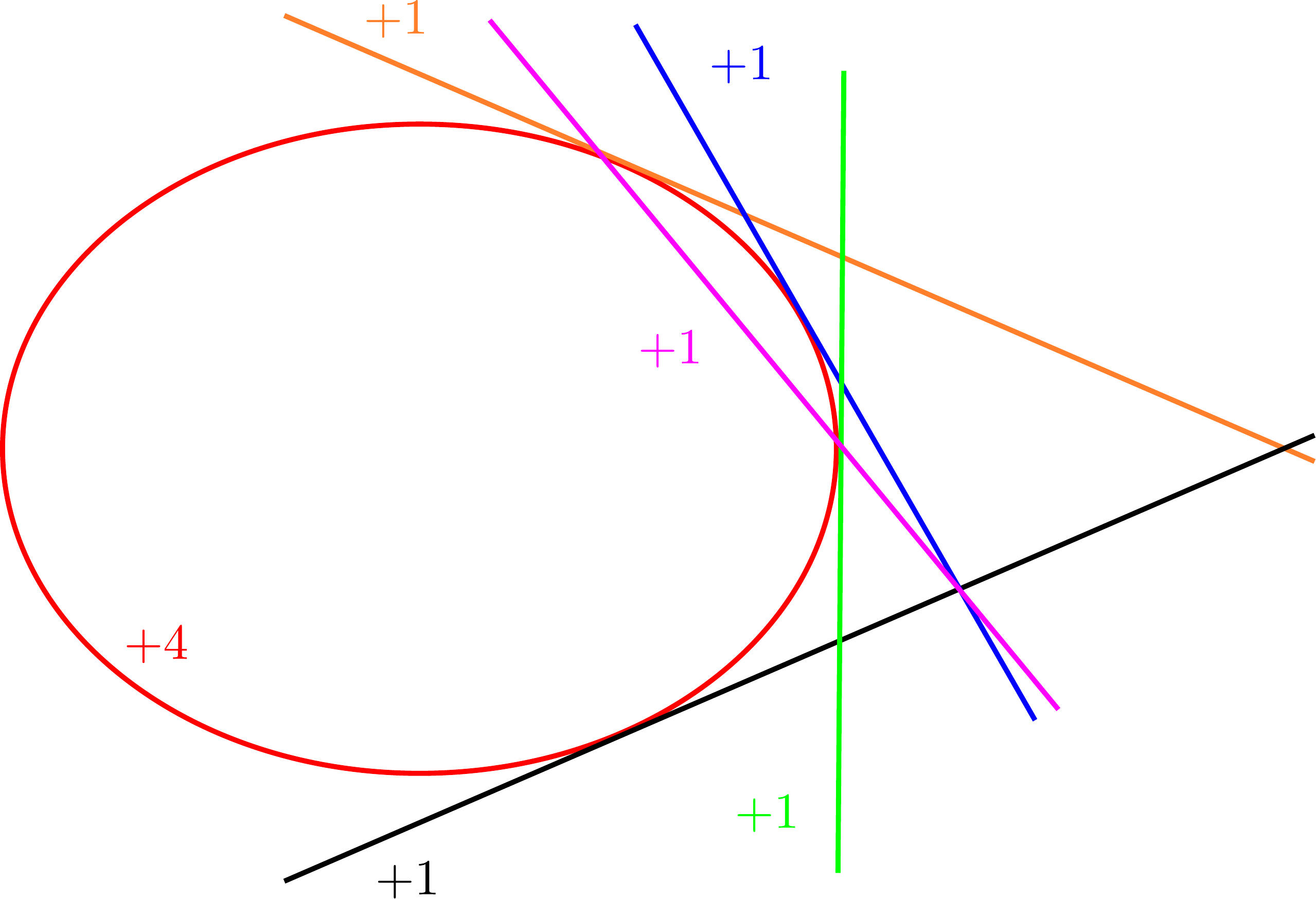}}%
\caption{Birational equivalence from the configuration mentioned in the proof of Proposition~\ref{p:[5]} to a ladybug configuration. The configurations (B) and (C) are the same, just redrawn indicating a symplectomorphism of $\CP \#3 \overline{\CP}$
identifying the black $+1$--curve with a line.}
\label{fig:ladybug2345}
\end{figure}

This has a unique symplectic isotopy class by Proposition~\ref{p:existenceladybug}. Since this curve configuration is birationally derived from $C$, Proposition~\ref{p:birationalequivalence} implies that $C$ has a unique isotopy class in $\CP$.

For the curve of type $[[5],[2,2,2],[2,2]]$, blow up three more times than the minimal resolution at the cusp of type $[5]$ as in Figure~\ref{fig:7McDuff5}. The homology classes relative to the $+1$--line are uniquely determined, and use nine $e_i$ classes (the same number of exceptional divisors as in the resolution). Therefore $C$ embeds symplectically minimally only in $\CP$. Using Lemma~\ref{l:blowdown} to blow down the exceptional divisors in the $e_i$ classes, the configuration of curves descends to the same configuration as in the previous case. Then the rational septic of type $[[5],[2,2,2],[2,2]]$ is also birationally derived from the ladybug configuration. This implies that $C$ has a unique isotopy class in $\CP$.

\begin{figure}[h]
	\centering
		\includegraphics[scale=0.3]{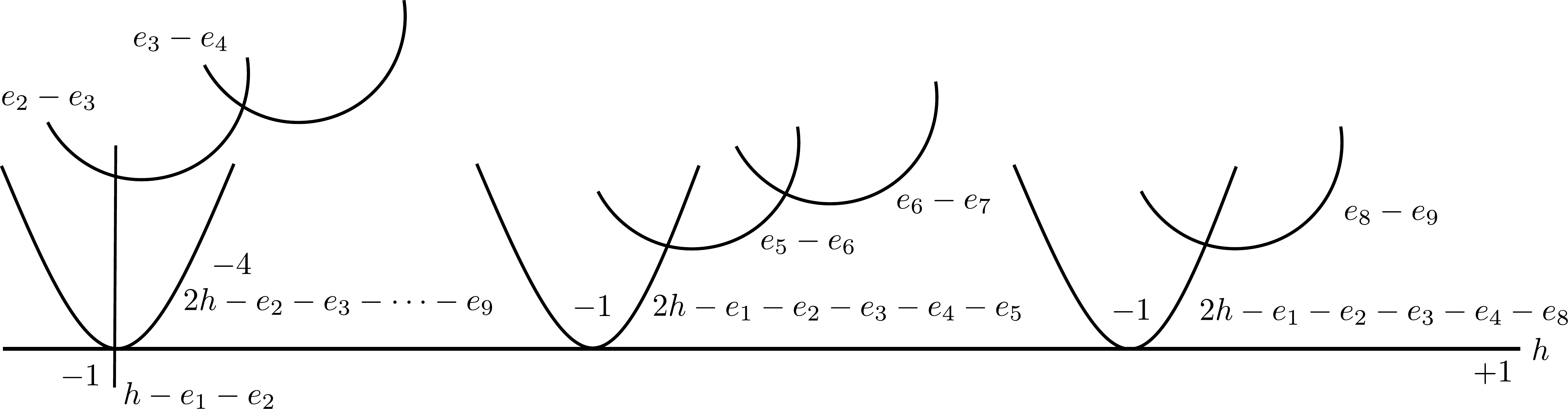}\\
	\caption{A resolution of a rational cuspidal septic of type $[[5],[2,2,2],[2,2]]$ with the only possible homological embedding.}
	\label{fig:7McDuff5}
\end{figure}
\end{proof}

\begin{prop}\label{p:septic[4,3]}
If a rational cuspidal septic $C$ is of type $[[4,3],[3,3]]$, then the only relatively minimal symplectic embedding of $C$ is into $\CP$ and
this embedding is unique up to symplectic isotopy.
\end{prop}

\begin{proof}
Blow up two more times than the minimal resolution at each cusp as in Figure~\ref{fig:7McDuff6}. The homology classes relative to the $+1$--line are uniquely determined, and use nine $e_i$ classes (the same number of exceptional divisors as in the resolution). Therefore $C$ embeds symplectically minimally only in $\CP$. Using Lemma~\ref{l:blowdown} to blow down the exceptional divisors in the $e_i$ classes, the configuration of curves descends to a configuration of four lines. This has a unique symplectic isotopy class by Proposition~\ref{p:sixlines}. Since this curve configuration is birationally derived from $C$, Proposition~\ref{p:birationalequivalence} implies that $C$ has a unique isotopy class in $\CP$.
\end{proof}

\begin{figure}[h]
	\centering
		\includegraphics[scale=0.3]{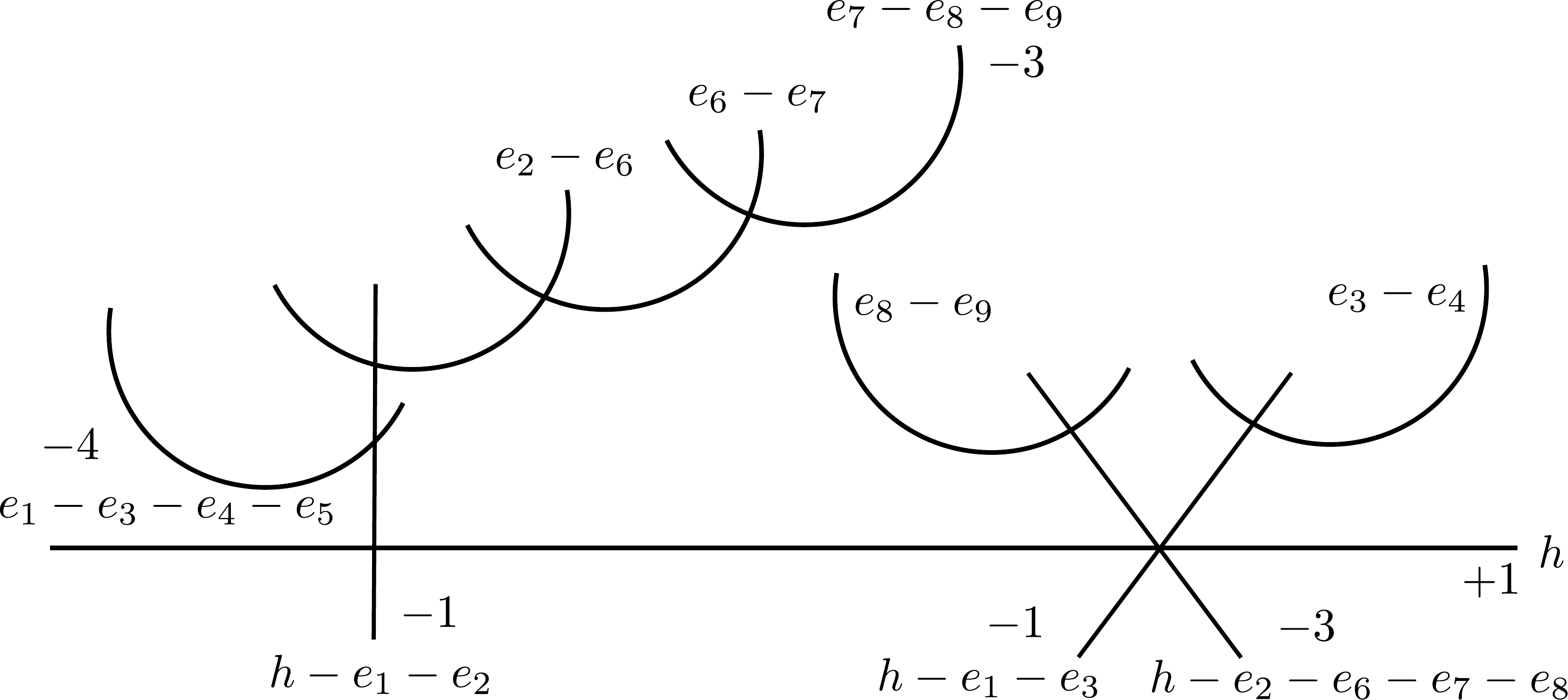}\\
	\caption{A resolution of a rational cuspidal septic of type $[[4,3],[3,3]]$ with the only possible homological embedding.}
	\label{fig:7McDuff6}
\end{figure}

\begin{prop} \label{p:(4,5),(3,10)}
If a rational cuspidal septic $C$ is of type $[[4],[3,3,3]]$, then the only relatively minimal symplectic embedding of $C$ is into $\CP$ and
this embedding is unique up to symplectic isotopy.
\end{prop}

\begin{proof}
Blow up three more times than the minimal resolution at the cusp of type $[4]$ and two more times than the minimal resolution at the cusp of type $[3,3,3]$ as in Figure~\ref{fig:7McDuff7}. The homology classes relative to the $+1$--line are uniquely determined, and use nine $e_i$ classes (the same number of exceptional divisors as in the resolution). Therefore $C$ embeds symplectically minimally only in $\CP$. Using Lemma~\ref{l:blowdown} to blow down the exceptional divisors in the $e_i$ classes, the configuration of curves descends to a configuration of five lines. This has a unique symplectic isotopy class by Proposition~\ref{p:sixlines}. Since this curve configuration is birationally derived from $C$, Proposition~\ref{p:birationalequivalence} implies that $C$ has a unique isotopy class in $\CP$.
\end{proof}

\begin{figure}[h]
	\centering
		\includegraphics[scale=0.3]{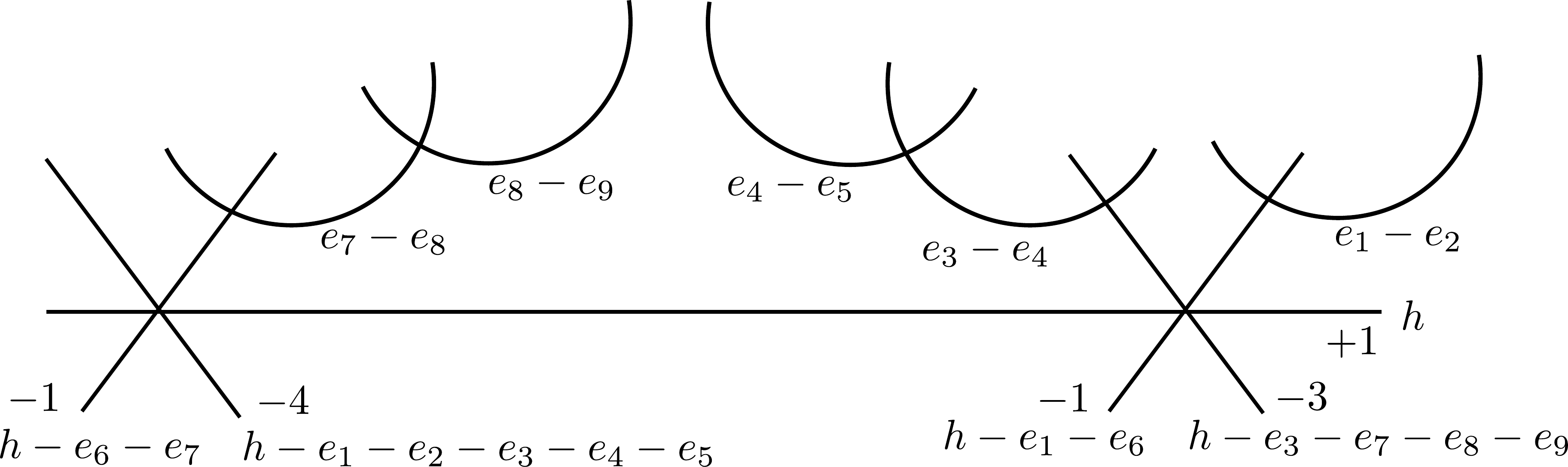}\\
	\caption{A resolution of a rational cuspidal septic of type $[[4],[3,3,3]]$ with the only possible homological embedding.}
	\label{fig:7McDuff7}
\end{figure}

\begin{prop} \label{p:septic[4,2,2,2]}
If a rational cuspidal septic $C$ is of type $[[4,2,2,2],[3,3]]$ then the only relatively minimal symplectic embedding of $C$ is into $\CP$ and this embedding is unique up to symplectic isotopy.
\end{prop}

\begin{proof}
Blow up one more time than the minimal resolution at each cusp as in Figure~\ref{fig:7McDuff8}. The homology classes relative to the $+1$--line are uniquely determined, and use eight $e_i$ classes (the same number of exceptional divisors as in the resolution). Therefore $C$ embeds symplectically minimally only in $\CP$. Using Lemma~\ref{l:blowdown} to blow down the exceptional divisors in the $e_i$ classes, the configuration of curves descends to a configuration of four lines $\{ L_i \}$ and a conic, where $L_1$, $L_2$ and $L_3$ are concurrent, $L_1$ and $L_2$ are tangent to the conic, $L_4$ passes through the intersection point between $L_1$ and the conic and through one intersection point between $L_3$ and the conic, and the other intersections are generic. To see that this configuration has a unique symplectic isotopy class, start with the conic, which is known to have a unique symplectic isotopy class, then add successively $L_1$, $L_2$, $L_3$ and $L_4$ using Theorem~\ref{t:addtheline}. Since this curve configuration is birationally derived from $C$, Proposition~\ref{p:birationalequivalence} implies that $C$ has a unique isotopy class in $\CP$.
\end{proof}

\begin{figure}[h]
	\centering
		\includegraphics[scale=0.3]{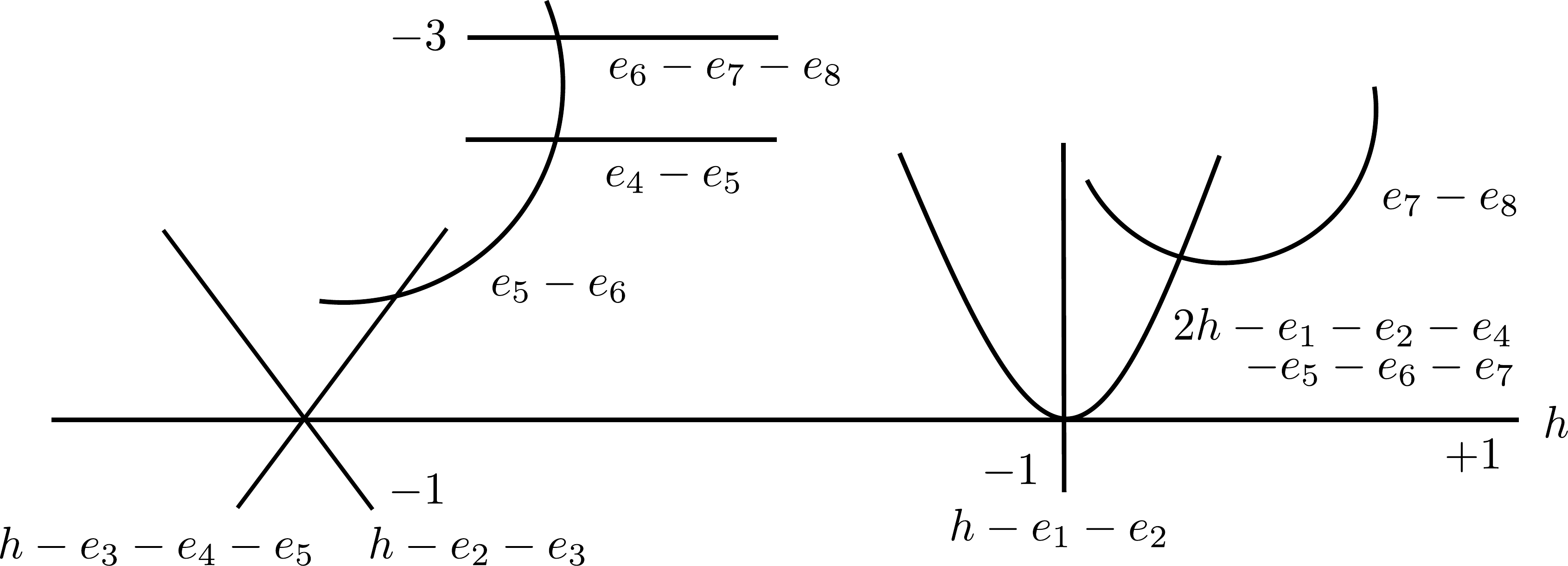}\\
	\caption{A resolution of a rational cuspidal septic of type $[[4,2,2,2],[3,3]]$ with the only possible homological embedding.}
	\label{fig:7McDuff8}
\end{figure}

\begin{prop}\label{p:septic[4,2,2],[3,3,2]}
If a rational cuspidal septic $C$ is of type $[[4,2,2],[3,3,2]]$, then the only relatively minimal symplectic embedding of $C$ is into $\CP$ and this embedding is unique up to symplectic isotopy.
\end{prop}

\begin{proof}
Blow up one more time than the minimal resolution at each cusp as in Figure~\ref{fig:7McDuff9}. The homology classes relative to the $+1$--line are uniquely determined, and use eight $e_i$ classes (the same number of exceptional divisors as in the resolution). Therefore $C$ embeds symplectically minimally only in $\CP$. Using Lemma~\ref{l:blowdown} to blow down the exceptional divisors in the $e_i$ classes, the configuration of curves descends to a configuration of five lines. This implies that $C$ has a unique isotopy class in $\CP$.
\end{proof}

\begin{figure}[h]
	\centering
		\includegraphics[scale=0.3]{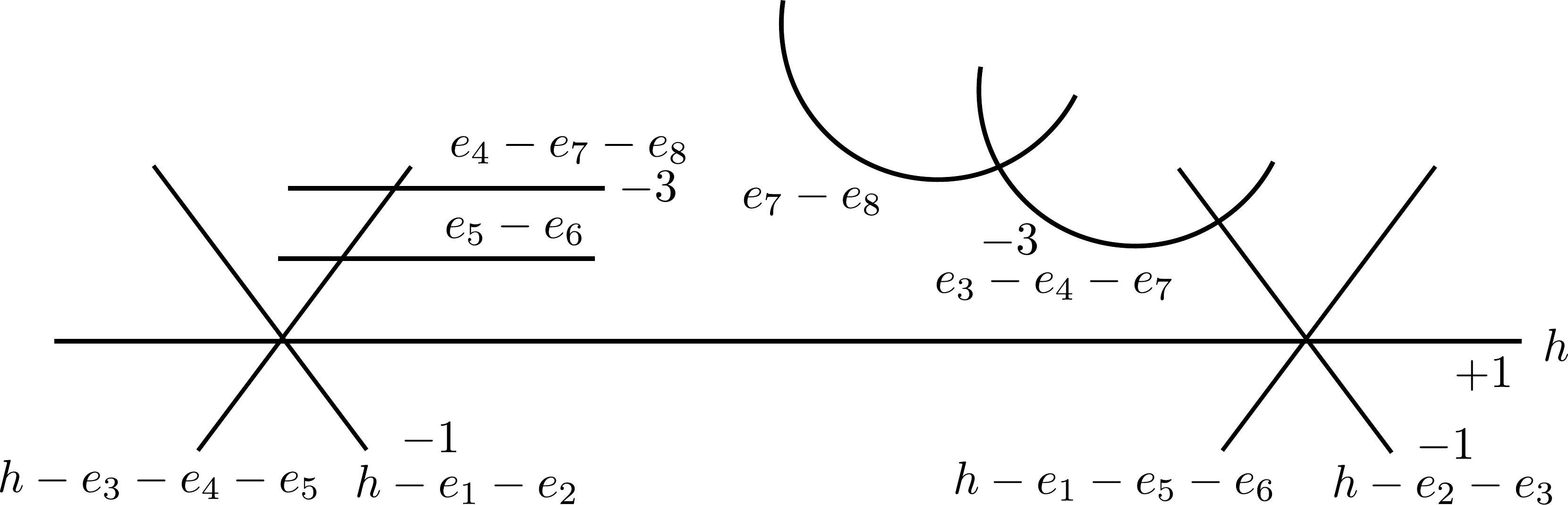}\\
	\caption{A resolution of a rational cuspidal septic of type $[[4,2,2],[3,3,2]]$ with the only possible homological embedding.}
	\label{fig:7McDuff9}
\end{figure}

For the next proposition, we introduce another auxiliary configuration and we show that it has a unique equisingular isotopy class. Let $\mathcal{SB}$ denote the stag beetle configuration consisting of a triangle of sides $L_1$, $L_2$ and $L_3$ with an inscribed conic, a line $L_4$ passing through the intersection between $L_1$ and $L_2$ and through the tangency of $L_3$ with the conic, a line $L_5$ passing through the other intersection $p$ of $L_4$ with the conic, and through the intersection between $L_2$ and $L_3$, a line $L_6$ passing through $p$ and through the tangency of $L_1$ with the conic and a line $L_7$ passing through the other intersection of $L_5$ with the conic and through the tangency of $L_1$ with the conic.

\begin{prop}\label{p:existencestagbeetle}
The stag beetle configuration $\mathcal{SB}$ has a unique equisingular symplectic isotopy class.
\end{prop}

\begin{figure}[h]
	\centering
		\includegraphics[scale=0.3]{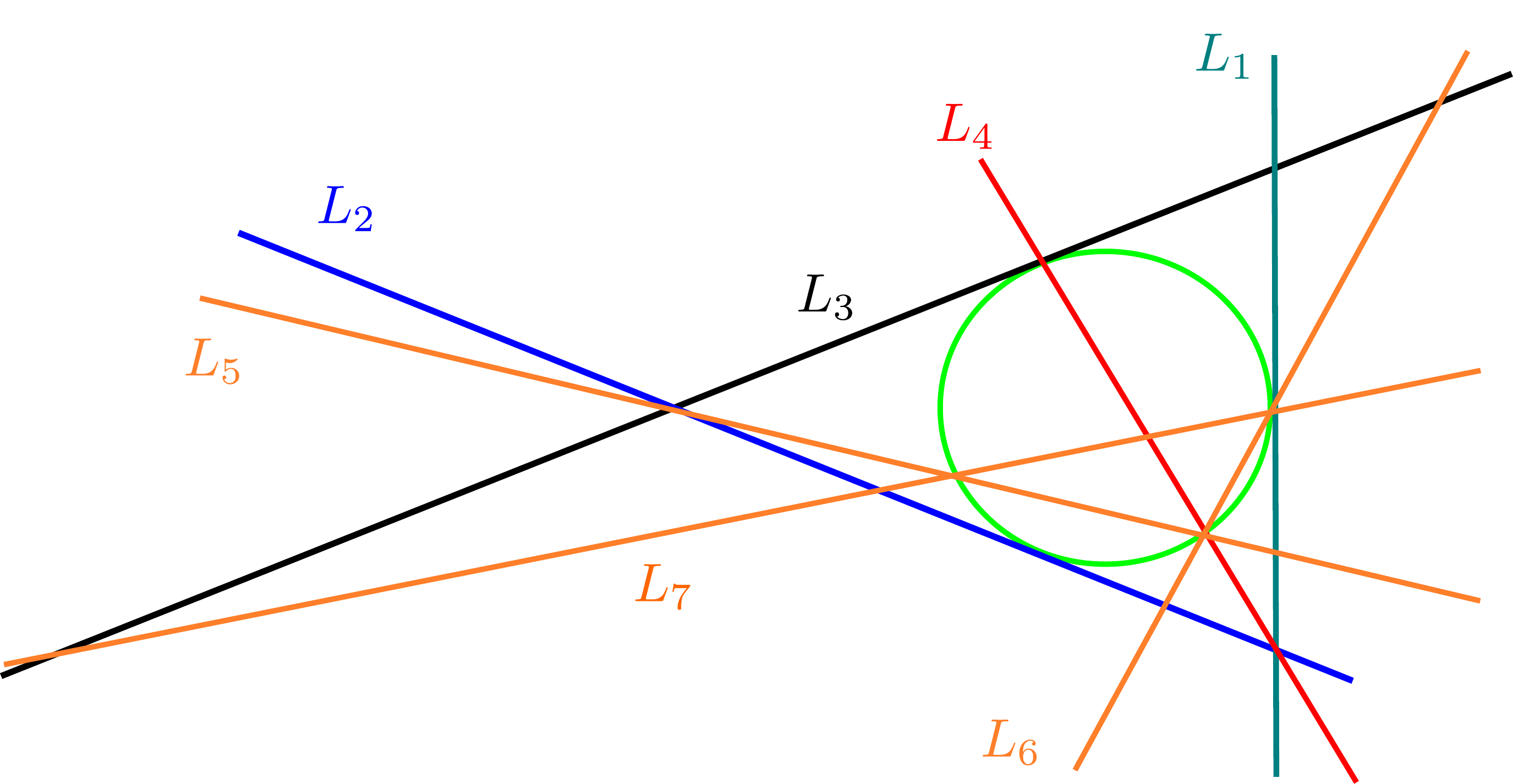}\\
	\caption{A stag beetle configuration $\mathcal{SB}$.}
	\label{fig:StagBeetle6}
\end{figure}

\begin{proof}
Start with a triangle with an inscribed conic and apply Theorem~\ref{t:addtheline}, using the same notations as the paragraph above, to successively add $L_4$, $L_5$, $L_6$ and $L_7$ (in this order). 
\end{proof}

\begin{prop}\label{p:septic[4,2,2],[3,3]}
If a rational cuspidal septic $C$ is of type $[[4,2,2],[3,3],[2]]$, then the only relatively minimal symplectic embedding of $C$ is into $\CP$ and this embedding is unique up to symplectic isotopy.
\end{prop}

\begin{proof}
Blow up one more time than the minimal resolution at the cusp of type $[4,2,2]$ and one more time than the minimal resolution at the cusp of type $[3,3]$ as in Figure~\ref{fig:7McDuff10}. The homology classes relative to the $+1$--line are uniquely determined, and use eight $e_i$ classes (the same number of exceptional divisors as in the resolution). Therefore $C$ embeds symplectically minimally only in $\CP$. Using Lemma~\ref{l:blowdown} to blow down the exceptional divisors in the $e_i$ classes, the configuration of curves descends to a configuration of four lines $\{ L_i \}$ and two conics $\{ Q_i \}$. The two conics intersect each other tangentially at a point with multiplicity $2$, $L_1$, $L_2$ and $L_3$ are concurrent, $L_1$ passes through the two transverse intersection points between $Q_1$ and $Q_2$, $L_2$ passes through the tangential intersection point between the two conics, $L_3$ is tangent to both of the conics, $L_4$ passes trough the intersection point between $L_3$ and $Q_1$, through one of the transverse intersection point between the two conics, and through the other intersection point between $Q_2$ and $L_2$, and the other intersections are generic. We next show that this configuration is birationally derived from a stag beetle configuration $\mathcal{SB}$.

\begin{figure}[htbp]
\centering
\subfloat[]{\label{fig:StagBeetlea}\includegraphics[width=0.4\linewidth]{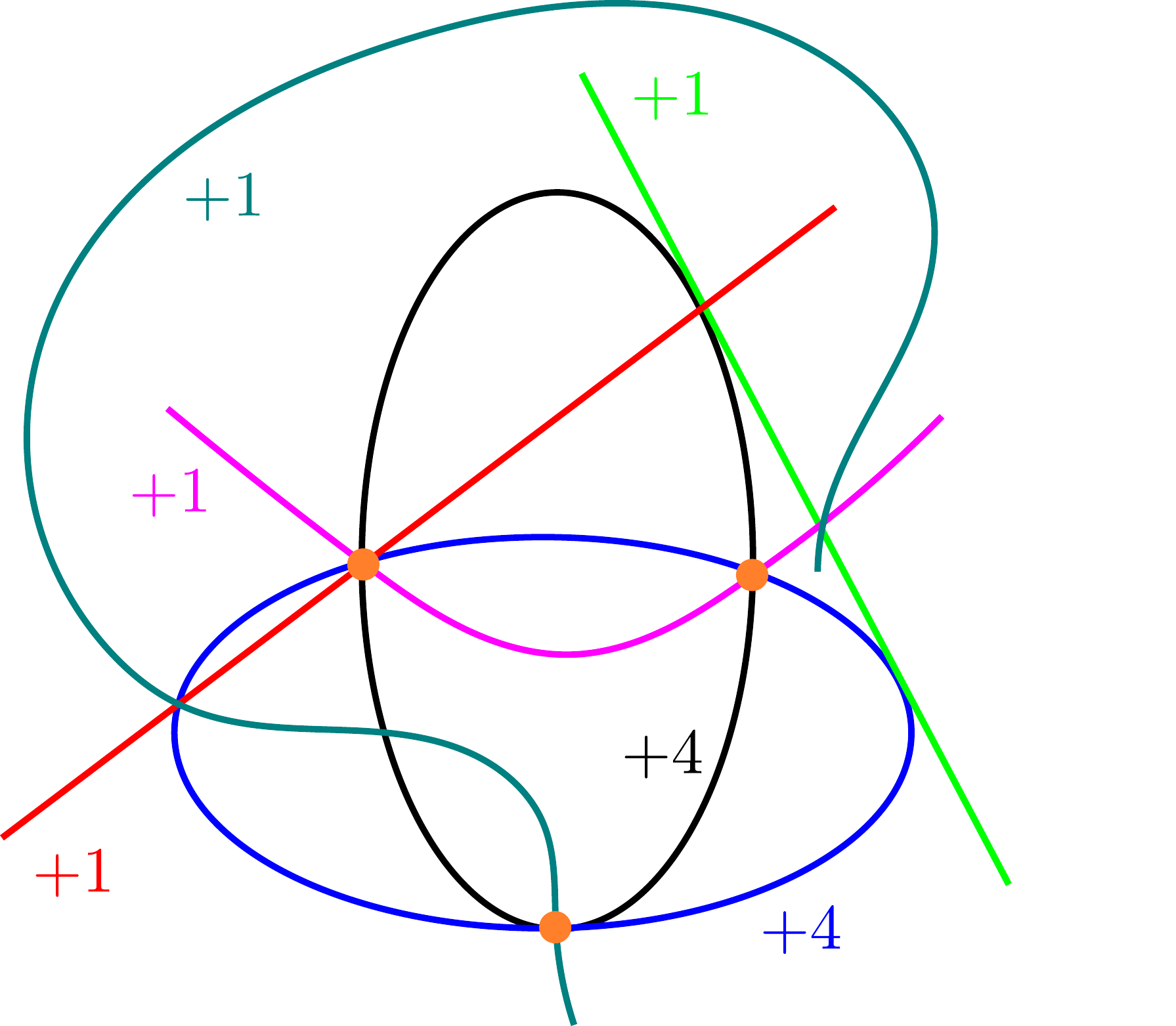}}\qquad
\subfloat[]{\label{fig:StagBeetleb}\includegraphics[width=0.4\linewidth]{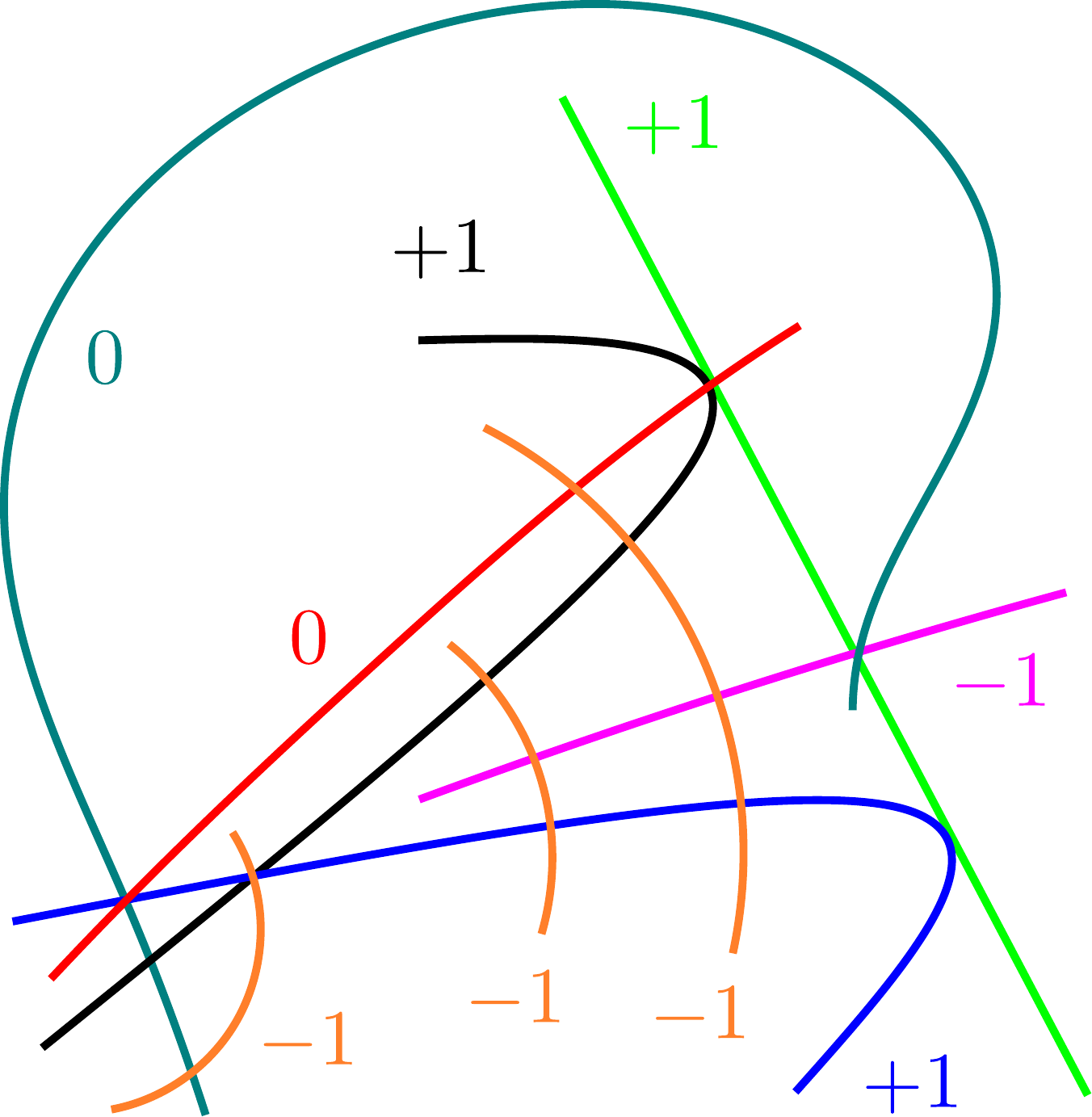}}\\
\subfloat[]{\label{fig:StagBeetlec}\includegraphics[width=0.65\textwidth]{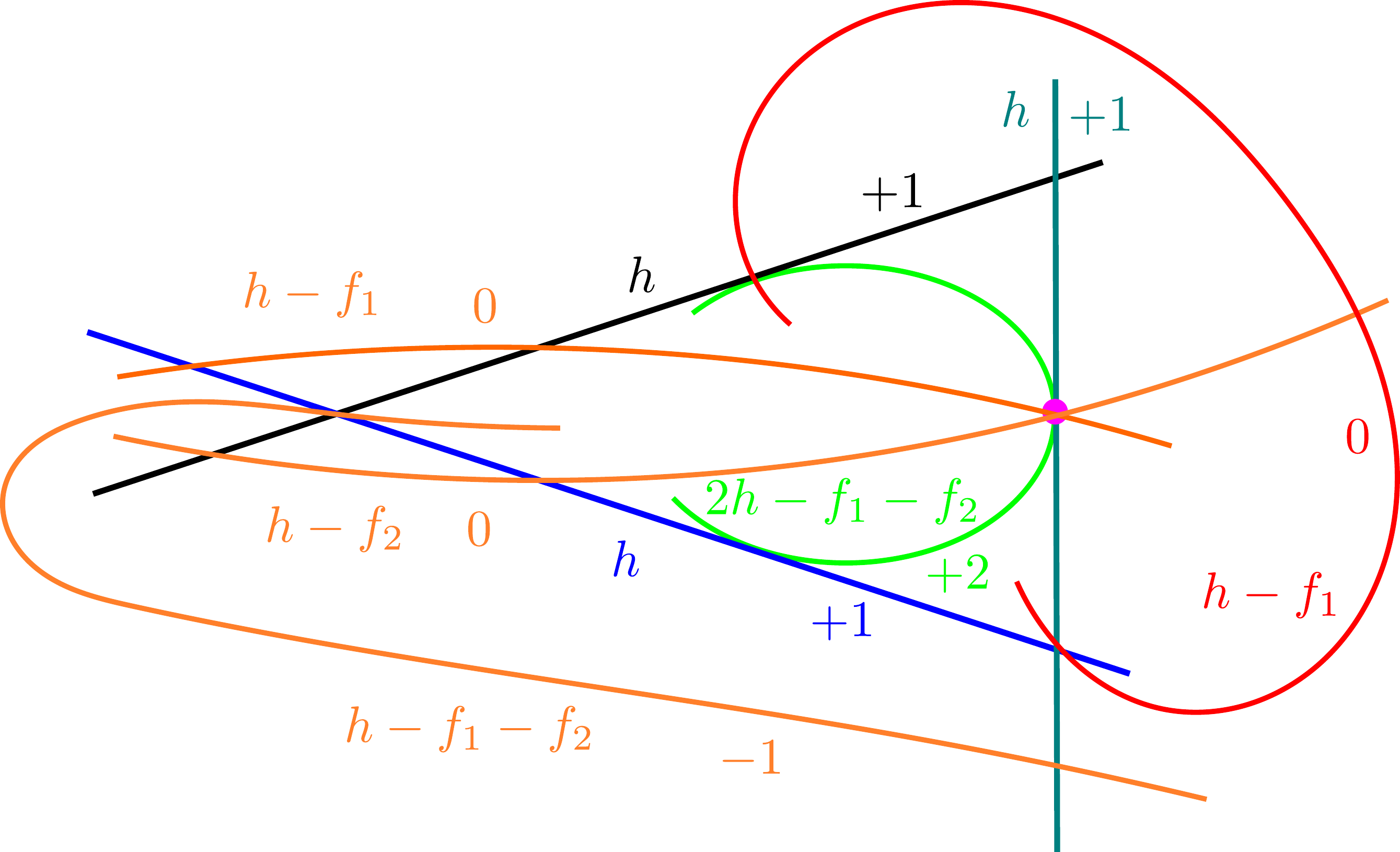}}\\
\subfloat[]{\label{fig:StagBeetled}\includegraphics[width=0.65\textwidth]{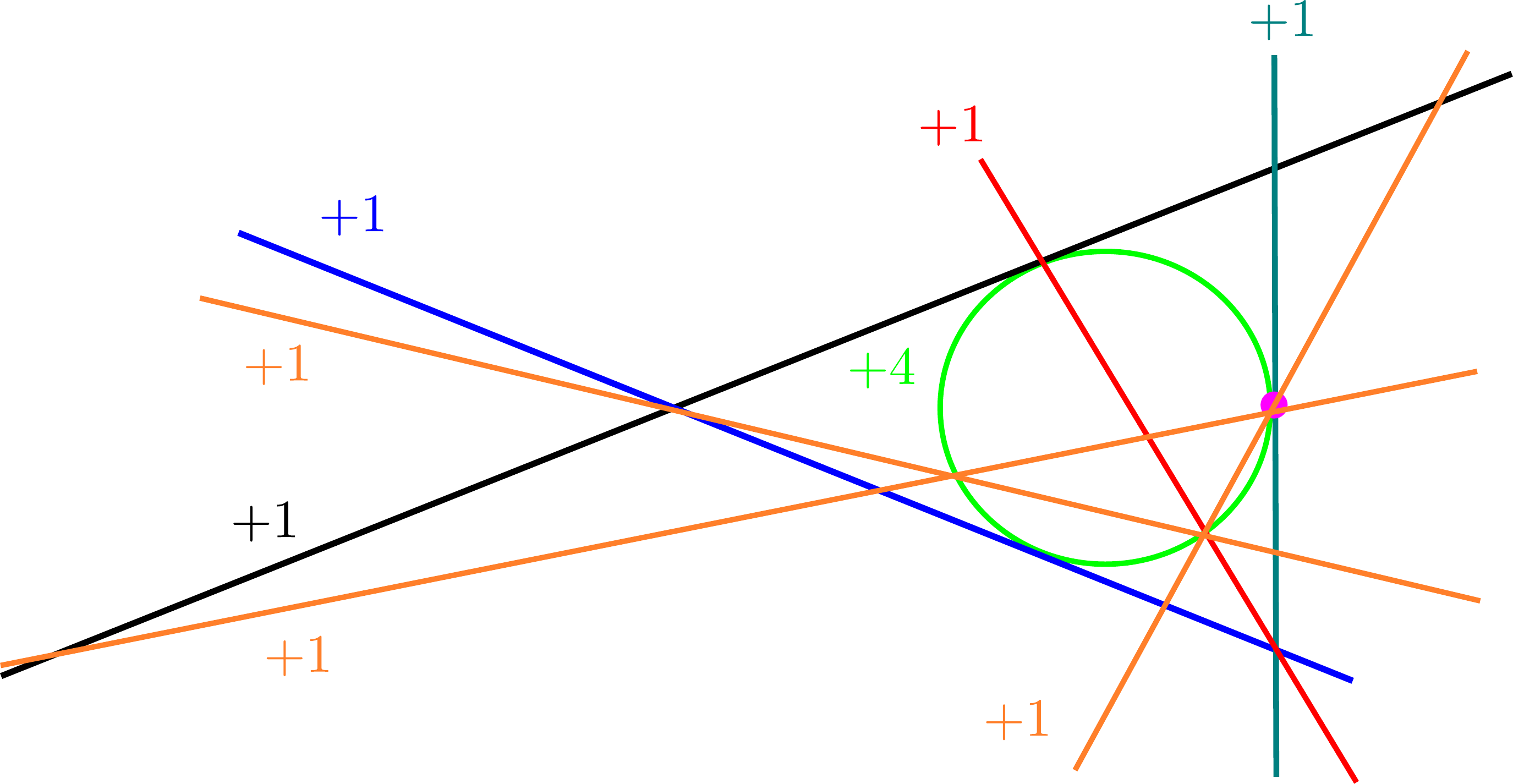}}%
\caption{Birational equivalence from the configuration mentionned in the paragraph above to a stag beetle configuration. The configurations (B) and (C) are the same, just redrawn indicating a symplectomorphism of $\CP \#3 \overline{\CP}$ identifying the black $+1$--curve with a line.}
\label{fig:StagBeetle2345}
\end{figure}

Blow up once at each of the three intersection points between the two conics and the other curves of the configuration (note that the exceptional curves associated to those blow-ups are actually $e_1$, $e_2$ and $e_3$ of Figure~\ref{fig:7McDuff10}), blow down the proper transform of $L_1$ (which is a $(-1)$--curve because $L_1$ passes through two of thoses points) then apply McDuff's Theorem to identify the proper transform of one of the conics to a line as in Figure~\ref{fig:StagBeetle2345} (the result is independent of the choice of the conic). The new homology classes relative to this $+1$--line are uniquely determined, and use two $f_i$ classes (classes of the new exceptional curves). We omit the details here, the argument is the same as for the ladybug case in Proposition~\ref{p:[5]}. Using Lemma~\ref{l:blowdown} to blow down the exceptional divisors in the $f_i$ classes, the configuration of curves descends to a stag beetle configuration. See Figure~\ref{fig:StagBeetle2345}.

Since this curve configuration is birationally derived from $C$, Proposition~\ref{p:birationalequivalence} implies that $C$ has a unique isotopy class in $\CP$.
\end{proof}

\begin{figure}[h]
	\centering
		\includegraphics[scale=0.3]{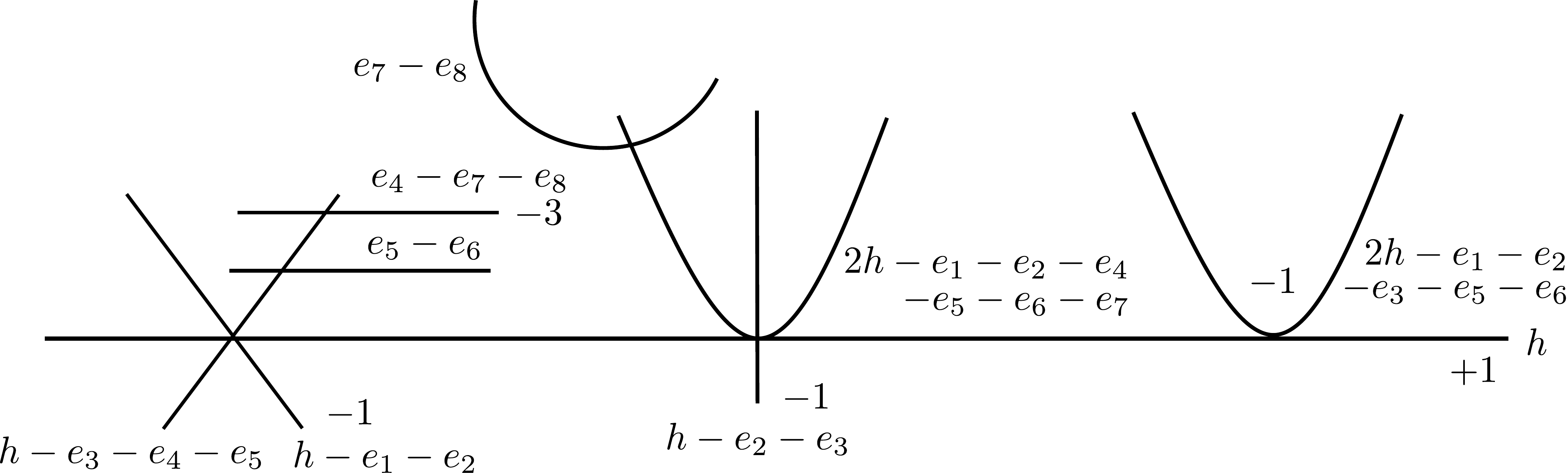}\\
	\caption{A resolution of a rational cuspidal septic of type $[[4,2,2],[3,3],[2]]$ with the only possible homological embedding.}
	\label{fig:7McDuff10}
\end{figure}

\begin{prop}\label{p:septic[3,3,3,3]}
If a rational cuspidal spetic $C$ is of type $[[3,3,3,3],[2,2,2]]$, then the only relatively minimal symplectic embedding of $C$ is into $\CP$ and
this embedding is unique up to symplectic isotopy.
\end{prop}

\begin{proof}
Perform the minimal resolution of both singularities. The homology classes relative to the $+1$--line are uniquely determined, and use seven $e_i$ classes (the same number of exceptional divisors as in the resolution), see Figure~\ref{fig:7McDuff37}. Therefore $C$ embeds symplectically minimally only in $\CP$. Use Lemma~\ref{l:blowdown} to blow down the exceptional divisor in the $e_4$ class, then the one in the $e_3$ class. Apply McDuff's Theorem once again on the $+1$--curve in the class $2h-e_1-e_2-e_5$. The new homology classes are uniquely determined, and use five $f_i$ classes (classes of the new exceptional curves), see Figure~\ref{fig:7McDuff37bis}. Using Lemma~\ref{l:blowdown} again to successively blow down the exceptional divisors in the $f_i$ classes, the configuration of curves descends to a $\mathcal{G}_3$ configuration with two additional lines (one tangent to one of the conic at the transverse point of intersection between the two conics, and the other one passing through this same point and through one of the tangent intersection between one of the conics and the line of the $\mathcal{G}_3$ configuration). This has a unique symplectic isotopy class by Proposition~\ref{p:existenceG3} and Theorem~\ref{t:addtheline}. Since this curve configuration is birationally derived from $C$, Proposition~\ref{p:birationalequivalence} implies that $C$ has a unique isotopy class in $\CP$.

\begin{figure}[h]
	\centering
		\includegraphics[scale=0.3]{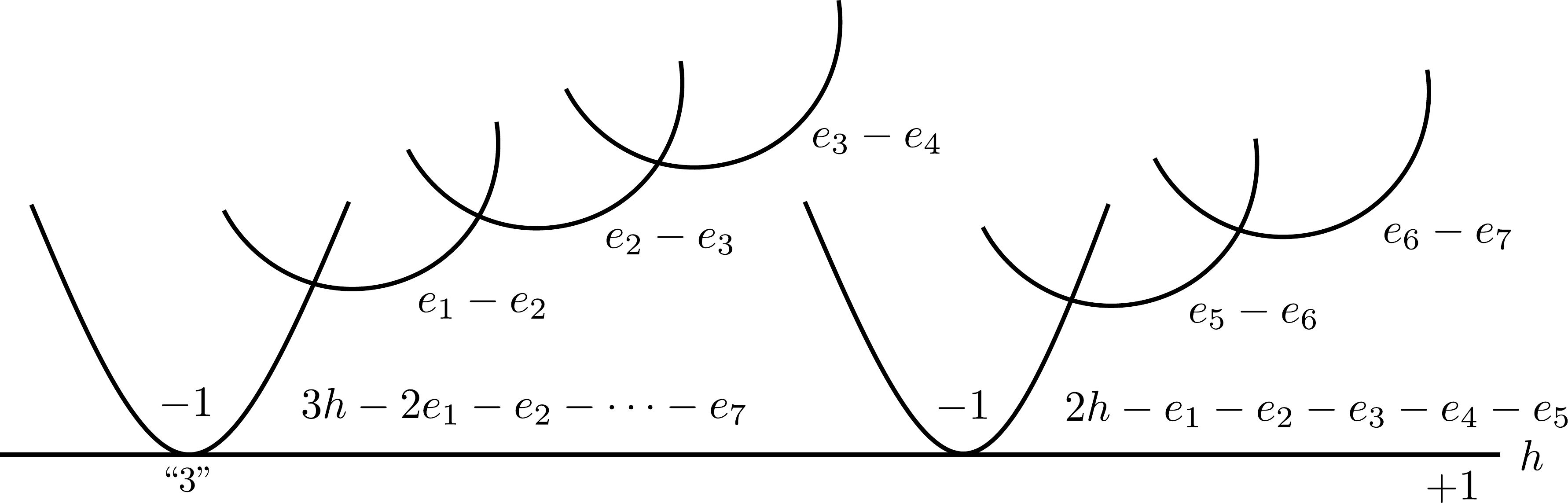}\\
	\caption{A resolution of a rational cuspidal septic of type $[[3,3,3,3],[2,2,2]]$ with the only possible homological embedding.}
	\label{fig:7McDuff37}
\end{figure}
\end{proof}

\begin{figure}[h]
	\centering
		\includegraphics[scale=0.3]{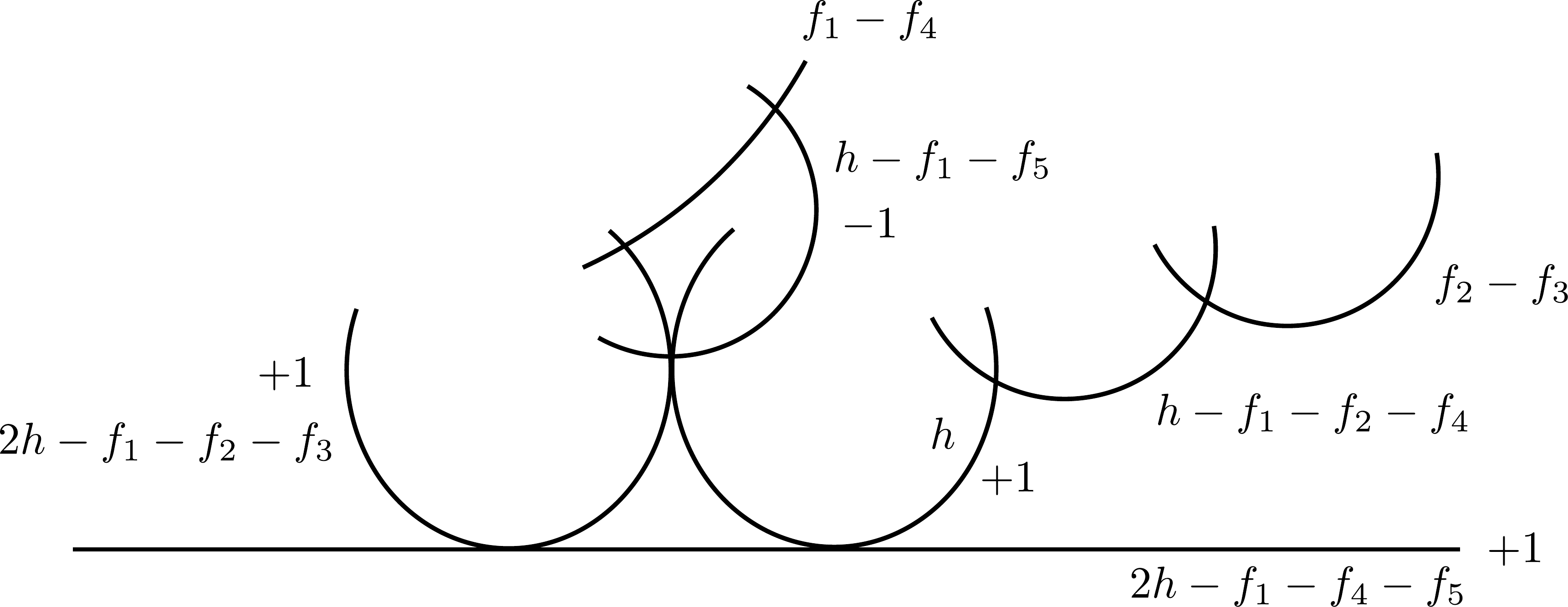}\\
	\caption{A configuration birationally derived from a rational cuspidal septic of type $[[3,3,3,3],[2,2,2]]$, with the only possible homological embedding.}
	\label{fig:7McDuff37bis}
\end{figure}

\subsection{Proof of Theorem~\ref{t:existence}}

We finally summarize the proof of Theorem~\ref{t:existence}.

\begin{proof}[Proof of Theorem~\ref{t:existence}]
The rational cuspidal sextics of type $[[5]]$, $[[3,3,3,2]]$ and the rational cuspidal septic of type $[[6]]$ are unicuspidal rational symplectic curves whose only cusp is the cone on a torus knot. Therefore by~\cite[Section~6.4]{GS} each of these curves is symplectically isotopic to a complex curve. Since the cusps $[5]$ and $[6]$ are of type $(p, p+1)$, the only relatively minimal symplectic embeddings of the curves of type $[[5]]$ and $[[6]]$ are into $\CP$, and  both of those embeddings are unique up to symplectic isotopy (by~\cite{GS}). The cusp $[3,3,3,2]$ is of type $(p, 4p-1)$, so there are exactly two relatively minimal symplectic embeddings of the curve of type $[[3,3,3,2]]$ into closed symplectic $4$--manifolds. One
into $\CP$ and another into $S^2 \times S^2$, each unique up to symplectomorphism.

The remaining cases are treated in Subsections~\ref{s:existencesextics} and~\ref{s:existenceseptics}.
\end{proof}

\begin{center}
\begin{tabular}{l|l||l|l}
\multicolumn{2}{c||}{Degree $6$} & \multicolumn{2}{c}{Degree $7$}\\[2pt]
Cusps (MS) & Proposition &	Cusps (MS) & Proposition\\[2pt]
\hline
& & &\\[-10pt]
$[4,2,2,2,2]$ & \ref{p:sextic[4,2,2,..]} &		$[5,2,2],[2,2,2]$ &	\ref{p:septic[5,2,2]} \\
$[4,2,2,2], [2]$ &	\ref{p:sextic[4,2,2,..]} &	$[5],[2,2,2,2,2]$ &	\ref{p:[5]}\\
$[4,2,2], [2,2]$ &	\ref{p:sextic[4,2,2,..]} &	$[5],[2,2,2,2],[2]$ &	\ref{p:[5]}\\
$[4], [2,2,2,2]$ & 	\ref{p:sextic[4]} &		$[5],[2,2,2],[2,2]$&	\ref{p:[5]}\\
$[4], [2,2,2],[2]$ &	\ref{p:sextic[4]} &	$[4,3],[3,3]$ &	\ref{p:septic[4,3]}\\
$[4], [2,2],[2,2]$ & 	\ref{p:sextic[4]} &	$[4],[3,3,3]$ &	\ref{p:(4,5),(3,10)} \\
$[3,3,3], [2]$ & 	\ref{p:sextic[3,3,..]} &  $[4,2,2,2],[3,3]$ & 	\ref{p:septic[4,2,2,2]}\\
$[3,3,2], [3]$ & 	\ref{p:sextic[3,3,..]} &	$[4,2,2],[3,3,2]$ &	\ref{p:septic[4,2,2],[3,3,2]}\\
$[3,3], [3,2]$ & 	\ref{p:sextic[3,3,..]} &	$[4,2,2],[3,3],[2]$ &	\ref{p:septic[4,2,2],[3,3]}\\
				 & 						 & 		$[3,3,3,3],[2,2,2]$ &	\ref{p:septic[3,3,3,3]}\\
\end{tabular}
\end{center}

\section{Obstructions}\label{s:obstruction}

The goal of this section is to obstruct the existence of all symplectic rational cuspidal curves for which we have not proved existence.

\begin{thm}\label{t:obstruction}
If $C$ is a symplectic rational cuspidal curve of degree $6$ or $7$, then it is equisingular to one appearing in Theorem~\ref{t:existence}.
\end{thm}

We can now also prove Proposition~\ref{p:cremona}, stating that every rational cuspidal curve of degree up to $7$ is Cremona equivalent to a line.

\begin{proof}[Proof of Proposition~\ref{p:cremona}]
If $C$ is a rational cuspidal curve of degree up to $5$, the result follows from~\cite[Theorem~1.3]{GS}.
For degrees $6$ or $7$, Theorem~\ref{t:obstruction} tells us that $C$ is one of the curves of Theorem~\ref{t:existence}. For each of these curves, we found a sequence of blow-ups such that the proper transform of $C$ is a line in a blow-up of $\CP$. This gives the required birational transformation.
\end{proof}

Recall that, if $C$ is a symplectic rational cuspidal curve in $\CP$, then $C$ satisfies the singular adjunction formula: $\sum_{p\in C} \mu_{(C,p)} = (d-1)(d-2)$, where $\mu_{(C,p)}$ is the Milnor number of the singularity of $C$ at $p$. (The sum is finite since $\mu_{(C,p)} = 0$ whenever $p$ is a non-singular point.) From a symplectic perspective, the formula arises from the non-singular adjunction formula (also known as the degree-genus formula) by replacing a neighbourhood of each singularity with its Milnor fibre; this is the symplectic analogue of replacing an algebraic curve $V(f)$ with $V(f+\epsilon g)$, where $g$ is a generic polynomial of the same degree as $f$ and $0<|\epsilon|\ll 1$ is small.

Let $p_1$ be a singular point of $C$. Comparing the proper transform of $C$ in the resolution of $(C,p_1)$ and the Milnor fibre of $(C,p_1)$, one sees that, if $[m^1_1,\dots,m^1_\ell]$ is the multiplicity sequence of $(C,p_1)$, then $\mu_{(C,p_1)} = \sum m^1_i(m^1_i-1)$, so that the adjunction formula reads as
\begin{equation}\label{e:adjunction}
\sum_i m^C_i(m^C_i-1) = (d-1)(d-2),
\end{equation}
where $[[m^C_1,\dots,m^C_N]]$ is the multiplicity multi-sequence of the curve $C$.

We distinguish three distinct levels of obstructions: a smooth level, a birational level, and a mixed level.

Recall that if $C \subset \CP$ is a rational cuspidal curve of degree $d$, whose singularities have links $K_1,\dots,K_\nu$, then the $X_{d^2}(K)$ smoothly embeds in $\CP$, where $K = K_1\#\dots\# K_\nu$ and $X_m(K)$ is the trace of $m$--surgery along $K$, i.e. the $4$--manifold obtained by attaching an $m$--framed $2$--handle to $B^4$ along $K \subset S^3 = \partial B^4$. (See, for instance,~\cite[Section~3.1]{BorodzikLivingston}.)
Conversely, whenever we have a trace embedding as above, we say that $C$ is a \emph{PL sphere}; as mentioned in the introduction, we think of $X_m(K)$ as a regular neighbourhood of a piecewise-linearly embedded $2$--sphere in a $4$--manifold with Euler number $m$. If $C \subset X_{d^2}(K_1\# \dots \# K_\nu)$ is a PL sphere in $\CP$ we call $d$ its degree; if $K_i$ is an algebraic knot (that is, the link of an irreducible singularity) for each $i$ and $2\sum_i g(K_i) = d(d-1)$, we also say that $C$ is an \emph{adjunctive PL sphere}. We extend the notation and terminology for symplectic curves to PL spheres: they have a \emph{degree}, determined by their homology class, and a \emph{type} and a \emph{multiplicity multisequence}, keeping track of their singularities.
Note that, for each $d$, there are finitely many possible configurations of singularities on a degree-$d$ adjunctive PL sphere.
When we say that a curve is \emph{smoothly obstructed}, we mean that the corresponding surgery trace does not embed in $\CP$.
In Section~\ref{ss:smooth} below we will look at smooth obstructions, like the ones coming from Heegaard Floer homology and from branched covers.

We say that a symplectic curve is \emph{obstructed by a birational transformation} if there is a birational transformation (that is, a sequence of blow-ups and blow-downs; see Section~\ref{s:recap}) to a configuration of symplectic curves that cannot be realised in $\CP$. For instance, in some cases we find a configuration that reduces to a configuration of conics and lines whose existence was ruled out in~\cite{GS}. We give more examples of obstructed configurations in Section~\ref{ss:auxiliary} below, and we then focus on sextics and septics obstructed by birational transformations in Sections~\ref{ss:birational6} and~\ref{ss:birational7}.

Finally, there are some \emph{mixed obstructions}, where we use a smooth obstruction together with positivity of intersections, which we will look at in Section~\ref{ss:mixed}. Here we use auxiliary $J$--holomorphic lines to create a configuration that we can then rule out topologically. For instance, the Riemann--Hurwitz obstruction and the Levine--Tristram signature obstruction falls into this category, as it uses positivity of intersections with a generic line. We view this third obstruction as something intermediate between the birational and the smooth ones described above.

In Section~\ref{ss:collectobstructions} we collate all obstructions to prove Theorem~\ref{t:obstruction}.

\subsection{Smooth obstructions}\label{ss:smooth}

The first obstruction we employ comes from Heegaard Floer homology~\cite{OSz-HF, OSz-HFPA}, and more specifically from correction terms~\cite{OSz-absolutely}; it was discovered by Borodzik and Livingston~\cite[Theorem~6.5]{BorodzikLivingston}, and it is best phrased in terms of the semigroup-counting function associated to the singularities.
Let $(C,p)$ be a singular point; the semigroup $\Gamma_{(C,p)} \subset \Z_{\ge 0}$ keeps track of the multiplicities of intersection of divisors with $C$ at $p$~\cite[Section~4.3]{Wall}; $\Gamma_{(C,p)}$ is equivalently encoded in the function $R_{(C,p)} \colon \Z \to \Z$ which counts elements of $\Gamma_{(C,p)}$ in intervals: $R_{(C,p)}(n) = \# (\Gamma_{(C,p)} \cap [0,n))$. (Note that $R_{(C,p)}(n) = 0$ whenever $n\le 0$.)

We define the function $R_C$ associated to a PL sphere $C$ as the infimum convolution of the $R$--functions of its singularities: if $C$ has singular points $p_1,\dots,p_\nu$, we let
\[
R_C := \diamond_{i=1}^\nu R_{(C,p_i)} \colon n \mapsto \min_{k_1+\dots+k_\nu = n}\{R_{(C,p_1)}(k_1) + \dots + R_{(C,p_\nu)}(k_\nu)\}
\]

\begin{thm}[\cite{BorodzikLivingston}]\label{t:semigroupobstruction}
If $C$ is an adjunctive PL sphere of degree $d$ in $\CP$, then $R_C(jd+1) = \frac{(j+1)(j+2)}2$ for each $j = -1, \dots, d-2$.
\end{thm}

We draw the following corollary.

\begin{prop}\label{p:multisequences}
The only possible multiplicity multisequences for an adjunctive PL sphere of degree $6$ in $\CP$ are:
\[
[[5]], [[4,2^{[4]}]], [[3,3,3,2]], [[3,3,2^{[4]}]], [[3,2^{[7]}]].
\]
The only possible multiplicity multisequences for an adjunctive PL sphere of degree $7$ in $\CP$ are:
\[
[[6]], [[5,2^{[5]}]], [[4,3^{[3]}]], [[4,3,3,2^{[3]}]], [[4,3,2^{[6]}]], [[3^{[4]},2^{[3]}]], [[3^{[3]}, 2^{[6]}]].
\]
\end{prop}

\begin{proof}
Results of Bodn\'ar and N\'emethi~\cite[Theorem~5.1.3]{BodnarNemethi} and Borodzik and Hedden~\cite[Section~5]{BorodzikHedden} now show that, in fact, $R_C$ depends only on the multiplicity multi-sequence; in particular, we can compute it for all curves whose singularities are of type $(m_i^C, m_i^C+1)$ (i.e. with multiplicity sequence $[m_i^C]$) to obstruct all curves with the same multiplicity multi-sequence.

We do one sample computation. Consider the multiplicity multisequence $[[2^{[10]}]]$ for a sextic. In this case, since  $[2^{[10]}]$ is the multiplicity sequence of $T(2,21)$ and its semigroup is $\{0,2,4,\dots,18,20,21,22,23,\dots\}$, we have:
\[
R_C(n) = R_{T(2,21)}(n) = \left\{
\begin{array}{ll}
\lceil {\textstyle \frac n2} \rceil & \textrm{if $n \le 20$,}\\
n-10 & \textrm{if $n > 20$.}
\end{array}\right.
\]
Therefore, $R_C(1\cdot 6 + 1) = R_C(7) = 4 \neq 3 = \frac{(1+1)(1+2)}2$, which contradicts Theorem~\ref{t:semigroupobstruction} for $j=1$.

We used a computer to carry out the calculation of the $R$--function for each multiplicity multisequence satisfying the singular adjunction formula; each multisequence not appearing in the statements is obstructed by Theorem~\ref{t:semigroupobstruction}.
\end{proof}

\begin{rmk}
It turns out that, if we only care about symplectic or complex curves, then all multiplicity multisequences not appearing in the proposition are ruled out by B\'ezout's theorem: indeed, a degree-$j$ curve can intersect $C$ at its singularity with multiplicity at most $jd$. This, in turn, translates into the inequality $R_C(jd+1) \ge \frac{(j+1)(j+2)}{2}$. (See~\cite{Borodzik-Bezout}.) We point out that B\'ezout's theorem holds for complex and symplectic curves, and thus gives a symplectic obstruction; Heegaard Floer homology, however, gives a trace embedding obstruction, which is strictly stronger.
\end{rmk}

When the degree is odd, we can also obtain obstructions by taking into account spin structures. Indeed if $C$ is a PL sphere of odd degree $d$ in $\CP$, then $W = \CP\setminus N(C)$ has a spin structure (see, for instance,~\cite[Section~4]{BorodzikHom}. Since $W$ is a rational homology ball, it has signature $0$; it follows that the Rokhlin invariant $\mu(\partial N(C)) \equiv 0 \pmod{16}$. (Here we think of the Rokhlin invariant as an element in $\Z/16\Z$.) Note that, since $d$ is odd, $\partial(N(C))$ has a unique spin structure.

Since $N(C)$ is the trace of $d^2$--surgery along a knot $K$, the Rokhlin invariant of its boundary is determined by $d$ and by the \emph{Arf invariant} $\Arf(K)$ of $K$. Recall that the Arf invariant, also called Arf--Robertello invariant, is a knot invariant which takes values in $\Z/2\Z$~\cite{Robertello}; it determines the type of the $\Z/2\Z$--quadratic form associated to any Seifert form of the knot. It can be computed from the Alexander polynomial: $\Arf(K) \equiv 0$ if $\Delta_K(-1) \equiv \pm1 \pmod8$, and $\Arf(K) \equiv 1$ if $\Delta_K(-1)\equiv \pm 3 \pmod8$. In the table below, we list the knots (identified by their topological cabling parametres) of Table~\ref{t:dictionary} according to their Arf invariant.

\begin{center}
\begin{tabular}{ll|ll}
\multicolumn{2}{c|}{$\Arf = 0$} & \multicolumn{2}{c}{$\Arf = 1$}\\
\hline
& & &\\[-10pt]
$(6,7)$ & \\
$(5,7)$ & & $(5,6)$\\
$(4,7)$ & & $(4,5)$\\
$(2,3; 2,2k+11)$ & if $k \equiv 2,3 \pmod 4$ & $(2,3; 2,2k+11)$ & if $k \equiv 0,1 \pmod 4$\\
$(3,6k\pm 1)$ & & $(3,6k\pm 2)$\\
$(2,2k+1)$ & if $k \equiv 0,3 \pmod 4$ & $(2,2k+1)$ & if $k \equiv 1,2 \pmod 4$
\end{tabular}
\end{center}

Gordon proved in~\cite{Gordon-Arf} that $\mu(S^3_n(K)) \equiv n-1 + 8\Arf(K) \pmod{16}$.
In the case of embedded PL spheres, combining Gordon's Rokhlin invariant computation with the observation that $\mu = 0$, we obtain the following obstruction, which is both easily computable and surprisingly strong.

\begin{prop}\label{p:arf}
If $C$ is a PL sphere of odd degree $d$ with singularities $K_1, \dots, K_\nu$, then $\sum_i \Arf(K_i) \equiv \frac{d^2-1}8 \pmod 2$. In particular, if $d = 7$, $\sum_i \Arf(K_i) \equiv 0 \pmod 2$.\qedhere
\end{prop}

\begin{rmk}\label{r:IHF}
There is a corresponding refinement of Theorem~\ref{t:semigroupobstruction} for curves of odd degree, due to Borodzik and Hom~\cite{BorodzikHom}, which takes into account the spin structure on the complement of the PL sphere. Instead of working with ``ordinary'' Heegaard Floer homology, they work with \emph{involutive} Heegaard Floer homology, a theory developed by Hendricks and Manolescu~\cite{HendricksManolescu}.
We note here that the `grouping of multiplicity' argument, which was crucial in the proof of Proposition~\ref{p:multisequences}, does \emph{not} work in the involutive setup.
Borodzik and Hom computed the obstruction for septics with two cusps~\cite[Section~5.2]{BorodzikHom}: their criterion obstructs the existence of PL rational cuspidal curves of types $[[4,3],[2^{[6]}]]$, $[[4,2^{[5]}],[3,2]]$, and $[[3^{[3]},2], [2^{[5]}]]$. As it turns out, all these curves are already obstructed by Proposition~\ref{p:arf}, but the results of~\cite{BorodzikHom} are potentially more powerful for curves with more cusps (or in other degrees).
\end{rmk}

Branched covers provide another powerful source of obstructions.
Before stating the next proposition, we set up some notation.
Let $(C,p) \cong (V(f(x,y)),0) \subset (\C^2,0)$ be a complex curve germ; we denote by $M_m(C,p)$ the Milnor fibre of the $m$--suspension of $(V(f(x,y)),0)$, i.e. the Milnor fibre of the singularity of $f(x,y) + z^m$ at the origin in $\C^3$. It is well-known that $M_m(C,p)$ retracts onto a bouquet of $2$--spheres, and in particular it is simply-connected~\cite{Milnor-hypersurface}. Note that if $(C,p)$ is non-singular, then the hypersurface $\{f(x,y) + z^m = 0\}$ is smooth at the origin, and therefore its Milnor fibre is a $4$--ball.
We want to give a more topological interpretation of $M_m(C,p)$: the projection of $M_m(C,p)$ onto the $xy$--plane in $\C^2$ exhibits $M_m(C,p)$ as an $m$--fold cover of $B^4$ branched over the Milnor fibre of $f(x,y)$ at $0$.
For convenience, denote by $X_{d,m}$ the $m$--fold cover of $\CP$ branched over a non-singular complex curve of degree $d$.

Given a $4$--manifold $X$, possibly with boundary, we denote with $b_2^-(X)$ (respectively, $b_2^+(X)$) the maximal dimension of a subspace of $H_2(X)$ that is negative definite (resp. positive definite) with respect to the intersection pairing on $X$.
The \emph{signature} of $X$ is the difference $\sigma(X) := b_2^+(X) - b_2^-(X)$.

\begin{thm}\label{p:signatureobstruction}
Let $C$ be an adjunctive PL sphere of degree $d$ in $\CP$. Then, for each prime power $m$ dividing $d$,
\begin{equation}\label{e:b2minus}
\sum_{p\in C} b_2^-(M_m(C,p)) \le b_2^-(X_{d,m}).
\end{equation}
\end{thm}

Note that the sum on the left-hand side of the inequality is finite, since $b_2^-(M_m(C,p))$ vanishes whenever $p$ is a non-singular point of $C$.

\begin{proof}
Let $p_1, \dots, p_\nu$ be the singular points of $C$, and $\{U_1,\dots,U_\nu\}$ a collection of small, pairwise disjoint closed $4$--balls, with $U_i$ centred at $p_i$ for each $i = 1, \dots, \nu$.
For each $i$, we can replace $(U_i,C)$ with the Milnor fibre pair $(B^4,M(C,p_i))$ of the singularity of $C$ at $p_i$, thus obtaining a non-singular surface $C'$ of genus $(d-1)(d-2)/2$, in the same homology class as $C$.

We now take the $m$--fold cyclic cover of $\CP$ branched over $C'$, obtaining a smooth closed $4$--manifold $X$, together with a surface $R \subset X$ and a map $\pi: X \to \CP$ that is $m$ to $1$ on $X\setminus R$ and $1$ to $1$ from $R$ onto $C'$. We want to show that the Betti numbers and the signature of $X$ are the same as those of $X_{d,m}$.

Since $C'$ has the same genus as a non-singular degree-$d$ complex curve, the (topological) Euler characteristics of $X$ is the same as the Euler characteristics of $X_{m,d}$. Since, additionally, $C'$ has the same self-intersection as a degree-$d$ complex curve, and since the signature of $X$ is determined by these data via the $G$--signature theorem, $X$ has the same signature as $X_{d,m}$ (see below for more details).


We now want to argue that $b_1(X) = b_3(X) = 0$. By Poincar\'e duality, it suffices to prove that $b_1(X)$ vanishes. To do so, we view $X$ as the union of two pieces, $W$ and $Z$, defined as follows. Consider disjoint paths $\gamma_1, \dots, \gamma_{\nu-1} \subset C \setminus (\cup_i {\rm Int}(U_i))$ connecting $\partial U_\nu$ to $\partial U_1, \dots, \partial U_{\nu-1}$, respectively. Let $U$ be a small regular neighbourhood of $\cup_i U_i \cup \cup_j \gamma_j$ in $\CP$ whose boundary intersects $C$ transversely. Note that $U$ is diffeomorphic to a $4$--ball.
Let $W$ be the $m$--fold cover of $U$, branched over $C' \cap U$. It is easy to see that $W$ is the boundary connected sum of the $M_m(C,p_i)$, and in particular $H_1(W) = \bigoplus_i H_1(M_m(C,p_i)) = 0$.
Let $Z$ be the $m$--fold cover of $\CP \setminus U$ branched over $C' \setminus U$. Note that $C'\setminus U = C\setminus U$ is a $2$--disc and that $H_1(\CP \setminus (U \cup C'))$ is finite. Since $m$ is a prime power and $C' \setminus U$ is a disc, it follows from~\cite[Section~1]{Gilmer} (see also~\cite[Section~3]{RubermanStarkston}) that $H_1(Z)$ is torsion.

We look at the Mayer--Vietoris long exact sequence for reduced homology corresponding to this decomposition:
\[
\dots \longrightarrow \widetilde H_1(W) \oplus \widetilde H_1(Z) \stackrel{\alpha}{\longrightarrow} \widetilde H_1(X) {\longrightarrow} \widetilde H_0(\partial W) \longrightarrow \dots
\]
Since $\partial W$ is connected, $\widetilde H_0(W) = 0$, and thus $\alpha$ is onto. Since $\widetilde H_1(W) = 0$ and $\widetilde H_1(Z)$ is finite, then $H_1(X) \cong \widetilde H_1(X)$, too, is finite, and in particular $b_1(X) = 0$.

Since $b_1(X_{d,m}) = 0 = b_1(X)$, $b_3(X_{d,m}) = 0 = b_3(X)$, and $\chi(X) = \chi(X_{d,m})$, we also have $b_2(X) = b_2(X_{d,m})$, as claimed.

As $X$ and $X_{d,m}$ have the same signature, we have $b^-_2(X) = b_2^-(X_{d,m})$.
However, $X$ contains $\sqcup_i M_m(C,p_i)$ and therefore the required inequality follows.
\end{proof}

\begin{rmk}\label{r:cyclicpi1}
If we assume that $C$ in the previous proposition is a complex curve, or that it is symplectic of degree at most $17$, the condition that $m$ is a prime power can be dropped. In fact, this follows from the fact that, under these assumptions, $\pi_1(\CP\setminus C')$ is cyclic for some desingularisation $C'$ of $C$, therefore implying that $X$ is simply-connected for every $m$. (This follows from the solution to the symplectic isotopy problem in degrees up to $17$~\cite{Gromov,Sikorav,Shevchishin,SiebertTian} if we assume that $C$ is symplectic.)
\end{rmk}

%
We will apply the previous proposition to double covers of sextics ($(d,m)=(6,2)$) and degree-$8$ curves ($(d,m) = (8,2)$), and $7$--fold covers of septics ($(d,m)=(7,7)$). For these values, we have:
\begin{itemize}
\item $b_2^-(X_{6,2}) = 19$; $X_{6,2}$ is diffeomorphic to a K3 surface;
\item $b_2^-(X_{8,2}) = 37$;
\item $b_2^-(X_{7,7}) = 146$; $X_{7,7}$ is diffeomorphic to a degree-$7$ hypersurface in $\mathbb{CP}^3$.
\end{itemize}
More generally, $b_2^-(X_{d,m})$ can be computed algebro-geometrically using the ramification formula, Noether's formula and an elementary computation of the (topological) Euler characteristics for branched covers; a more topological (and more general) tool is given by the $G$--signature theorem of Atiyah and Singer~\cite{AtiyahSinger} (see~\cite{Gordon} for a more elementary approach in dimension $4$).

As for the local contributions, a topological way to compute $b_2^-(M_m(C,p))$ comes from a result of Kauffman~\cite[Corollary~5.7]{Kauffman}.
Moreover, Milnor fibres of isolated hypersurface singularities in $\C^3$ are homotopic to bouquets of spheres of (real) dimension $2$, so $b_1(M_m(C,p)) = b_3(M_m(C,p)) = 0$. It now suffices to compute the Euler characteristics and the signature to have $b_2^-(M_m(C,p))$.
The input for Kauffman's result is a Seifert surface for the link of the singularity of $(C,p)$. We will not give further details here; calculations can be performed algorithmically and have been implemented for a computer to run.

\begin{prop}\label{p:branched67}
There is no adjunctive PL sphere of degree $6$ in $\CP$ with simple singularities whose complement has cyclic fundamental group.

If an adjunctive PL sphere $C$ of degree $7$ in $\CP$ has multiplicity multisequence $[[3,3,3,2^{[6]}]]$ and $\pi_1(\CP \setminus C)$ is cyclic, then it must have a singularity of type $(3,k)$ for some $k\ge 8$.
\end{prop}

\begin{proof}
We compute $b_2^-(M_m(C,p))$ for simple singularities and for the singularity of type $[3,3]$ when $m=2$ and $m=7$. We obtain:
\[
\begin{array}{l|cccc}
\textrm{singularity (MS)} & [3,3] & [3,2] & [3] & [2^{[\ell]}]\\
\hline
\\[-10 pt]
b_2^-(M_2(C,p)) & 10 & 8 & 6 & 2\ell\\
b_2^-(M_7(C,p)) & 58 & 40 & 30 & 10\ell
\end{array}
\]
For sextics, when $m=2$, the sum of all local contributions is 20, independently of the partition. For septics of type $[[3,3,3,2^{[6]}]]$ without singularities of type $(3,k)$ with $k\ge 8$, when $m=2$, the sum is either $148$ or $150$. Either way, Theorem~\ref{p:signatureobstruction} would be violated.
\end{proof}

\begin{rmk}
For instance, in the case of sextics, the only possible singularities appearing on a curve with multiplicity multisequence $[[3,2^{[7]}]]$ are simple (i.e. ADE) singularities; as remarked above, the assumption that the fundamental group of the complement is cyclic is not necessary. We also note that the result for sextics was already observed in~\cite[Proposition~7.13]{GS}.
\end{rmk}

\subsection{Mixed obstructions}\label{ss:mixed}

In this section we deal with topological obstructions for symplectic (or almost-complex) curves. Recall that if $C$ is a symplectic rational cuspidal curve $\CP$, there exists an almost-complex structure $J$ on $\CP$, compatible with the Fubini--Study symplectic structure, such that $C$ is $J$--holomorphic. Moreover, by work of Gromov on pseudo-holomorphic curves~\cite{Gromov}, for every point in an almost-complex projective plane $(\CP,J)$ there exists a pencil of $J$--holomorphic lines with base at the point sweeping out $\CP$.

The first obstruction we encounter is the Riemann--Hurwitz obstruction. This is classically known for complex rational cuspidal curves; the adaptation to the symplectic case is straightforward (building on the seminal work of Micallef and White~\cite{MicallefWhite}, and Gromov~\cite{Gromov}), and was written down in~\cite[Section~3.5]{GS}.

\begin{prop}\label{p:RHobstruction}
Suppose that $C$ is a symplectic rational cuspidal curve of degree $d$ in $\CP$, whose singularities have multiplicity sequences $[m^i_1, \dots, m^i_{\ell_i}]$ for $i=1,\dots,\nu$. Then
\begin{equation}\label{e:RH}
2(d-m^1_1) \ge 2 + (m^1_2 -1) + \sum_{i=2}^\nu (m^i_1-1),
\end{equation}
where by convention $m^1_2 = 1$ if $\ell_1 = 1$.
\end{prop}

As mentioned in the introduction, the Riemann--Hurwitz obstruction eliminates many cases that were unobstructed by Theorem~\ref{t:semigroupobstruction}. We do not explicitly list them here, but rather we give one example.

\begin{ex}
Look at all septics with multiplicity multisequence $[[4,3,3,3]]$; we claim that the only two possible combinations of singularities allowed by the Riemann--Hurwitz obstruction are $[[4,3],[3,3]]$ and $[[4],[3,3,3]]$.
First, we observe that the only possible singularities of multiplicity $4$ are $[4,3]$ and $[4]$. We distinguish two cases, depending on which one of the two occurs.

If $[4,3]$ occurs, choosing to label this as the singularity with $i=1$, then the inequality~\eqref{e:RH} reads: $2(7-4) \ge 2 + (3-1) + \sum_{i>1} (m^i_1-1)$; since $m_1^i = 3$ for $i>1$, this implies that there is at most one more singularity, which has to be of type $[3,3]$.

If, on the other hand, $[4]$ occurs, choosing to label this with $i=1$,~\eqref{e:RH} reads: $2(7-4) \ge 2 + (1-1) + \sum_{i>1} (m^i_1-1)$, so there are at most two more singularities. If there is one, it is of type $[3,3,3]$. If, on the other hand, there are two, one of them is of type $[3,3]$ and the other of type $[3]$; projecting from the former, the inequality~\eqref{e:RH} gives: $8 = 2(7-3) \ge 2 + (3-1) + (4-1) + (3-1) = 9$, a contradiction.
\end{ex}

We now combine the previous proposition with a variation on a theme from the previous section.

\begin{prop}\label{p:28-34}
There is no symplectic rational cuspidal septic in $\CP$ with a singularity of type $[4]$ and either one of type $[3,2]$ or one of type $[3]$.
\end{prop}

\begin{proof}
Suppose that such a curve $C_0$ existed, and call $p$ its singularity of type $[4]$ and $q$ that of type $[3,2]$ or $[3]$. Note that, by the Riemann--Hurwitz obstruction of the above proposition, all other singularities of $C_0$ are simple. (This requires some elementary case-by-case analysis; in fact, all singularities are forced to be of multiplicity $2$, i.e. of type $[2^{[k]}]$.)

Choose an almost-complex structure $J$ on $\CP$ compatible with the curve, and let $C$ be the (reducible) degree $8$ $J$--holomorphic curve comprising $C_0$ and the line through $p$ and $q$. Note that the line intersects $C_0$ only at $p$ and $q$, and it does so transversely, since the sum of the multiplicities of the singularities at $p$ and $q$ is $7$.

If we (symplectically) smooth the singularities of $C$, we obtain a non-singular symplectic curve $C'$ of degree $8$ in $\CP$, which is isotopic to a complex curve~\cite{SiebertTian}. In particular, the fundamental group of $\CP\setminus C'$ is cyclic; see Remark~\ref{r:cyclicpi1}.

We now take the double cover of $\CP$ branched over $C'$, and apply Theorem~\ref{p:signatureobstruction}\footnote{Strictly speaking, we are applying a variation of the proposition for reducible curves. The adaptation of the proof to this case is straightforward, the crucial point being that $C'$ is a symplectic smoothing of $C$ whose complement has cyclic fundamental group.}.
We can compute $b_2^-(M_2(C,p))$ and $b_2^-(M_2(C,q))$ either from the singularity theory viewpoint or from the topological perspective, as we did above. We find that $b_2^-(M_2(C,p)) = 17$; if $(C_0,q)$ is of type $[3,2]$, then $b_2^-(M_2(C,q)) = 11$; if, on the other hand, $(C_0,q)$ is of type $[3]$, then $b_2^-(M_2(C,q)) = 9$.

Since all other singularities of $C$ are those of $C_0$, and these are simple, then the local contributions to $b_2^-$ add up to $10$ if $(C_0,q)$ is of type $[3,2]$, and to $12$ if $(C_0,q)$ is of type $[3]$. Either way, the sum of the local contributions for $C$ is $17 + 11 + 10 = 17 + 9 + 12 = 38$, which is larger than $b_2^-(X_{8,2}) = 37$, which we computed in the previous section.
\end{proof}

We now turn to the Levine--Tristram signature obstruction, also known as spectrum semicontinuity; this is due to Borodzik and N\'emethi~\cite[Corollary~2.5.4]{BorodzikNemethi}.
We state it in terms of link signatures $\sigma_L(\cdot)$ and nullities $\eta_L(\cdot)$, which are two integer-valued function on $S^1$ defined in terms of Seifert matrices; we refer to any classical textbooks for a reference, for instance~\cite[Chapter~8]{Lickorish}.
The statement is borrowed from~\cite[Section~2.4]{GS}; we let $T(d,d)$ be the $(d,d)$--torus link, i.e. the link of the singularity $\{x^d+y^d = 0\}$ at the origin of $\C^2$.

\begin{prop}\label{p:LT}
Let $C$ be a degree-$d$ symplectic rational cuspidal curve in $\CP$, whose singularities have links $K_1,\dots,K_\nu$. Let $K = K_1\#\dots\#K_\nu$. Then
\begin{equation}\label{e:LT}
|\sigma_{T(d,d)}(\zeta) - \sigma_K(\zeta)| + |\eta_{T(d,d)}(\zeta)-\eta_K(\zeta)| \le d - 1
\end{equation}
for every $\zeta \in S^1$.
\end{prop}

The idea of the proof is that, by removing the neighbourhood of a generic line in $\CP$ and a neighbourhood $N$ of a path in $C$ connecting all singular points of $C$, we get a real surface of genus $0$ connecting the link $T(d,d)$ (which we can think of as the ``link at infinity'' of $C$) to the knot $K$ (which appears at the boundary of the $N$). The inequality then reduces to classically known facts about link cobordisms (in the form explicitly stated in~\cite[Theorem~1.2]{ConwayNagelToffoli}, setting $\mu=1$ and $c=0$).

There are twelve curves that are unobstructed by the previous results, but obstructed by the above criterion; we collect them in the next proposition.
\begin{prop}\label{p:LT7}
There is no symplectic rational cuspidal curve of degree $7$ in $\CP$ of either of the following types:
\[
\begin{array}{llll}
[[4], [3, 3, 2], [2, 2]], & [[3^{[4]}, 2], [2]^2], & [[3^{[4]}]], [2]^3], & [[3^{[3]}, 2], [3], [2]^2],\\ \relax
[[3, 3, 2]^2, [2]], & [[3, 3, 2], [3, 3], [2]^2], & [[3^{[3]}, 2], [2^{[4]}], [2]], & [[3^{[3]}, 2], [2^{[3]}], [2]^2],\\ \relax
[[3^{[3]}, 2], [2, 2], [2]^{3}], & [[3^{[3]}], [2^{[4]}], [2]^2], & [[3^{[3]}], [2^{[3]}], [2]^3], & [[3, 3, 2], [3], [2^{[4]}], [2]].
\end{array}
\]
\end{prop}

\begin{proof}
We give the proof for the first case. The other are analogous, and the calculation was carried out by a computer.

We want to show that there is no symplectic rational cuspidal curve of degree $7$ with singularities $[[4], [3, 3, 2], [2, 2]]$; the links of the three singularities are $T(4,5)$, $T(3,8)$, and $T(2,5)$. We look at $\zeta = \exp(4\pi i/7)$:
\[
\begin{array}{ccccc}
L & T(7,7) & T(4,5) & T(3,8) & T(2,5)\\
 \sigma_L(\exp(4\pi i/7)) & -19 & -6 & -8 & -2\\
 \eta_L(\exp(4\pi i/7)) & 5 & 0 & 0 & 0.
\end{array}
\]
In particular, the inequality~\eqref{e:LT} is violated: $8 = \left|-19 + 6 + 8 + 2\right| + |5| \not\le 7 - 1 = 6$.
\end{proof}

We present here two variations on the same theme. The first one involves a small tweak to Proposition~\ref{p:LT}; the second result recovers a special case of a result of Za\u{\i}denberg and Lin~\cite{ZaidenbergLin}, asserting that affine injective morphisms $\C \to \C^2$ have one singularity; in more topological terms, they are cones.

\begin{prop}\label{p:LT7t}
There is no symplectic rational cuspidal septic in $\CP$ of type $[[3,3,2],[3,3],[2,2]]$, $[[3,3],[3,3],[2,2,2]]$ or $[[3,3],[3,3],[2,2],[2]]$.
\end{prop}

\begin{proof}[Sketch of proof]
Suppose that such a curve $C$ existed. By Theorem~\ref{t:addtheline}, if $C$ exists then there also exists a configuration comprising $C$ and a line $\ell$ with a simple tangency to $C$ at a non-singular point, and transverse to $C$ everywhere else.
Choose an almost-complex structure $J$ compatible with $\omega_{\rm FS}$ and such that $C \cup \ell$ is $J$--holomorphic.

As in the proof of Proposition~\ref{p:LT} that we sketched above, from $C$ and $\ell$ we obtain a planar cobordism from the connected sum $K$ of the links of the singularities of $C$ to the link at infinity of $C$ with respect to $\ell$. In this case, the link at infinity, that we call $T'$ is obtained from $T(6,6)$ (where $6$ is the number of geometric intersections of $C$ and $\ell$) by taking the $(2,1)$--cable of one of its components (corresponding to one of the intersection points being a tangency).

The corresponding signature inequality that we obtain is:
\[
|\sigma_{T'}(\zeta) - \sigma_K(\zeta)| + |\eta_{T'}(\zeta)-\eta_K(\zeta)| \le 5.
\]
Note that on the right-hand side we have a $5$ instead of a $6$, since the right-hand side measures the topology of the cobordism, rather than the degree of the curve.
The inequality is violated at $\zeta = \exp(6\pi i/7)$ for all three types in the statement.
\end{proof}

\begin{prop}\label{p:ZaidenbergLin}
If a rational cuspidal septic $C$ in $\CP$ has a singularity of type $[4,3]$, then it is of type $[[4,3],[3,3]]$.
\end{prop}

\begin{proof}
Choose an almost-complex structure $J$ compatible with $\omega_{\rm FS}$ and with respect to which $C$ is $J$--holomorphic. Call $p$ the singular point of $C$ of type $[4,3]$ and $K$ the connected sum of the links of the other singularities of $C$. The tangent line $\ell$ to $C$ at $p$ intersects $C$ only at $p$, because the third element of the semigroup of $C$ is $7$. (See Lemma~\ref{l:generictangent}.)

The link at infinity of $C$ with respect to $\ell$, \emph{viewed from the complement of $\ell \subset \CP$}, is the torus knot $T(3,7)$. This is a special case of Za\u{\i}denberg and Lin's analysis~\cite{ZaidenbergLin, Neumann}; in this case, we can see it by considering the curve $\{x^3z^4-y^7 = 0\}$ with the line $\ell = \{z = 0\}$, giving a model for the link at infinity of $C$ which is the link of the only singularity of the affine curve $\{x^3 - y^7 = 0\}$.
Removing a neighbourhood of $\ell$ and a connected neighbourhood of the singularities of $C$, we obtain a planar cobordism from $K$ to $T(3,7)$. Since both ends of the cobordism are knots, this is a concordance.

Now observe that if a rational cuspidal septic has a singularity of type $[4,3]$, the remaining possible multiplicity sequences concatenate to $[4]$, $[3,3]$, $[3,2,2,2]$, or $[2^{[6]}]$ ($[4]$ is, in fact, excluded by the Heegaard Floer criterion.) In particular, the only possible singularities that we can obtain have one Puiseux pair, i.e. their links are torus knots. Since Litherland proved that torus knots are linearly independent in the concordance group~\cite{Litherland}, the only possibility we are left with is that $C$ has a singularity of type $[3,3]$ (whose link is indeed $T(3,7)$).
\end{proof}

\subsection{Auxiliary configurations}\label{ss:auxiliary}

Some auxiliary configurations appear frequently as obstructions to the existence of curves.
Certain configurations of conics and lines were shown to be symplectically obstructed in~\cite[Section~5]{GS}; among these, the configuration $\G_4$, comprising two conics with a tangency of order $4$ (which is their unique intersection point, $p$), and a line not containing $p$ and tangent to both.

The main result of this subsection is the following symplectic non-realisability result.

\begin{prop}\label{p:auxiliaryP2}
The following configurations of symplectic curves in $\CP$ are not symplectically realised:
\begin{itemize}
\item[$\Vc$:] the configuration $V \cup \ell_1 \cup \ell_2$ comprising a rational cuspidal quartic $V$ and two lines $\ell_1$, $\ell_2$, with only simple singularities whose Milnor numbers sum to at least 20;

\item[$\Qc$:] the configuration $Q \cup \ell$ comprising a rational cuspidal quintic $Q$ and a line $\ell$, with only simple singularities whose Milnor numbers sum to at least 20.
\end{itemize}
\end{prop}

We recall that simple (or ADE, or du Val) curve singularities are classified as follows:
\begin{itemize}
\item an $A_{2n+1}$--singularity is a tangency of order $n$ (so that $A_1$ is a double point and $A_3$ is a simple tangency);
\item an $A_{2n}$--singularity is a cusp of type $[2^{[n]}]$;
\item a $D_{n+3}$--singularity is a singularity of type $A_n$ plus a transverse branch (so that $D_4$ is an ordinary triple point);
\item an $E_6$--singularity (respectively, $E_8$--singularity) is a cusp of type $[3]$ (resp., $[3,2]$);
\item an $E_7$--singularity is a cusp of type $[2]$ with a tangent branch (with multiplicity of intersection $3$ to the cusp).
\end{itemize}
We also recall that the Milnor number of a singularity of type $A_n$, $D_n$, or $E_n$ is $n$.

An example of a configuration of type $\Vc$ is given by a rational cuspidal quartic $V$, an inflection line $\ell_i$, and a line with a quadruple tangency $\ell_d$, with both tangencies occurring at non-singular points of $V$: in this case, the singularities of $V$ are simple (they are either of type $E_6$, $A_6$, $A_4$ and $A_2$, or three of type $A_2$), and the other singularities of the configuration are of types $A_5$ (the inflection line), $A_7$ (the quadruple tangency), and two of type $A_2$ (the intersections of $\ell_i$ with $V$ and $\ell_d$).

An example of a configuration of type $\Qc$ is a rational cuspidal quintic $Q$ with simple singularities and a line $\ell$ which is both a flex and a tangent (at two distinct non-singular points of $Q$): the singularities of the quintic have Milnor numbers summing to $12$, and the other singularities of the configuration are of type $A_5$ and $A_3$.

The proof of Proposition~\ref{p:auxiliaryP2} is very similar to the proof of Proposition~\ref{p:branched67} above, and it is based on Theorem~\ref{p:signatureobstruction} (see also Remark~\ref{r:cyclicpi1}).

\begin{proof}
We prove the statement for a configuration of type $\Vc$ and $\Qc$. If we smooth the singularities of the configuration symplectically, we obtain an irreducible, non-singular symplectic curve $C'$ of degree $6$ (either $6=4+1+1$ for $\Vc$ or $6=5+1$ for $\Qc$). By the solution to the symplectic isotopy problem in degree $6$~\cite{Shevchishin, SiebertTian}, $C'$ is isotopic to a complex curve, and in particular we can apply Theorem~\ref{p:signatureobstruction} with $m=2$. (Alternatively, we know that the double cover of $\CP$ branched over a sextic is a K3 surface, so that $X = X_{6,2}$ is a K3.)
We can now apply Theorem~\ref{p:signatureobstruction}: by the assumption, the left hand side of the inequality~\eqref{e:b2minus} is $20$, while the right-hand side is $b_2^-(X) = 19$.
\end{proof}

\begin{rmk}\label{r:simpledouble}
We point out that the proposition above can also be proved by hand, without appealing to the solution of the symplectic isotopy problem in degree 6. We sketch here an argument, partly inspired by Ruberman and Starkston's beautiful paper~\cite[Sections~3 and~4]{RubermanStarkston}, leveraging on the assumption that all singularities of the configurations we consider are simple (see also~\cite[Section~7.2]{GompfStipsicz}).
Call $\mathcal{C}$ a configuration of type $\Vc$ or $\Qc$. We can resolve each singularity by blowing up $\CP$ at the singular points.
Since the singularities of $\mathcal{C}$ are simple, up to further blow-ups, there is a (possibly reducible, but reduced) non-singular curve $D\subset X$ in the total transform of $\mathcal{C}$, which contains the proper transform of the $\mathcal{C}$ itself, and whose homology class is divisible by $2$. Moreover, since all curves in $\mathcal{C}$ are rational, so are all curves in $D$.
%
%
%
Taking the double cover of $X$ branched over $D$, we obtain a (symplectic) $4$--manifold $\tilde X$. We claim, but do not verify, that, independently of which singularities $C$ had, $\tilde X$ can be blown down to a $4$--manifold $X'$ that has the same homology and same signature as a $K3$ surface. (If we worked in the complex setting, this would be essentially obvious from the ramification formula.) We can easily compute the Euler characteristics and the signature from multiplicativity and additivity of the Euler characteristics and from the $G$--signature theorem as above; to show that the first Betti number vanishes, we need to use rationality of the components of $D$ as in~\cite[Corollary~3.4]{RubermanStarkston}.

Inside $X'$ we find (as explicit plumbings) the Milnor fibres of the suspensions of the singularities of $\mathcal{C}$, which are simple surface singularities, which give a negative definite submanifold $Z$ of $X'$ with $b_2^-(Z) = 20$. But this contradicts the fact that $b_2^-(X') = 19$.
\end{rmk}
%

\subsection{Birational transformations: sextics}\label{ss:birational6}

In this section, we obstruct the existence of the five remaining sextics, those of types $[[3,3,2],[2^{[3]}]]$, $[[3,3,2],[2,2],[2]]$, $[[3,3],[2^{[4]}]]$, $[[3,3],[2^{[3]}], [2]]$, and $[[3,3],[2,2],[2,2]]$. In all these cases, the self-intersection of the proper transform of the curve in the minimal resolution is $+2$, so in principle we could apply McDuff's theorem to obstruct them (or to construct them, if they existed). We find it more convenient to use a branched cover argument (together with a birational transformation) instead.

\begin{prop}\label{p:obstructedsextics}
There is no symplectic rational cuspidal sextic in $\CP$ of type $[[3,3,2],[2^{[3]}]]$, $[[3,3,2],[2,2],[2]]$, $[[3,3],[2^{[4]}]]$, $[[3,3],[2^{[3]}], [2]]$, or $[[3,3],[2,2],[2,2]]$.
\end{prop}

\begin{proof}
Suppose that such a sextic $C$ existed, and call $p$ its singularity of multiplicity $3$, and $q$ another singular point of $C$ that is \emph{not} a simple cusp (which has multiplicity $2$). In particular, $(C,q)$ is of type $A_{2k}$ for some $k>1$.
Call $t$ the tangent to $C$ at $p$, and $\ell$ the line through $p$ and $q$.
Since the semigroup of the singularity at $p$ starts with $0,3,6,\dots$ and that of the singularity at $q$ starts with $0,2,4,\dots$, $t\cap C = \{p\}$ and $\ell \cap C = \{p,q,r\}$, where $r$ is a non-singular point of $C$.

Blow up twice at $p$ and once at $q$. The proper transforms $\tilde t$ and $\tilde \ell$ of $t$ and $\ell$ become $(-1)$--curves, that we can contract. We can also contract the exceptional divisor $e$ of the first blow-up at $p$ (which is a $(-2)$--curve intersecting $\tilde\ell$ transversely once).

We claim that, by doing so, we obtain a configuration of type $\Vc$ in $\CP$. Let us call $e_p$ and $e_q$ the other exceptional divisors of the second blow-up at $p$ and of the blow-up at $q$, respectively.

In order to prove the claim, we distinguish two cases, according to whether $p$ is of type $[3,3,2]$ or of type $[3,3]$.

In the former case, the configuration of $V \cup e'_p \cup e'_q$ in the blow-down comprises:
\begin{itemize}
\item a rational cuspidal curve $V$ with a simple cusp (at $p'$: this is what is left over from the singularity of $(C,p)$ after the two blow-ups), a singularity of type $A_{2k-2}$ at $q'$ (left over from the singularity of $(C,q)$), and a singularity of type $A_{6-2k}$ (where, by convention, $A_0$ is a non-singular point); the self-intersection of $V$ is $36-9-9-4+1+1+1 = 16$, so indeed $V$ is a quartic.
\item a line $e'_p$ that is tangent to the simple cusp of $V$ at $p'$, and intersecting $V$ at non-singular point $r'$ (this is the contraction of $\ell$);
\item a line $e'_q$ that is tangent to $V$ at $r'$ and passing through $q'$.
\end{itemize}
This means that $V \cup e'_p \cup e'_q$ has a singularity of type $E_7$ (at $p'$), one of type $D_{2k+1}$ (at $q'$), one of type $D_6$ (at $r'$), and one of type $A_{6-2k}$ (the other singularity of $V$). They are all simple, and their Milnor numbers sum to $7+2k+1+6+6-2k = 20$, so they form a configuration of type $\Vc$, which is obstructed by Proposition~\ref{p:auxiliaryP2}.

In the case where $(C,p)$ was of type $[3,3]$, the same argument as above gives a configuration $V\cup e'_p \cup e'_q$. In this case, though, the other singularity of $C$ (and hence of $V$) is now of type $A_{8-2k}$. With the same labelling of points and curves, $V$ is smooth at $p'$ and $e'_p$ is an inflection line to it.
All in all, we have traded a point of type $E_7$ and one of type $A_{6-2k}$ for one of type $A_5$ and one of type $A_{8-2k}$; therefore, the sum of the Milnor numbers is the same, and we find another configuration of type $\Vc$. This concludes the proof.
\end{proof}

\subsection{Birational transformations: septics}\label{ss:birational7}

In this section, we obstruct the existence of the remaining septics, by means of birational transformations.

\begin{prop}\label{p:4222222,3}
There is no symplectic rational cuspidal septic $C$ in $\CP$ of type $[[4,2^{[6]}],[3]]$.
\end{prop}

\begin{proof}
Suppose that such a curve $C$ exists, and call $p$ the singular point of type $[4,2^{[6]}]$ and $q$ that of type $[3]$. Let $J$ be an almost-complex structure on $\CP$ compatible with the Fubini--Study symplectic form, and such that $C$ is $J$--holomorphic. Let $t_p$ be the tangent to the cusp of $C$ at $p$. Note that the multiplicity of intersection of $t_p$ and $C$ at $p$ is $6$ (it is at least $6$, since the first three elements of the semigroup of $(C,p)$ are: $0$, $4$ and $6$, and at most $7$ because the curve is a septic and because of positivity of intersection, but $7$ in not an element of the semigroup of $(C,p)$). In particular, $t_p$ intersects $C$ transversely in another (non-singular) point.

Now blow up $\CP$ five times at $p$. Call $e_1,\dots, e_5$ the components of the total transform of $C$ (excluding $C$), numbered in their order of appearance.

The line $t_p$ has been blown up twice, so it can be contracted. In the resulting surface, $e_2$ contracts to a $(-1)$--curve, since it intersects $t_p$ transversely once. So we can contract it. Following the same argument, we can contract successively $e_3$, $e_4$ and $e_1$. Note that the resulting symplectic $4$--manifold $X$ after blowing down is again $\CP$: indeed, $e_5$ blows down to a $+1$--sphere $\Sigma$, and McDuff's theorem tells us that the pair $(X, \Sigma)$ is symplectomorphic to $(\CP, \lambda \omega_{\rm FS})$ for some positive $\lambda$.

The blow-down of $C$ is now a curve $C'$ that has a tangency of order $3$ to $\Sigma$ at a simple cusp and intersects $\Sigma$ transversaly with order $2$ at a cusp of type $[2,2]$, and is otherwise disjoint from $\Sigma$. (The intersection of order $2$ at the cusp of type $[2,2]$ comes from blowing up five times at $p$, and the tangency of order $3$ at a simple cusp is created by blowing down $t_p$, $e_2$, $e_3$, $e_4$ and $e_1$.) It follows that $C'$ is a quintic in $\CP$. But then the configuration $C' \cup \Sigma$ is a configuration of type $\mathcal{Q}$, which cannot exist by Proposition~\ref{p:auxiliaryP2}.
\end{proof}

\begin{prop}\label{p:42222,322}
There is no symplectic rational cuspidal septic $C$ in $\CP$ of type $[[4,2^{[4]}],[3,2],[2]]$ or $[[4,2^{[4]}],[3],[2,2]]$.
\end{prop}

\begin{proof}
We prove the result for the case of $[[4,2^{[4]}],[3,2],[2]]$. The other case is completely analogous. Suppose that such a curve $C$ exists, and call $p$ the singular point of type $[4,2^{[4]}]$, $q$ that of type $[3,2]$, and $r$ that of type $[2]$. Let $J$ be an almost-complex
structure on $\CP$ compatible with the Fubini--Study symplectic form, and such that $C$ is $J$--holomorphic. Let $t_p$ be the tangent to the cusp of $C$ at $p$. Note that the multiplicity of intersection of $t_p$ and $C$ at $p$ is $6$ (it is at least $6$, since the first three elements of the semigroup of $(C,p)$ are: $0$, $4$ and $6$, and at most $7$ because the curve is a septic and because of positivity of intersection, but $7$ is not an element of the semigroup of $(C,p)$). In particular, $t_p$ intersects $C$ transversely in another (non-singular) point.

Now blow up $\CP$ five times at $p$, so that the proper transform of $C$ has only two remaining singularities at $q$ and $r$. Call $e_1,\dots, e_5$ the components of the total transform of $C$ (excluding $C$), numbered in their order of appearance in the minimal resolution of $p$.

The line $t_p$ has been blown up twice, so it can be contracted. In the resulting surface, $e_2$ contracts to a $(-1)$--curve, since it intersects $t_p$ transversely once. So we can contract it. Following the same argument, we can contract successively $e_3$, $e_4$ and $e_1$. Note that the resulting symplectic $4$--manifold $X$ after blowing down is again $\CP$: indeed, $e_5$ blows down to a $+1$--sphere $\Sigma$, and McDuff's theorem tells us that the pair $(X, \Sigma)$ is symplectomorphic to $(\CP, \lambda \omega_{\rm FS})$ for some positive $\lambda$.

The blow-down of $C$ is now a curve $C'$ that has a tangency of order $2$ to $\Sigma$ at a non-singular point and a tangency of order $3$ to $\Sigma$ at a simple cusp, and is otherwise disjoint from $\Sigma$. (The tangency of order $2$ at a non-singular point comes from the minimal resolution of $p$, and the tangency of order $3$ at a simple cusp is created by blowing down $t_p$, $e_2$, $e_3$, $e_4$ and $e_1$.) It follows that $C'$ is a quintic in $\CP$. But then the configuration $C' \cup \Sigma$ is exactly a configuration of type $\mathcal{Q}$, which cannot exist by Proposition~\ref{p:auxiliaryP2}.
\end{proof}

\begin{prop}\label{p:422,3,2222}
There is no symplectic rational cuspidal septic $C$ in $\CP$ of type $[[4,2,2],[3],[2^{[4]}]]$.
\end{prop}

\begin{proof}
Suppose that such a curve $C$ exists, and call $p$ the singular point of type $[4,2,2]$, $q$ that of type $[3]$, and $r$ that of type $[2^{[4]}]$. Let $J$ be an almost-complex structure on $\CP$ compatible with the Fubini--Study symplectic form, and such that $C$ is $J$--holomorphic. Let $l_{p,q}$ be the unique $J$--holomorphic line that passes through $p$ and $q$ and $t_p$ be the tangent to the cusp of $C$ at $p$. Note that the multiplicity of intersection of $t_p$ and $C$ at $p$ is $6$ (it is at least $6$, since the first three elements of the semigroup of $(C,p)$ are: $0$, $4$ and $6$, at most $7$ because the curve is a septic and because of positivity of intersection, but $7$ in not an element of the semigroup of $(C,p)$). In particular, $t_p$ intersects $C$ transversely in another (non-singular) point.

Now blow up $\CP$ twice at $p$ and once at $q$, so that the proper transform of $C$ has only one remaining singularity at $r$. Call $e_1,e_2,e_3$ the components of the total transform of $C$ (excluding $C$), numbered so that $e_3$ contracts to $q$, $e_2 \cdot e_2 =- 1$ and $e_1 \cdot e_1 = -1$.

Each of $l_{p,q}$ and $t_p$ have been blown up twice, so they can be contracted. In the resulting surface, $e_1$ contracts to a $(-1)$--curve, since it intersects $l_{p,q}$ transversely once. So we can contract it. Note that the resulting symplectic $4$--manifold $X$ after blowing down is again $\CP$: indeed, $e_2$ blows down to a $+1$--sphere $\Sigma$, and McDuff's theorem tells us that the pair $(X, \Sigma)$ is symplectomorphic to $(\CP, \lambda \omega_{\rm FS})$ for some positive $\lambda$.

The blow-down of $C$ is now a curve $C'$ that has a tangency of order $3$ to $\Sigma$ at a simple cusp, intersects $\Sigma$ transversaly at a non-singular point, and is otherwise disjoint from $\Sigma$. (The tangency of order $3$ at a simple cusp comes from blowing up twice at $p$, and the transverse intersection is created by blowing down $l_{p,q}$.) It follows that $C'$ is a quartic in $\CP$. Finally, the curve $e_3$ blows down to a $+1$--sphere $\Sigma '$ in $\CP$ with a tangency of order $3$ to $C'$ at a non-singular point (coming from the blow up of $C$ at $q$) and intersecting $C'$ transversely once at the simple cusp that has tangent $\Sigma$. But then the configuration $C' \cup \Sigma '$ is a configuration of type $\mathcal{V}$, which cannot exist by Proposition~\ref{p:auxiliaryP2}.
\end{proof}

\begin{prop}\label{p:V}
There is no symplectic rational cuspidal septic in $\CP$ with singularities $[[4], [3,3], [2^{[3]}]]$ or $[[4], [3,3], [2,2],[2]]$.
\end{prop}

\begin{proof}
We prove the result for the case of $[[4], [3,3], [2^{[3]}]]$; the other case is completely analogous.
Suppose that such a curve $C$ exists, and call $p$ the singular point of type $[4]$, $q$ that of type $[3,3]$, and $r$ that of type $[2,2,2]$.
Let $J$ be an almost-complex structure on $\CP$ compatible with the Fubini--Study symplectic form, and such that $C$ is $J$--holomorphic. Let $\ell_{pq}$ be the unique $J$--holomorphic line that passes through $p$ and $q$, and $t_q$ be the tangent to the cusp of $C$ at $q$.
Note that by Lemma~\ref{l:generictangent}, up to perturbing $J$ through $\omega_{\rm FS}$--compatible almost-complex structures $J'$ for which $C$ is $J'$--holomorphic, we can assume that the multiplicity of intersection of $t_q$ and $C$ at $q$ is $6$. (It is at least $6$, since the first three elements of the semigroup of $(C,q)$ are: $0$, $3$ and $6$.) In particular, $t_q$ intersects $C$ transversely at $C$ in another (non-singular) point, while $\ell_{pq}$ intersects $C$ only at $p$ and $q$.

Now blow up $\CP$ once at $p$ and twice at $q$, so that the proper transform of $C$ has only one remaining singularity at $r$. Call $e_1$, $e_2$, and $e_3$ the components of the total transform of $C$ (excluding $C$), numbered so that $e_1$ contracts to $p$, and $e_2\cdot e_2 = -2$, and $e_3\cdot e_3 = -1$.

Each of $\ell_{pq}$ and $t_q$ has been blown up twice, so they can be contracted. In the resulting surface, $e_2$ contracts to a $(-1)$--curve, since it intersects $\ell_{pq}$ transversely once and is disjoint from $t_q$. So we can contract it. Note that the resulting symplectic $4$--manifold $X$ after blowing down is again $\CP$: indeed, $e_1$ blows down to a $+1$--sphere $\Sigma$, and McDuff's theorem tells us that the pair $(X, \Sigma)$ is symplectomorphic to $(\CP,\lambda\omega_{\rm FS})$ for some positive $\lambda$.

The blow-down of $C$ is now a curve $C'$ that has a tangency of order $4$ to $\Sigma$, and is otherwise disjoint from $\Sigma$. (The tangency of order $4$ comes from blowing up $C$ at $p$, and no other intersections are created in the process.) It follows that $C'$ is a quartic in $\CP$. Finally, the curve $e_3$ blows down to a $+1$--sphere $\Sigma'$ in $\CP$ with a tangency of order $3$ to $C'$ (coming from blowing up $C$ at $q$) and intersecting $C'$ transversely at another point. But then the configuration $C' \cup \Sigma \cup \Sigma'$ is exactly a configuration of type $\Vc$, which cannot exist by Proposition~\ref{p:auxiliaryP2}.
\end{proof}

Together with the Riemann--Hurwitz obstruction and the Levine--Tristram signature obstruction, the previous proposition obstructs all non-existing symplectic rational cuspidal septics with multiplicity multisequence $[[4,3,3,2,2,2]]$.
Note that the two curves of types $[[4],[3,3,2],[2,2]]$ and $[[4],[3,3,2],[2],[2]]$, that are obstructed by Proposition~\ref{p:LT7}, are also obstructed by a $\Vc$--configuration. Indeed, the same sequence of blow-ups and blow-downs as in the proof above yields a configuration of a quartic and two lines, one of which has a tangency of order 4 with the quartic, and the other is tangent to the quartic at a simple cusp (i.e. type $[2]$).

\begin{prop}\label{p:Q1}
There is no symplectic rational cuspidal septic $C$ in $\CP$ with a singularity of type $[3^{[3]}]$ and whose other singularities are simple.
\end{prop}

\begin{proof}
Suppose that such a curve exists. Let $p$ be the singular point of $C$ of type $[3^{[3]}]$, and $t_p$ the tangent line to $C$ at $p$. The multiplicity of intersection of $t_p$ and $C$ at $p$ is exactly $6$, so $t_p$ intersects $C$ transversely at a non-singular point.

Blow up $C$ three times at $p$, so as to resolve the singularity of $C$ at $p$. The line $t_p$ lifts to a $(-1)$--curve; the three blow-ups give three curves, $e_1$, $e_2$, and $e_3$, in the blown-up $\CP$, with self-intersections $-2$, $-2$, and $-1$, respectively. Moreover, $t_p$ intersects $e_2$ transversely once, and is disjoint from $e_1$ and $e_3$; the proper transform of $C$ is disjoint from $e_1$ and $e_2$, and has a tangency of order $3$ with $e_3$.

Contract the proper transform of $t_p$, then the blow-down of $e_2$ (which has become a $(-1)$--curve after the first blow-down), and finally the blow-down of $e_1$. We obtain a closed $4$--manifold $X$, in which $e_3$ has blown down to a $+1$--curve $\Sigma$ of genus $0$. By McDuff's theorem, $X$ is symplectomorphic to $(\CP, \lambda \omega_{\rm FS})$ for some $\lambda > 0$, in which $\Sigma$ is a line.

The proper transform of $C$ blows down to a curve $C'$ that has a tangency of order $3$ to $\Sigma$ (coming from the tangency of $e_3$ and $C$), and another tangency of order $2$ to $\Sigma$ (coming from the blow-downs). Moreover, its singularities are all simple, since they are the singular points of $C \setminus \{p\}$. Therefore, $C \cup \Sigma$ is a configuration of type $\Qc$, which is obstructed by Proposition~\ref{p:auxiliaryP2}.
\end{proof}


\begin{prop}\label{p:Q2.1}
There is no symplectic rational cuspidal septic $C$ in $\CP$ with a singularity of type $[3^{[3]},2]$ and whose other singularities are simple.
\end{prop}

In fact, the assumption that all other singularities are simple is redundant here and in the next proposition: it follows from Proposition~\ref{p:multisequences} and the fact that singularities with only $2$s and at most one $3$ in their multiplicity sequence are simple.

\begin{proof}
Suppose that such a curve $C$ exists, and call $p$ its singular point of type $[3^{[3]},2]$. Call $t_p$ the tangent to $C$ at $p$, which intersects $C$ also at a non-singular point. Blow up three times at $p$, so that the proper transform of $C$ has all singularities of $C\setminus\{p\}$, plus a simple cusp (which is the remnant from $p$).

Now, as in the previous proofs, contract the proper transform of $t_p$, and the exceptional divisors of the first two blow-ups. The exceptional divisor of the third blow-up contracts to a $+1$--curve $\Sigma$; again, by McDuff's theorem, up to rescaling, we have a configuration of curves in $\CP$, where $\Sigma$ is a line. The proper transform of $C$ blows down to a rational curve $C'$ with a simple cusp, to which $\Sigma$ is tangent, and simple singularities; $\Sigma$ has also a tangency to $C'$ at a smooth point (coming from the blow-downs).

Therefore, $C\cup \Sigma$ is a configuration of type $\Qc$, which is obstructed by Proposition~\ref{p:auxiliaryP2}.
\end{proof}


\begin{prop}\label{p:Q12}
There is no symplectic rational cuspidal septic $C$ in $\CP$ with a singularity of type $[3,3,2]$, one of type $[3,2]$, and whose other singularities are simple.
The same holds true if the second singularity is assumed to be of type $[3]$ instead of $[3,2]$.
\end{prop}

\begin{proof}
Suppose that such a curve $C$ exist, and call $p$ and $q$ its singular points of types $[3,3,2]$ and $[3,2]$ (or $[3]$), respectively.

Call $t_p$ be the tangent to $C$ at $p$, and $\ell$ the line through $p$ and $q$. In the case when $q$ is of type $[3]$, by Lemma~\ref{l:generictangent}, we can assume that $\ell$ is not tangent to $C$ at $q$; in either case each of $t_p$ and $\ell$ intersect $C$ at a non-singular point of $C$. (The two points are necessarily distinct, since $t_p$ and $\ell$ already intersect at $p$.)

Blow up twice at $p$ and once at $q$, so that the proper transform of $C$ has a simple cusp over $p$, and at worst a simple cusp over $q$. Call $e_1$, $e_2$, and $e_3$ the exceptional divisors in the blown-up $\CP$, and observe that the proper transforms of $t_p$ and $\ell$ are $(-1)$--curves; we contract them, as well as $e_1$ (the exceptional divisor of the first blow-up at $p$).

As in the previous proofs, $e_3$ contracts to a $+1$--curve $\Sigma$, so we obtain a symplectic configuration in $\CP$, comprising $\Sigma$ and a quintic $C'$, to which $e_3$ has a tangengy of order 3 (coming from the singularity at $q$: this is at a simple cusp or at a smooth point, depending on which case we are considering) and a simple tangency (coming from contracting the proper transform of $\ell$ and $e_1$); moreover, $C'$ has only simple singularities.

In the first case, when $(C,q)$ is of type $[3,2]$, $C'\cup \Sigma$ is of type $\Qc$; in the second, when $(C,q)$ is of type $[3]$, $C'\cup \Sigma$ is of type $\Qc$. In either case, the configuration $C'\cup \Sigma$ is obstructed, and therefore so is $C$.
\end{proof}

\begin{rmk}\label{r:Q3222}
For the septic $C$ of type $[[3,3,2],[3,2]^2]$, which is obstructed by the previous proposition, there is an alternative proof that we would like to mention. We can consider the triangle in $\CP$ whose vertices are the three singular points of $C$; we blow up at these vertices, and contract the proper transform of the three sides.
(This is a quadratic Cremona transformation of $\CP$, based at the three singular points of $C$.)
The curve $C$ is transformed into a quintic in $\CP$ of type $[[3,2],[2],[2]]$, which was proven not to exist in~\cite[Proposition~7.5]{GS}.
\end{rmk}

%
%
%

\begin{prop}\label{p:33332}
There is no symplectic rational cuspidal septics in $\CP$ of type $[[3^{[4]},2],[2,2]]$.
\end{prop}

\begin{proof}
Suppose that such a curve $C$ existed; we look at the minimal resolution of all singularities of $C$.
We obtain a configuration of eight curves in the $7$--fold blow-up $X$ of $\CP$, as depicted in Figure~\ref{f:33332,22}.

\begin{figure}
\centering
\includegraphics[width = 0.8\textwidth]{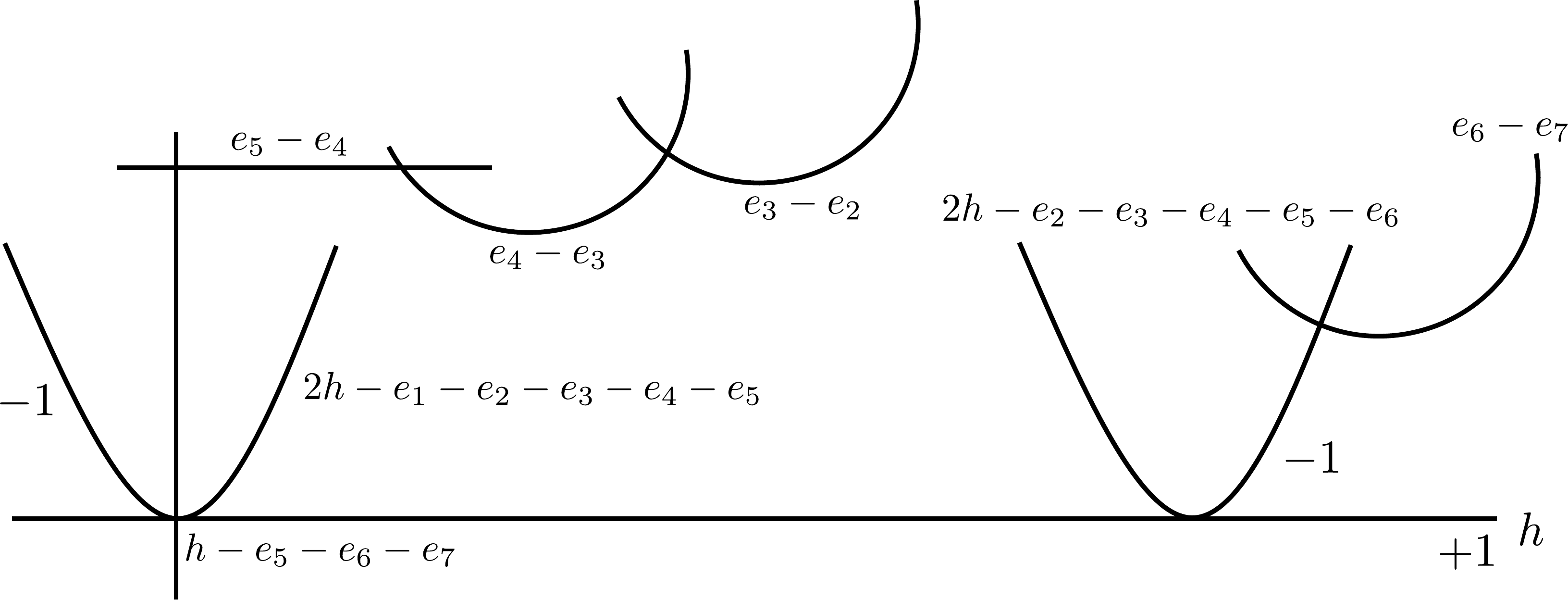}
\caption{The total transform of the curve of Proposition~\ref{p:33332} and its unique homological embedding corresponding to a sextic in $\CP$.}\label{f:33332,22}.
\end{figure}

The proper transform of $C$ is a $+1$--curve; by McDuff theorem, there is a birational transformation of $X$ such that the proper transform $C'$ of $C$ is sent to a line. In particular, in the new coordinates, we can write $[C'] = h$. The homology classes of the other curves in the configurations are determined by their genus (which is $0$ for all components), multiplicity of intersection with $C'$, and self-intersection, and they are exhibited in the same figure. Choose an almost-complex structure $J$ on $X$ such that the whole configuration is $J$--holomorphic.

We now sketch the uniqueness of the embedding. We start with $F_5$, the left-most $(-1)$--curve tangent to $C'$ (coming from the fifth blow-up at the singularity of type $[3^{[4]},2]$): since $F_5$ has a tangency of order $2$ to $C'$, and it is rational of self-intersection $-1$, it is in the homology class $2h-e_{i_1}-\dots-e_{i_5}$ for some indices $i_1,\dots,i_5$. Up to reordering, we can suppose $i_j = j$ for each $j$.
The $(-2)$--curve $F_4$ that intersects $C'$ transversely at the tangency of $C'$ and $F_5$ is now forced to be in the homology class $h - e_{j_1} - e_{j_2} - e_{j_3}$ for some indices $j_1$, $j_2$, $j_3$. Since it also intersects $F_5$ transversely once, exactly one of these indices, say $j_1$, has to satisfy $1\le j_1 \le 5$. Up to permuting the indices, we can choose $(j_1,j_2,j_3) = (5,6,7)$

From here, it is easy to see that the chain of $(-2)$--curves starting at $F_4$ has to be $e_5-e_4$, $e_4-e_3$, and $e_3-e_2$. The other $(-1)$--curve $F_7$ intersecting $C'$ has to be in the homology class $2h-e_{k_1}-\dots-e_{k_5}$ for some $k_1, \dots, k_5$, and it has to share four indices with $F_5$ and two with $F_4$; moreover, since it is disjoint from the $(-2)$--chain, it must either contain all of $e_2$, $e_3$, $e_4$, and $e_5$, or contain none of them. This forces the homology class to be $2h-e_2-e_3-e_4-e_5-e_{k_5}$, and $k_5$ to be $6$ (up to swapping the indices $6$ and $7$, which have symmetric roles). The final $(-2)$--curve has to be in the class $e_6-e_7$.

Now we can contract the unique $J$--holomorphic $(-1)$--curve in the homology classes $e_1, \dots, e_7$, obtaining a map $\pi: X \to \CP$. As seen many times in Section~\ref{s:existence}, the fact that $e_5-e_4$, $e_4-e_3$, and $e_3-e_2$ are all realised as $J$--holomorphic curves implies that, since $F_5$ and $F_7$ share $e_2, \dots, e_5$, blowing down $e_2, \dots, e_5$ we create a tangency of order $4$ between their blow-downs.

Now the curve $C'$ blows down to a line that is tangent to both $\pi(F_4)$ and $\pi(F_5)$, at two distinct points, so $\pi(C') \cup \pi(F_4) \cup \pi(F_5)$ forms a $\G_4$--configuration. But this configuration cannot exist, thus giving a contradiction.
\end{proof}

\begin{prop}\label{p:3333}
There is no symplectic rational cuspidal septic in $\CP$ of type $[[3^{[4]}],[2,2],[2]]$.
\end{prop}

This is very similar to the previous proposition, so we only sketch the proof.

\begin{proof}[Sketch of proof]
Suppose that such a $C$ existed.
Again, we look at the minimal resolution of $C$ in a $7$--fold blow-up $X$ of $\CP$.
The proper transform $C'$ of $C$ in $X$ is a smooth $+1$--sphere, so McDuff's theorem gives us a birational transformation of $X$ in which $C'$ is sent to a line.

The configuration given by the total transform of $C$ in $X$ is shown in Figure~\ref{f:3333,22,2}. We now proceed as above to determine the homology classes of the various curves: the curve $F_4$ with order of tangency $3$ to $C'$ is necessarily in the homology class $3h-2e_1-e_2-\dots-e_7$; this forces the $(-2)$--chain starting from it to be $e_1-e_2, \dots, e_3-e_4$.

\begin{figure}
\centering
\includegraphics[width = \textwidth]{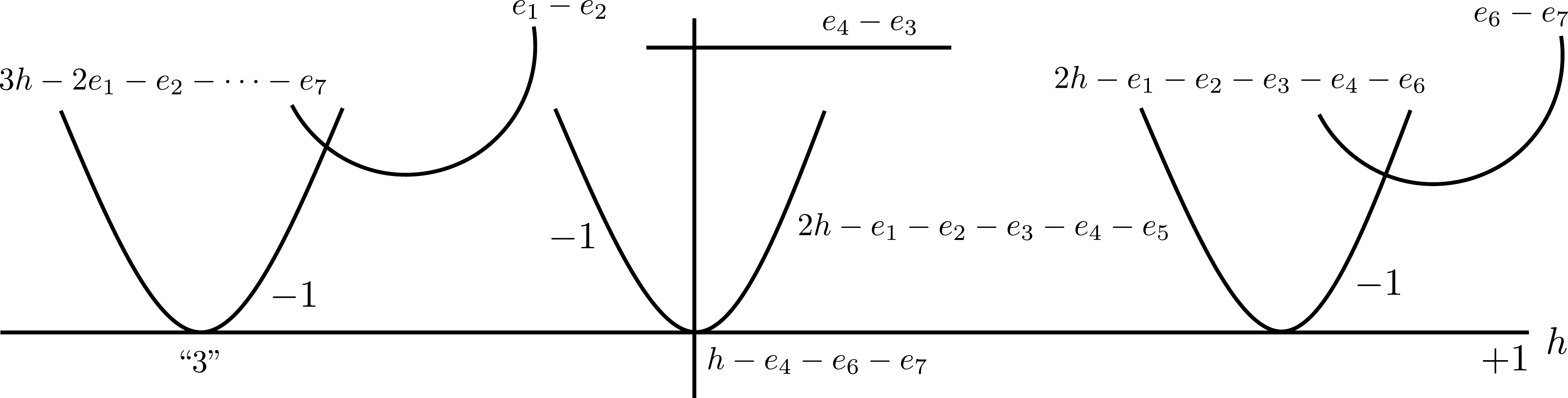}
\caption{The total transform of the curve of Proposition~\ref{p:33332} and its unique homological embedding corresponding to a sextic in $\CP$.}\label{f:3333,22,2}.
\end{figure}

The other two $(-1)$--curves $F_6$ and $F_7$ are in homology classes of the form $2h-e_{i_1}-\dots-e_{i_5}$, and since they are disjoint from $F_4$, they must contain $e_1$; in order to be disjoint from the $(-2)$--chain, they also have to contain $e_2$, $e_3$, and $e_4$.

Blowing down $e_1, \dots, e_7$, $C' \cup F_6 \cup F_7$ contracts to a $\G_4$--configuration, yielding a contradiction.
\end{proof}

\begin{rmk}\label{r:3333}\label{r:33332}
Proofs that are analogous to those of the last two propositions also obstructs the existence of symplectic rational cuspidal septics of types $[[3^{[4]},2],[2],[2]]$ and $[[3^{[4]}],[2],[2],[2]]$, which are also obstructed by Levine--Tristram signatures~\ref{p:LT7}.
\end{rmk}

%
%
%

\subsection{The proof of Theorem~\ref{t:obstruction}} \label{ss:collectobstructions}

\begin{proof}[Proof of Theorem~\ref{t:obstruction}]
The Heegaard Floer obstruction (Proposition~\ref{p:multisequences}) and the Riemann--Hurwitz obstruction (Proposition~\ref{p:RHobstruction}) exclude all possible configurations of singularities on a curve of degree $6$ and $7$, except for the sextics covered in Propositions~\ref{p:branched67} and~\ref{p:obstructedsextics} and the septics listed in Table~\ref{t:summary}.

For each of the curves listed in the table, we give a reference to the proposition (or propositions) in which it is obstructed, as well as the tool or the configuration used to obstructed: Arf stands for the Arf invariant, LT for Levine--Tristram signature, BC for branched cover, and the others correspond to a configuration coming from a birational transformation.
\end{proof}

\begin{longtable}{llll}
Singularities (MS)			&	Singularities (Top)					& Obstruction		&	Reference\\
\hline
$[4,3],[2^{[6]}]$	 		& 	$(4,7)$, $(2,13)$					& Arf, LT			& 	\ref{p:arf},~\ref{p:ZaidenbergLin}\\
$[4,3],[2^{[5]}],[2]$	 	& 	$(4,7)$, $(2,11)$, $(2,3)$			& LT		 	& 	\ref{p:ZaidenbergLin}\\
$[4,3],[2^{[4]}],[2,2]$	 	& 	$(4,7)$, $(2,9)$, $(2,5)$			& LT			& 	\ref{p:arf},~\ref{p:ZaidenbergLin}\\
$[4,3],[2^{[3]}]^2$	 	& 	$(4,7)$, $(2,7)^2$					& LT			& 	\ref{p:ZaidenbergLin}\\
$[4,2^{[6]}],[3]$	 		& 	$(2,3;2,21)$, $(3,4)$				& $\Qc$			&	\ref{p:4222222,3} \\
$[4,2^{[5]}],[3,2]$	 		& 	$(2,3;2,19)$, $(3,5)$				& Arf	 		&	\ref{p:arf}\\
$[4,2^{[5]}],[3],[2]$	 	& 	$(2,3;2,19)$, $(3,4)$, $(2,3)$		& Arf			& 	\ref{p:arf}\\
$[4,2^{[4]}],[3,2],[2]$	 	& 	$(2,3;2,17)$, $(3,5)$, $(2,3)$		& Arf, $\Qc$		& 	\ref{p:arf},~\ref{p:42222,322} \\
$[4,2^{[4]}],[3],[2,2]$	 	& 	$(2,3;2,17)$, $(3,4)$, $(2,5)$		& $\Qc$			& 	\ref{p:42222,322} \\
$[4,2^{[3]}],[3,2],[2,2]$	& 	$(2,3;2,15)$, $(3,5)$, $(2,5)$		& Arf			& 	\ref{p:arf}\\
$[4,2^{[3]}],[3],[2^{[3]}]$	& 	$(2,3;2,15)$, $(3,4)$, $(2,7)$		& Arf			& 	\ref{p:arf}\\
$[4,2,2],[3,2],[2^{[3]}]$	& 	$(2,3;2,13)$, $(3,5)$, $(2,7)$		& Arf			& 	\ref{p:arf}\\
$[4,2,2],[3],[2^{[4]}]$	 	& 	$(2,3;2,13)$, $(3,4)$, $(2,9)$		& $\Vc$			& 	\ref{p:422,3,2222} \\
$[4],[3,3,2],[2,2]$	 		& 	$(4,5)$, $(3,8)$, $(2,5)$			& Arf, LT, $\Vc$	& 	\ref{p:arf},~\ref{p:LT7},~\ref{p:V}\\
$[4],[3,3,2],[2]^2$	 	& 	$(4,5)$, $(3,8)$, $(2,3)^2$			& LT, $\Vc$		& 	\ref{p:LT7},~\ref{p:V}\\
$[4],[3,3],[2^{[3]}]$	 	& 	$(4,5)$, $(3,7)$, $(2,7)$			& Arf, $\Vc$	 	& 	\ref{p:arf},~\ref{p:V}\\
$[4],[3,3],[2,2],[2]$	 	& 	$(4,5)$, $(3,7)$, $(2,5)$, $(2,3)$	& Arf, $\Vc$		& 	\ref{p:arf},~\ref{p:V}\\
$[4],[3,2],[2^{[5]}]$	 	& 	$(4,5)$, $(3,5)$, $(2,11)$			& BC			& 	\ref{p:28-34} \\
$[4],[3,2],[2^{[4]}],[2]$	& 	$(4,5)$, $(3,5)$, $(2,9)$, $(2,3)$	& BC			& 	\ref{p:28-34} \\
$[4],[3,2],[2^{[3]}],[2,2]$	& 	$(4,5)$, $(3,5)$, $(2,7)$, $(2,5)$	& BC			& 	\ref{p:28-34} \\
$[4],[3],[2^{[6]}]$	 		& 	$(4,5)$, $(3,4)$, $(2,13)$			& Arf, BC			& 	\ref{p:arf},~\ref{p:28-34} \\
$[4],[3],[2^{[5]}],[2]$	 	& 	$(4,5)$, $(3,4)$, $(2,11)$, $(2,3)$	& BC			& 	\ref{p:28-34} \\
$[4],[3],[2^{[4]}],[2,2]$	& 	$(4,5)$, $(3,4)$, $(2,9)$, $(2,5)$	& Arf, BC			& 	\ref{p:arf},~\ref{p:28-34} \\
$[4],[3],[2^{[3]}]^2$	 	& 	$(4,5)$, $(3,4)$, $(2,7)^2$			& BC			& 	\ref{p:28-34} \\
$[3^{[4]},2],[2,2]$	 		&	$(3,14)$, $(2,5)$					& $\G_4$		& 	\ref{p:33332} \\
$[3^{[4]},2],[2]^2$	 	&	$(3,14)$, $(2,3)^2$				& Arf, LT, $\G_4$ 	& 	\ref{p:arf},~\ref{p:LT7},~\ref{r:33332}\\
$[3^{[4]}],[2,2],[2]$	 	&	$(3,13)$, $(2,5)$, $(2,3)$			& $\G_4$		& 	\ref{p:3333} \\
$[3^{[4]}],[2]^3$	 		&	$(3,13)$, $(2,3)^3$				& LT, $\G_4$ 		& 	\ref{p:LT7},~\ref{r:3333}\\
$[3^{[3]},2],[3,2],[2]$	 	&	$(3,11)$, $(3,5)$, $(2,3)$			& $\Qc$	 		& 	\ref{p:Q2.1}\\
$[3^{[3]},2],[3],[2,2]$	 	&	$(3,11)$, $(3,4)$, $(2,5)$			& $\Qc$		 	& 	\ref{p:Q2.1}\\
$[3^{[3]},2],[3],[2]^2$	 	&	$(3,11)$, $(3,4)$, $(2,3)^2$			& Arf, LT, $\Qc$	& 	\ref{p:arf},~\ref{p:LT7},~\ref{p:Q2.1}\\
$[3^{[3]},2],[2^{[5]}]$	 	&	$(3,11)$, $(2,11)$					& Arf, $\Qc$		& 	\ref{p:arf},~\ref{p:Q2.1}\\
$[3^{[3]},2],[2^{[4]}],[2]$	&	$(3,11)$, $(2,9)$, $(2,3)$			& Arf, LT, $\Qc$	& 	\ref{p:arf},~\ref{p:LT7},~\ref{p:Q2.1}\\
$[3^{[3]},2],[2^{[3]}],[2,2]$	&	$(3,11)$, $(2,7)$, $(2,5)$			& Arf, $\Qc$	 	& 	\ref{p:arf},~\ref{p:Q2.1}\\
Singularities (MS)			&	Singularities (Top)					& Obstruction		&	Reference\\
\hline
$[3^{[3]},2],[2^{[3]}],[2]^2$	&	$(3,11)$, $(2,7)$, $(2,3)^2$			& LT, $\Qc$		& 	\ref{p:LT7},~\ref{p:Q2.1}\\
$[3^{[3]},2],[2,2]^2,[2]$	&	$(3,11)$, $(2,5)^2$, $(2,3)$			& Arf, $\Qc$	 	& 	\ref{p:arf},~\ref{p:Q2.1}\\
$[3^{[3]},2],[2,2],[2]^3$	&	$(3,11)$, $(2,5)$, $(2,3)^3$			& LT, $\Qc$		& 	\ref{p:LT7},~\ref{p:Q2.1}\\
$[3^{[3]}],[3,2],[2,2]$	 	&	$(3,10)$, $(3,5)$, $(2,5)$			& $\Qc$		 	& 	\ref{p:Q1}\\
$[3^{[3]}],[3,2],[2]^2$	 	&	$(3,10)$, $(3,5)$, $(2,3)^2$		& Arf, $\Qc$		 & 	\ref{p:arf},~\ref{p:Q1}\\
$[3^{[3]}],[3],[2^{[3]}]$	&	$(3,10)$, $(3,4)$, $(2,7)$			& $\Qc$	 		& 	\ref{p:Q1}\\
$[3^{[3]}],[3],[2,2],[2]$	&	$(3,10)$, $(3,4)$, $(2,5)$, $(2,3)$	& $\Qc$		 	& 	\ref{p:Q1}\\
$[3^{[3]}],[2^{[6]}]$	 	&	$(3,10)$, $(2,13)$					& $\Qc$	 		& 	\ref{p:Q1}\\
$[3^{[3]}],[2^{[5]}],[2]$	&	$(3,10)$, $(2,11)$, $(2,3)$			& Arf, $\Qc$	 	& 	\ref{p:arf},~\ref{p:Q1}\\
$[3^{[3]}],[2^{[4]}],[2,2]$	&	$(3,10)$, $(2,9)$, $(2,5)$			& $\Qc$		 	& 	\ref{p:Q1}\\
$[3^{[3]}],[2^{[4]}],[2]^2$	&	$(3,10)$, $(2,9)$, $(2,3)^2$		& Arf, LT, $\Qc$	& 	\ref{p:arf},~\ref{p:LT7},~\ref{p:Q1}\\
$[3^{[3]}],[2^{[3]}]^2$	 	&	$(3,10)$, $(2,7)^2$				& Arf, $\Qc$	 	& 	\ref{p:arf},~\ref{p:Q1}\\
$[3^{[3]}],[2^{[3]}],[2,2],[2]$	&	$(3,10)$, $(2,7)$, $(2,5)$, $(2,3)$	& Arf, $\Qc$	& 	\ref{p:arf},~\ref{p:Q1}\\
$[3^{[3]}],[2^{[3]}],[2]^3$	&	$(3,10)$, $(2,7)$, $(2,3)^3$		& LT, $\Qc$		& 	\ref{p:LT7},~\ref{p:Q1}\\
$[3^{[3]}],[2,2]^3$	 	&	$(3,10)$, $(2,5)^3$				& $\Qc$	 		& 	\ref{p:Q1}\\
$[3^{[3]}],[2,2]^2,[2]^2$	&	$(3,10)$, $(2,5)^2$, $(2,3)^2$		& Arf, $\Qc$	 	& 	\ref{p:arf},~\ref{p:Q1}\\
$[3,3,2]^2,[2]$	 		&	$(3,8)^2$, $(2,3)$					& Arf, LT	 		& 	\ref{p:arf},~\ref{p:LT7}\\
$[3,3,2],[3,3],[2,2]$	 	&	$(3,8)$, $(3,7)$, $(2,5)$ 			& LT			& 	\ref{p:LT7t}\\
$[3,3,2],[3,3],[2]^2$	 	&	$(3,8)$, $(3,7)$, $(2,3)^2$ 			& Arf, LT	 		& 	\ref{p:arf},~\ref{p:LT7}\\
$[3,3,2],[3,2]^2$	 		&	$(3,8)$, $(3,5)^2$		 			& Arf, $\Qc$ 		& 	\ref{p:arf},~\ref{p:Q12},~\ref{r:Q3222}\\
$[3,3,2],[3,2],[2^{[4]}]$	&	$(3,8)$, $(3,5)$, $(2,9)$ 			& Arf, $\Qc$	 	& 	\ref{p:arf},~\ref{p:Q12}\\
$[3,3,2],[3,2],[2^{[3]}],[2]$	&	$(3,8)$, $(3,5)$, $(2,7)$, $(2,3)$ 	& $\Qc$		 	& 	\ref{p:Q12}\\
$[3,3,2],[3,2],[2^{[2]}]^2$	&	$(3,8)$, $(3,5)$, $(2,5)^2$			& Arf, $\Qc$	 	& 	\ref{p:arf},~\ref{p:Q12}\\
$[3,3,2],[3],[2^{[5]}]$	 	&	$(3,8)$, $(3,4)$, $(2,11)$			& Arf, $\Qc$	 	& 	\ref{p:arf},~\ref{p:Q12}\\
$[3,3,2],[3],[2^{[4]}],[2]$	&	$(3,8)$, $(3,4)$, $(2,9)$, $(2,3)$ 	& Arf, LT, $\Qc$	& 	\ref{p:arf},~\ref{p:LT7},~\ref{p:Q12}\\
$[3,3,2],[3],[2^{[3]}],[2,2]$	&	$(3,8)$, $(3,4)$, $(2,7)$, $(2,5)$ 	& Arf, $\Qc$	 	& 	\ref{p:arf},~\ref{p:Q12}\\
$[3,3]^2,[2^{[3]}]$	 	&	$(3,7)^2$, $(2,7)$					& LT	 		& 	\ref{p:LT7t}\\
$[3,3]^2,[2,2],[2]$	 	&	$(3,7)^2$, $(2,5)$, $(2,3)$			& LT	 		& 	\ref{p:LT7t}\\
$[3,3],[3,2],[2^{[5]}]$	 	&	$(3,7)$, $(3,5)$, $(2,11)$			& Arf, BC	 		& 	\ref{p:arf},~\ref{p:branched67}\\
$[3,3],[3,2],[2^{[4]}],[2]$	&	$(3,7)$, $(3,5)$, $(2,9)$, $(2,3)$	& Arf, BC	 		& 	\ref{p:arf},~\ref{p:branched67}\\
$[3,3],[3,2],[2^{[3]}],[2,2]$	&	$(3,7)$, $(3,5)$, $(2,7)$, $(2,5)$	& Arf, BC	 		& 	\ref{p:arf},~\ref{p:branched67}\\
$[3,3],[3],[2^{[6]}]$	 	&	$(3,7)$, $(3,4)$, $(2,13)$			& BC	 		& 	\ref{p:branched67}\\
$[3,3],[3],[2^{[5]}],[2]$	&	$(3,7)$, $(3,4)$, $(2,11)$, $(2,3)$	& Arf, BC	 		& 	\ref{p:arf},~\ref{p:branched67}\\
$[3,3],[3],[2^{[4]}],[2,2]$	&	$(3,7)$, $(3,4)$, $(2,9)$, $(2,5)$	& BC	 		& 	\ref{p:branched67}\\
$[3,3],[3],[2^{[3]}]^2$	 	&	$(3,7)$, $(3,4)$, $(2,7)^2$			& Arf, BC	 		& 	\ref{p:arf},~\ref{p:branched67}\\
$[3,2]^3,[2^{[3]}]$	 	&	$(3,5)^3$, $(2,7)$					& BC	 		& 	\ref{p:branched67}\\
$[3,2]^2,[3],[2^{[4]}]$	 	&	$(3,5)^2$, $(3,4)$, $(2,9)$			& Arf, BC	 		& 	\ref{p:arf},~\ref{p:branched67}\\
$[3,2],[3]^2,[2^{[5]}]$	 	&	$(3,5)$, $(3,4)^2$, $(2,11)$			& Arf, BC	 		& 	\ref{p:arf},~\ref{p:branched67}\\
$[3]^3,[2^{[6]}]$	 		&	$(3,4)^3$, $(2,13)$				& BC	 		& 	\ref{p:branched67}\\
$[3]^3,[2^{[5]}],[2]$		&	$(3,4)^3$, $(2,11)$, $(2,3)$			& Arf, BC	 		& 	\ref{p:arf},~\ref{p:branched67}\\
$[3]^3,[2^{[4]}],[2,2]$		&	$(3,4)^3$, $(2,9)$, $(2,5)$			& BC	 		& 	\ref{p:branched67}\\
$[3]^3,[2^{[3]}]^2$	 	&	$(3,4)^3$, $(2,7)^2$				& Arf, BC	 		& 	\ref{p:arf},~\ref{p:branched67}\\
\caption{The summary of obstructed septics.}\label{t:summary}
\end{longtable}

\section{Contact structures and symplectic fillings}\label{s:fillings}

By solving the isotopy problem for rational cuspidal sextics and septics, we also obtained classification results of strong symplectic fillings of certain contact structures naturally associated to the curves. We recall here the connection between these two problems. More details can be found in~\cite[Section~2.3]{GS}.

\begin{thm}[\cite{GS}]
Let $C$ be a curve with specified singularity types, genus, and Euler number $s = C\cdot C > 0$. Then there exists a compact symplectic $4$--manifold $(N,\omega_N)$ with concave boundary such that $N$ is a regular neighbourhood of $C$ and $[C]^2 = s$ in $N$. Moreover, every symplectic embedding of $C$ into a symplectic manifold $(X, \omega)$ has a concave neighbourhood that is deformation equivalent to $(N,\omega_N)$.
\end{thm}

In particular, we can associate to a curve type $C$ (that is, a list of topological singularities types, the geometric genus of the curve, and its positive Euler number) a contact $3$--manifold $(Y_C,\xi_C)$, where $Y_C = -\partial N$ is the boundary of $N$ with its orientation reversed and $\xi_C$ is induced from an inward-pointing Liouville vector field on a neighbourhood of $\partial N$.

There are natural concave caps of the contact manifold $(Y_C, \xi_C)$, each obtained from the regular neighbourhood $N$ above by blow-ups in the interior, an operation that does not affect the symplectic structure near the boundary, nor the contact structure on the boundary. In Sections~\ref{s:existencesextics} and~\ref{s:existenceseptics} we already essentially used these caps to classify symplectic fillings of $\xi_C$ that are rational homology balls---this corresponds to looking for homological embeddings that use as many exceptional divisors as we did on $N$.
Recall that a singular symplectic curve $C$ is minimally embedded in a symplectic $4$--manifold $(X, \omega)$ if $X \cap C$ contains no exceptional symplectic $(-1)$--spheres.

The analysis of homological embeddings and isotopy classes carried out in Section~\ref{s:existence} (combined with the results of~\cite[Section~6]{GS}) translates into two statements about uniqueness of fillings.

\begin{prop}
Let $C$ be a rational cuspidal curve with normal Euler number $36$, and let $\xi_C$ be the corresponding contact structure.

If $C$ is of one of type $[[3,3,3,2]]$, $[[3,3,3], [2]]$, or $[[3,3], [3,2]]$, then $\xi_C$ has two symplectic fillings up to diffeomorphism; one is a rational homology ball, corresponding to an embedding in of $C$ in $\CP$, and one has $b_2 = 1$, corresponding to an embedding in $S^2\times S^2$. Moreover, in the latter filling there is a symplectic $(-4)$--sphere, and rationally blowing down along it yields the first filling.

If $C$ has the type of any other singular sextic in Theorem~\ref{t:existence}, then $\xi_C$ has a unique minimal filling, which is a rational homology ball.\qedhere
\end{prop}

\begin{prop}
If $C$ is a rational cuspidal curve with normal Euler number $49$ and the type of a septic in Theorem~\ref{t:existence}, then the contact structure $\xi_C$ has a unique minimal filling, which is a rational homology ball.\qedhere
\end{prop}

In the first statement, a \emph{rational blow-down} is a symplectic cut-and-paste operation consisting in replacing the neighbourhood of a plumbing of symplectic spheres (in this case, a single $(-4)$--sphere) with a symplectic rational homology ball with the same contact boundary (in this case, the complement of the neighbourhood of a conic in $\CP$, which can also be described as the unit disc cotangent bundle of the real projective plane). This operation was defined by Fintushel and Stern~\cite{FintushelStern}, and, together with its generalisation, has been widely used to construct small exotic $4$--manifolds. The two statements we have just given say, in essence, that cuspidal contact structures in these degrees do not give any ``new'' operation similar to rational blow-downs.

\bibliography{septics}
\bibliographystyle{amsalpha}

\end{document}